\documentclass{amsart}
\usepackage{amssymb,graphicx,rotating,multirow,multicol,mathtools,bbm,eufrak}
\usepackage{soul,setspace}
\usepackage{tikz}
\usepackage{amsmath,amsthm,amsfonts,amscd}
\usepackage{caption,comment,marginnote}
\usepackage{float}
\usepackage{hyperref}
\usepackage{epigraph}
\usepackage[all]{xy}

\usepackage[symbol]{footmisc}

\swapnumbers
\newtheorem{theorem}{Theorem}[subsection]

\newtheorem{lemma}[theorem]{Lemma}

\newtheorem{cor}[theorem]{Corollary}
\newtheorem{proposition}[theorem]{Proposition}
\theoremstyle{remark}
\newtheorem{remark}[theorem]{Remark}
\numberwithin{equation}{subsection}
\newcommand{\K}{\ensuremath{\mathbb{K}}}
{\catcode`\@=11 \gdef\mnote#1{\marginpar{\footnotesize
 \tolerance\@M\spaceskip2.6\p@ plus10\p@ minus.9\p@\rm#1}}}
\def\Dg:{\endgraf{\bf Dg:\enspace}\ignorespaces}

\let\Bbb\mathbb
\let\Cal\mathcal
\def\sm{\smallsetminus}

\def\Map{\operatorname{Mod}}
\def\MW{\operatorname{MW}}

\newcommand{\be}{\begin{equation}}
\newcommand{\ee}{\end{equation}}

\makeatletter
\newcommand*{\encircled}[1]{\relax\ifmmode\mathpalette\@encircled@math{#1}\else\@encircled{#1}\fi}
\newcommand*{\@encircled@math}[2]{\@encircled{$\m@th#1#2$}}
\newcommand*{\@encircled}[1]{%
  \tikz[baseline,anchor=base]{\node[draw,circle,outer sep=0pt,inner sep=.2ex] {#1};}}
\makeatother

\makeatletter
\newcommand{\oset}[2]{%
  {\mathop{#2}\limits^{\vbox to -2\ex@{\kern-\tw@\ex@
   \hbox{\scriptsize #1}\vss}}}}
\makeatother

\newcommand*{\bottan}{$\overline{\raisebox{1pt}[0.95\height]{$ \bigcirc$}}$}

\newcommand*{\bott}{$\overline{\raisebox{0pt}[1.05\height]{\large $\bigcirc$}}$}

\def\upp{\setul{4pt}{}\ul{\large $\bigcirc$}}

\def\uptan{\setul{2pt}{}\ul{\large $\bigcirc$}}

\def\VV{\Large $\bigcup$\hskip-3.3mm O}
\def\bottann{\raisebox{2mm}{\scalebox{-1}{\uptan}}}
\def\upbottan{\uptan\hskip-4mm \bottann}
\def\A{\mathbf{A}}

\def\cod#1#2#3{#1\,#2\,#3}

\def\L{\Lambda}
\def\LL{\mathcal L}

\def\Z{\Bbb Z}
\def\R{\Bbb R}
\def\C{\Bbb C}

\def\Q{\Bbb Q}

\def\P{\Bbb P}
\def\T{\Bbb T}
\def\D{\Delta}

\def\DD{\operatorname{discr}}

\def\W{W^\R}
\def\Kl{\Bbb K}
\def\SSS{\Bbb S}
\def\S{\Sigma}
\def\FFF{\mathcal F^\infty}

\def\n{\mathfrak n}

\def\Nlin{\Cal N}
\def\Hl{\Cal H_{\ell}}

\def\Tri{\ell}

\def\del{\partial}
\def\ind{\operatorname{ind}}
\def\Maps{\operatorname \Map^s}

\def\Im{\operatorname{Im}}
\def\Ker{\operatorname{Ker}}

\def\Rp#1{\Bbb{RP}^{#1}}
\def\Aut{\operatorname{Aut}}

\def\rk{\operatorname{rk}}
\def\id{\operatorname{id}}

\def\Arf{\operatorname{Arf}}
\def\conj{\operatorname{conj}}

\def\bhv{\operatorname{\varUpsilon}}
\def\q{\operatorname{\widehat q}}

\newcommand*\circled[1]{\tikz[baseline=(char.base)]{
            \node[shape=circle,draw,inner sep=1pt] (char) {#1};}}
\def\Sec{\operatorname{Sec}}

\def\q{\frak q}
\def\l{\lambda}
\def\d{\delta}
\let\ll\lambda

\def\dsum{\bot\!\!\!\bot}
\let\+\dsum
\let\a=\alpha
\let\b=\beta
\def\bb{\frak b}
\let\k=\varkappa
\let\g=\gamma
\let\e=\varepsilon
\let\v=\upsilon
\let\kk=\kappa
\let\lla=\langle
\let\rra=\rangle

\let\ge\geqslant 
\let\le\leqslant 
\let\la\langle
\let\ra\rangle
\let\til\widetilde

\def\gg#1#2{\la{#2}-{#1}\ra}
\def\o{\frak o}

\makeatletter
\newcommand{\addresseshere}{%
  \enddoc@text\let\enddoc@text\relax}
\makeatother

\title[]{The real Mordell-Weil group of rational elliptic surfaces and 
real lines on del Pezzo surfaces of degree $K^2=1$}
\author[]
{S.~Finashin, V.~Kharlamov}
\address{Middle East Technical University,
Department of Mathematics\endgraf Ankara 06531 Turkey}
\address{Universit\'{e} de Strasbourg et IRMA (CNRS)\endgraf 7 rue Ren\'{e}-Descartes, 67084 Strasbourg Cedex, France}
\makeatletter

\@namedef{subjclassname@2020}{%
\textup{2020} Mathematics Subject Classification}
\makeatother


\begin{document}
\begin{abstract} 
We undertake a study of topological properties of the real Mordell-Weil group $\MW_\R$ of real  
rational elliptic surfaces $X$ which we accompany by a related study of real lines 
on $X$ and on the "subordinate" del Pezzo surfaces $Y$ of degree 1. We give an explicit description of isotopy types of real lines on $Y_\R$
and an explicit presentation of $\MW_\R$ in the mapping class group $\Map(X_\R)$.
Combining these results we establish an explicit formula for the action of $\MW_\R$ in $H_1(X_\R)$.
\end{abstract}

\maketitle

\setlength\epigraphwidth{.80\textwidth}
\epigraph{The most fascinating thing about algebra and geometry is the way\\ 
they struggle to help each other to emerge from the chaos of non–being,\\
from those dark depths of subconscious where all roots of intellectual creativity reside.
}{Yu.~I.~ Manin "Von Zahlen und Figuren"\\}

\section{Introduction}\label{intro}
By a {\it line} on a del Pezzo or 
elliptic surface 
we mean a rational embedded $(-1)$-curve (in other words,
a rational embedded curve of anti-canonical degree $1$). 
As is known, in the case of relatively minimal non-singular rational elliptic surfaces without multiple fibers the set of lines coincides with the set of sections.
We send the reader to consult Section \ref{conventions} for an account of 
other
specific terminological conventions we use.

\subsection{Prologue}
Our initial motivation came from a search for how the wall-crossing invariant 
count of real rational curves on real del Pezzo surfaces introduced
in \cite{Surgery} can be extended to other real rational surfaces. This brought us to investigate one of the first cases, the case of lines on a real rational elliptic surface, and to study directly related questions arising in this setting: (1)
how the real lines are arranged on  the real loci $X_\R$ of  real, without multiple fibers, relatively minimal,
non-singular rational elliptic surfaces $X$;
(2)  how the real Mordell-Weil group of $X$ acts on its real lines and what is its presentation in the mapping class group, $\Map(X_\R)$, of its real locus.

To respond to these questions, we perform, first, a study of real lines of {\it subordinate} real del Pezzo surfaces of degree 1 (that is the surfaces $Y$ obtained
by contracting a line on $X$), introduce a division of real lines in 5 types, enumerate the lines of each type for every  real deformation class of
del Pezzo surfaces of degree 1
and describe their position on the real locus of the surface up to isotopy. It is by combining these results with a study of a topological analog of the real Mordell-Weil group that we respond to the 
questions posed above.

\subsection{On the del Pezzo side}\label{onside}
 A standard model for a real del Pezzo surface $Y$ of degree 1 is given by a double covering of a real quadratic cone $Q\subset \P^3$ branched 
 at the vertex of $Q$ and along a transversal intersection  $C$ of $Q$ with a real cubic surface. This reduces the study of real lines
on $Y$  to a study of the {\it positive tritangents}, that is, the real hyperplane sections $l$ of $Q$ tritangent to $C$ whose real part $l_\R$ is contained in 
the half $Q_\R^+$ of $Q_\R\sm C_\R$ which is the image of $Y_\R$.

The real deformation classes of sextics
$C\subset Q$ that arise as branching loci for $Y\to Q$
are listed in  Tab. \ref{tab1}
(see, {\it e.g.}, \cite[A3.6.1]{DIK}).
There, the code $\la |||\ra$ refers to $C_\R$ having three ``parallel'' connected
components {\it embracing the vertex $\v$} of $Q$.
The code $\la p | q\ra$ with $p\ge 0,  q\ge0$
means that $C_\R$ contains one component which embraces the vertex and $p+q$ components which bound disjoint discs and placed
 in $Q_\R$ so that: $q$ of them bound disc components of $Q_\R^+$ and are called {\it negative ovals}, while the other $p$ bound disc components of the opposite half of $Q_\R$ and are called {\it positive ovals}.
 The components embracing the vertex
are called {\it $J$-components.}

Our division of real lines on $Y$ in 5 types is invariant under {\it Bertini involution} (that is the deck transformation of the covering $Y\to Q$) and can be translated into a division of positive tritangents to $C$ in 5 types as follows.
For a given tritangent, 
we let $\tau$ be the number of ovals with an odd number of tangency points counted with multiplicities,
and if $1\le\tau\le3$ assign type $T_\tau$ to this tritangent.
If $\tau=0$, we distinguish two types, $T_0$ and $T_0^{*}$.
A tritangent with $\tau=0$ is of type $T_0^*$  if it has two tangency points to the same oval and these tangency points belong to the arcs of the oval separated by
the generatrix of $Q_\R$
traced through the tangency point with the  $J$-component 
(see Fig. \ref{0type-examples}); otherwise (if the tangency points with the oval are not separated by the generatrix
through the J-tangency point, or if there
are no points of tangency with the ovals)
the tritangent is classified as type $T_0$.
\begin{figure}[h!]
\caption{}\label{0type-examples}
\begin{tabular}{cc}
\hskip-3mm\includegraphics[height=1.4cm]{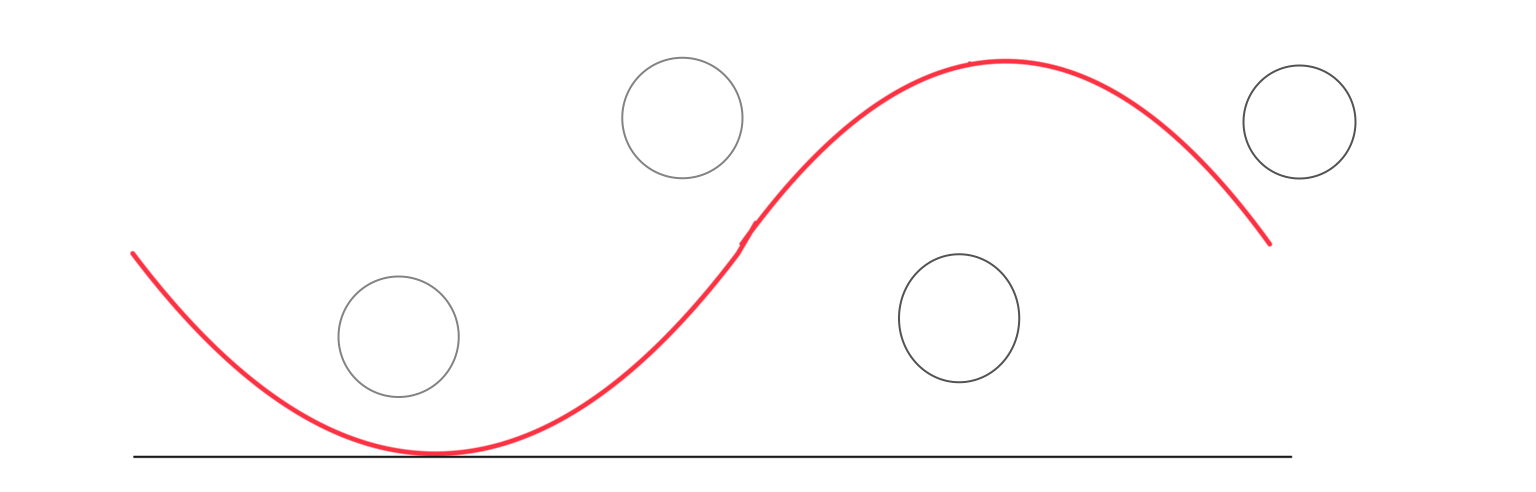}
\includegraphics[height=1.4cm]{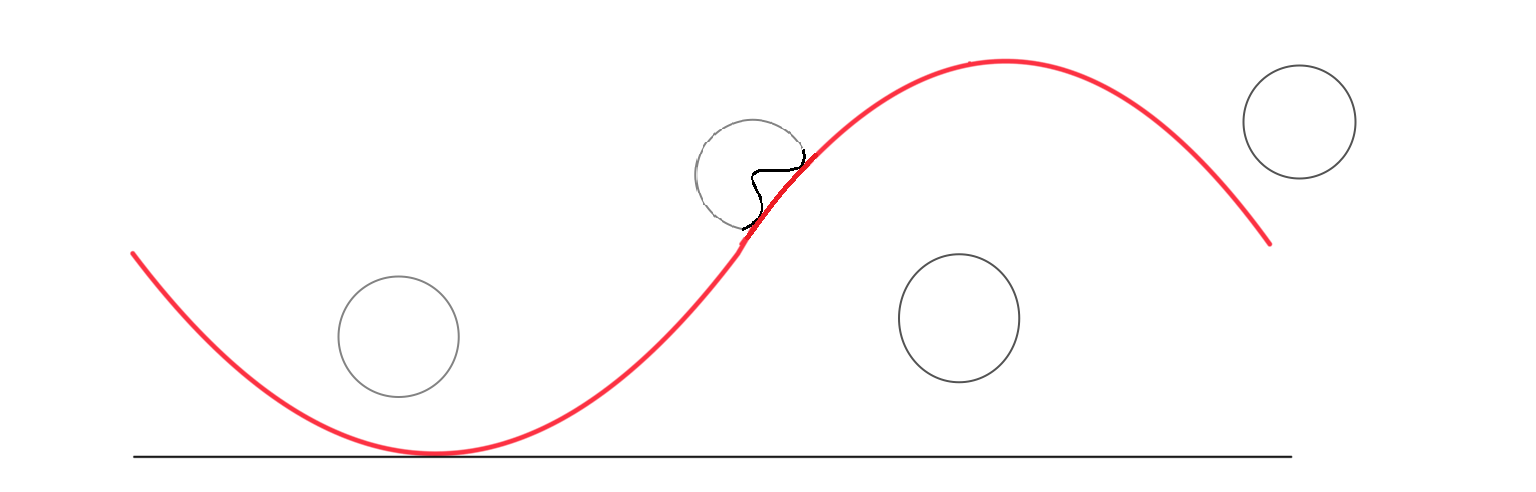}
\includegraphics[height=1.4cm]{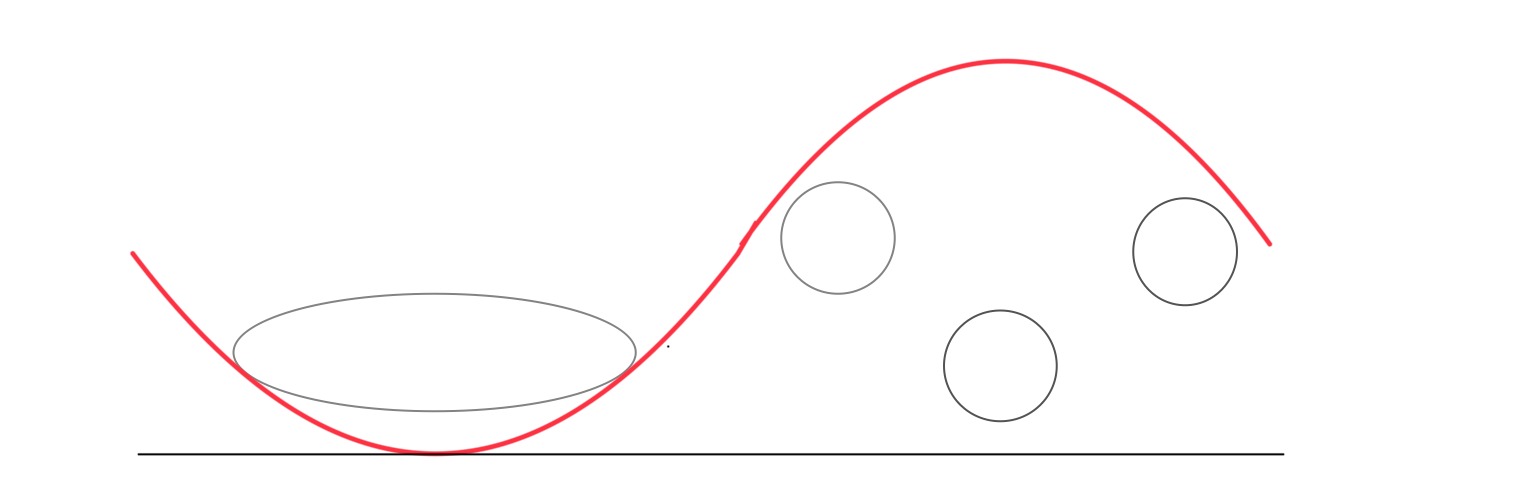}\\
\hskip0mm$T_0$\hskip40mm$T_0$\hskip40mm $T_0^*$
\end{tabular}
\caption*{\footnotesize The bottom segments depict the $J$-component. The curved lines represent a positive tritangent. The central sketch shows 
a pair of tangency points not separated by a generatrix through a J-tangency point.}
\end{figure}

\begin{theorem}\label{5types}
The number of  positive tritangents of  a given  type depends on the topological type
of  $C_\R\subset Q_\R$, and the choice of a half $Q^+_\R$,  as is indicated in Tab.\,\ref{tab1}.
\end{theorem}
One of the main tools in the proof of this theorem is a certain {\it oval-bridge decomposition} (see Section \ref{oval-bridge-classes}) 
which allows us to develop a lattice arithmetic approach not only to enumerate tritangents, but also to control their isotopic types.
In what concerns the isotopy types, formulating the results requires a special encoding and therefore we refer the reader to Section \ref{topology} for precise statements.

\begin{table}[h!]
\caption{Number of 
positive tritangents of given type}\label{tab1}
\boxed{
\begin{tabular}{c|c|c|c|c|c|c|c|c|c|c|c}
$C_\R$&$\la4|0\ra$&$\la3|0\ra$&$\la2|0\ra$&$\la1|0\ra$&$\la0|0\ra$&$\la1|1\ra$
&$\la|||\ra$&$\la0|1\ra$&$\la0|2\ra$&$\la0|3\ra$&$\la0|4\ra$
\\
\hline
$T_0$&4&4&4&4&4&3&12&3&2&1&0\\
$T_0^*$&4&3&2&1&0&1&0&0&0&0&0\\
$T_1$&32&24&16&8&0&8&0&0&0&0&0\\
$T_2$&48&24&8&0&0&0&0&0&0&0&0\\
$T_3$&32&8&0&0&0&0&0&0&0&0&0\\
\end{tabular}}
\end{table}

\subsection{On the Mordell-Weil side}
In most of our results  on elliptic surfaces, we make the following  assumption  (recall that for us a real line is the same as a real section).

{\bf Assumption $\A$.} {\it 
$X$ is a real, non-singular, relatively minimal rational elliptic surface 
whose fibers are all reduced irreducible and whose set of real lines is non-empty.} 

The real Mordell-Weil group of $X$  has a simple lattice description
and is preserved under deformations through surfaces satisfying assumption $\A$
(see Proposition \ref{lattice-MW}). However, there is no "royal road" to extract from such a description
a topological information on the action of the Mordell-Weil group on
the real loci. In our study of rational elliptic surfaces, $f: X\to \P^1$, satisfying Assumption A we overcome this difficulty by appealing systematically to subordinate del Pezzo surfaces, $Y$, 
for which we developed in Section \ref{arithmetic}
a good control on the isotopy types of real lines via an arithmetic of the lattices $\L_Y=\ker(1+\conj_*)\cap K_Y^\perp\subset H_2(Y)$ that are associated with the action in $H_2(Y)$
of the complex conjugation involution $\conj: Y\to Y$.
These lattices are determined by the topology of $Y_\R$ (see Tab. \ref{C-Y-X-correspondence}).
Furthermore, the
pullback map $H_2(Y)\to H_2(X)$ identifies $\L_Y$ with $\L_X=\ker(1+\conj_*)\cap \la K_X,L\ra^\perp\subset H_2(X)$, where $L$ is the line chosen for the contraction $X\to Y$, and we use a shortened notation $\L$ for both of them, when
it does not lead to confusion.

\begin{table}[h!]
\caption{}
\label{C-Y-X-correspondence}
\scalebox{0.9}{
\hskip-2mm
\begin{tabular}{c|c|c|c|c|c|c|c}
\hline
$C_\R$&$\la4|0\ra$&$\la3|0\ra$&$\la2|0\ra$&$\la1|0\ra$&$\la1|1\ra$&$\la|||\ra$&$\la0|q\ra$, $q\le 4$\\
$Y_\R$&$\Rp2\#4\T^2$&$\Rp2\#3\T^2$&$\Rp2\#2\T^2$&$\Rp2\#\T^2$&\scalebox{0.9}{$\Rp2\#\T^2\+\SSS^2$}&$\Rp2\+\Kl$&$\Rp2\+q\SSS^2$\\
$X_\R$&$\Kl\#4\T^2$&$\Kl\#3\T^2$&$\Kl\#2\T^2$&$\Kl\#\T^2$&\scalebox{0.9}{$\Kl\#\T^2\+\SSS^2$}&$\Kl\+\Kl$&$\Kl\+q\SSS^2$\\
\hline
\scalebox{0.85}{$\L$}&$E_8$&$E_7$&$D_6$&$D_4\oplus A_1$&$D_4$&$D_4$&$(4-q)A_1$
\\
\hline
\end{tabular}}
\caption*{\footnotesize $\T^2$, $\K$ and $\Rp2$ stand for a torus, a Klein bottle, and a projective plane (for us, $\Rp2$ is a topological surface, whereas $\P^2_\R$
belongs to the algebro-geometric category). The signs $\#$ and $\+$ stand for the connected sum and the disjoint sum respectively.}
\end{table}

Finally,
we complete this approach by giving for all types of real rational elliptic surfaces
an explicit presentation of the real Mordell-Weil group in the mapping class group of the real locus of a surface (see Sections \ref{topologicalMW} and \ref{case-by-case}).

As a first application of the above approach, we observe the following infiniteness results for the integer homology classes realized in $H_1(X_\R)$
by real lines and real vanishing cycles (by a {\it vanishing cycle} we mean a vanishing, arbitrarily oriented,
embedded circle
appearing on $X_\R$ under a real nodal degeneration).

To state these results,
we choose an orientation of $\P^1_\R$, orient the real lines $L_\R\subset X_\R$ so that 
$f\vert_{L_\R}:L_\R\to \P^1_\R$ is orientation-preserving, and denote by $\Nlin$
the number of 
classes $[L_\R]\in H_1(X_\R)$ realized by real lines. By {\it vanishing classes} we mean classes realized by the vanishing cycles.

\begin{theorem}\label{line-number}
Under the assumption $\A$, the topology of $X_\R$ determines
the number $\Nlin$ of classes realized in $H_1(X_\R)$ by real lines 
as  is
indicated in Tab.\,\ref{line-number-table}.
In particular,  $\Nlin$ is infinite if and only if $X_\R$ contains a component $\Kl\#p\T^2$ with $p\ge1$.
\end{theorem}
\vskip-1mm
\begin{table}[h!]
\caption{Number $\Nlin$ of classes in $H_1(X_\R)$ realized by real lines}
\label{line-number-table}
\boxed{
\begin{tabular}{c|c|c|c|c|c}
$X_\R$&$\Kl\#p\T^2, 0<p\le4$&$\Kl\#\T^2\+\SSS^2$&$\Kl\+\Kl$&$\Kl\+q\SSS^2, 0\le q<4$&$\Kl\+4\SSS^2$\\
\hline
$\Nlin$&$\infty$&$\infty$&$4$&$2$&$1$
\end{tabular}}
\end{table}

\begin{theorem}\label{vanishing-classes}
If 
$X$ satisfies the assumption $\A$ and 
$X_\R$ contains a component $\Kl\#p\T^2$ with $p\ge1$, then $H_1(X_\R)$ contains an infinite number of vanishing classes.
\end{theorem}

For a better presentation of the results on topological properties of the action of the  real
Mordell-Weil group, $\MW_\R(X)$, on the real locus $X_\R$
of a real elliptic surface $X$, 
we define a topological analog of $\MW_\R(X)$ as a subgroup $\Maps(X_\R)$ of the mapping class group $\Map(X_\R)$
formed by isotopy classes of fiber-preserving diffeomorphisms $X_\R\to X_\R$ acting by group-shifts in each non-singular real fiber.
One of the objects of our study is the natural homomorphism $\Phi:\MW_\R(X)\to \Maps(X_\R)$.

\begin{theorem}\label{MW-homomorphism}
Under the assumption $\A$, the homomorphism $\Phi:\MW_\R(X)\to \Maps(X_\R)$
has the image and kernel determined by $X_\R$ as is indicated in Tab.\,\ref{MW-action}.
\end{theorem}

\begin{table}[h!]
\caption{Image and kernel of $\Phi:\MW_\R(X) \to
\Maps(X_\R)$}\label{image-kernel}
\label{MW-action}
\scalebox{0.83}{
\hskip-2mm\begin{tabular}{c|c|c|c|c|c|c|c}
\hline
$X_\R$&$\Kl\#4\T^2$&$\Kl\#3\T^2$&$\Kl\#2\T^2$&$\Kl\#\T^2$&\scalebox{0.9}{$\Kl\#\T^2\+\SSS^2$}&$\Kl\+\Kl$&$\K\+q\SSS^2, q\le4$\\
\hline
\scalebox{0.9}{$\MW_\R=\L$}&$E_8$&$E_7$&$D_6$&$D_4\oplus A_1$&$D_4$&$D_4$&$(4-q)A_1$\\
$\Map^s(X_\R)$&$\Z^8\oplus\Z/2$&$\Z^6\oplus\Z/2$&$\Z^4\oplus\Z/2$&$\Z^2\oplus\Z/2$&$\Z^2\oplus\Z/2$&$\Z/2\oplus\Z/2$&$\Z/2$\\
$\Im(\Phi)$&$\Z^8$&$\Z^6\oplus\Z/2$&$\Z^4\oplus\Z/2$&$\Z^2\oplus\Z/2$&$\Z\oplus\Z/2$&$\Z/2\oplus\Z/2$&$\begin{cases}\Z/2, q<4\\ 0, \ \ \ q=4\end{cases}$\\
$\Ker(\Phi)$&$0$&$\Z$&$\Z^2$&$\Z^3$&$\Z^3$&$\Z^4$&$\Z^{4-q}$\\
\hline
\end{tabular}}
\end{table}

For an explicit presentation of the subgroup $\Maps(X_\R)\subset \Map(X_\R)$ and that of $\Phi(\MW_\R(X))$, we 
refer the reader to Sections \ref{topologicalMW} and \ref{case-by-case}. 

Theorem \ref{MW-homomorphism} implies (see Theorem \ref{lines-vs-sections}) that
for all types of $X_\R$ except $\K\#4\T^2$, $\K\#\T^2\+\SSS^2$, and $\K\+4\SSS^2$, all 
sections
of the mapping $f_\R:  X_\R\to \P^1_\R$ 
that are smooth in the sense of differential topology
are realized by real lines $L_\R\subset X_\R $ up to
ambient isotopies $X_\R\times [0,1]\to X_\R$.
If $X_\R=\K\+4\SSS^2$, then only one
isotopy class
is realized by real lines. An explicit $\Z/2$-valued obstruction for the case $\K\#4\T^2$, and an explicit $\Z$-valued obstruction for the
case $\K\#\T^2\+\SSS^2$, are given in  Theorems  \ref{section-criterion-X} and \ref{section-criterion-X1|1}, respectively.

Under the assumptions of Theorem \ref{vanishing-classes}, we fix a direct sum decomposition of $H_1(X_\R)=H_1(\K\#p\T^2)$ as follows.
First, we choose a real line 
$L_{\R}\subset X_\R=\K^2\#p\T^2\+q\SSS^2$ and a connected fiber $F_\R\subset X_\R$. If $p\ge 1$, then, in addition to the classes $[F_\R]$ (of order 2) and 
$[L_\R]$, the group $H_1(X_\R)$ contains the classes of 
positive
ovals $o_i$, $i=1,\dots, p$ (as we identify $C_\R$ in $Q_\R$ with its lifting in $X_\R$).
Furthermore, for each oval $o_i$ we pick
a real non-singular elliptic fiber intersecting it (see Fig.\, \ref{real-locus}). Such a fiber
has 2 connected components among which we denote by $a_i$ the one
intersecting $L_\R$ and by $b_i$ the other one (see details in Sec.\,\ref{decompH1}, including the
orientation conventions for the classes involved).
\begin{figure}[h!]
\caption{}\label{real-locus}
\includegraphics[height=2.2cm]{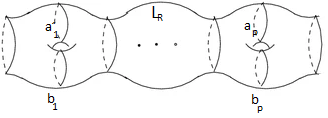}
\end{figure}
The  classes $[F_\R],b_1,o_1,\dots,b_p,o_p,[L_\R]$ form a basis giving
a direct sum decomposition 
\be\label{decomposition}
 H_1(X_\R)=
 \Z/2\oplus\Big\lbrack\bigoplus_{i=1}^p(\Z\oplus\Z)\Big\rbrack\oplus\Z.
 \ee
The class of any real line has a unique
coordinate expression of the form
$$\kk [F_\R]+\sum_{i=1}^p(m_ib_i+\k_i o_i)+[L_\R],\quad \text{ where }\quad \kk\in\Z/2, \ m_i\in\Z,\  \k_i\in\{0,1\},$$
with respect to this basis.

\begin{theorem}\label{real-action-matrix} Let $X$ satisfy
the assumption $\A$, 
$X_\R=\Kl\#p\T^2\+q\SSS^2$
and  $g\in \MW_\R$ send a real line $L$ to 
a real line $L'$.
Then the action of $g$ in $H_1(X_\R)$
is described by the following matrix with respect to the  basis $[F_\R],b_1,o_1,\dots,b_p,o_p,[L_\R]:$
\vskip1mm
\boxed{
\begin{tabular}{c|cc|c|cc|c}
$[1]$&$[\k_1]$&$[m_1]$&\dots&$[\k_p]$&$[m_p]$&$\kk$\\ 
\hline
$0$&$(-1)^{\k_1}$&$-2m_1$&\dots&$0$&$0$&$m_1$\\
$0$&$0$&$(-1)^{\k_1}$&\dots&$0$&$0$&$\k_1$\\
\hline
\dots&\dots&\dots&\dots&\dots&\dots&\dots\\
\hline
$0$&$0$&$0$&\dots&$(-1)^{\k_p}$&$-2m_p$&$m_p$\\
$0$&$0$&$0$&\dots&$0$&$(-1)^{\k_p}$&$\k_p$\\
\hline
$0$&$0$&$0$&\dots&$0$&$0$&$1$
\end{tabular}}
\newline
Here, $\kk\in\Z/2,\ m_i\in\Z,\ \k_i\in\{0,1\}$ are the coefficients in the
coordinate expression
$$
[L'_\R]-[L_\R]=\kk [F_\R]+\sum_{i=1}^p (m_ib_i+\k_io_i) \in H_1(X_\R),
$$
and the notation $[a]$ used in the first row stands for $a\,\operatorname{mod}\,2$.
\end{theorem}

\begin{theorem}\label{MW-sum-general}
Assume that $X$ satisfies assumption $\A$ and $X_\R=\Kl\#p\T^2\+q\SSS^2$. Let $L, L', L''\subset X$ 
be real lines,  $g$ be an element of $\MW_\R$ that sends $L$ to $L'$, and
 $L',L''$ have  coordinate expressions
$$
\begin{aligned}
[L'_\R]-[L_\R]=\kk_1 [F_\R]+\sum_{j=1}^p(m_{1j}b_j+\k_{1j}o_j)\\
[L''_\R] -[L_\R]=\kk_2 [F_\R]+\sum_{j=1}^p(m_{2j}b_j+\k_{2j}o_j)
\end{aligned}\quad
\kk_{ij}\in\Z/2,\ m_{ij}\in\Z,\ \k_{ij}\in\{0,1\}.
$$
Then  the class $[g(L'')_\R]\in H_1(X_\R)$ of 
the line  $g(L'')$ has a coordinate expression
$$
[g(L'')_\R]- [L_\R]=\kk [F_\R]+\sum_{j=1}^p(m_jb_j+\k_j o_i) 
$$
 where $ \kk=\kk_1+\kk_2+\sum_{j=1}^p(\k_{1j}m_{2j}+\k_{2j}m_{1j})\ {\rm{mod}\,2}$
and
$$
\begin{bmatrix}
(-1)^{\k_{1j}} &m_{1j}\\
0&(-1)^{\k_{1j}}
\end{bmatrix}
\begin{bmatrix}
(-1)^{\k_{2j}} &m_{2j}\\
0&(-1)^{\k_{2j}}
\end{bmatrix}
=\begin{bmatrix}
(-1)^{\k_j} &m_j\\
0&(-1)^{\k_j}
\end{bmatrix}.\ 
$$
\end{theorem}

\subsection{Plan of the paper} We start Section \ref{prelim} by recalling the deformation classifications of
sextic curves on a quadric cone, del Pezzo surfaces of degree 1, and real rational elliptic surfaces satisfying assumption $\A$.
We also remind
 a lattice arithmetic description of lines, 
and apply it to introduce the notion of oval- and bridge-classes and to determine their mutual intersections. In Section \ref{arithmetic}, 
we develop a certain mod 2 arithmetic of roots, and based on it, introduce our principal tool for enumerating the positive tritangents, an {\it oval-bridge decomposition}. By a systematic use of this tool we not only prove Theorem \ref{5types} but moreover supply the enumeration of positive tritangents with  information 
on their position with respect to the ovals. It is this information that we use in Section \ref{topology} for giving an explicit
description of isotopy types of positive tritangents and that of isotopy types of real lines on real rational elliptic surfaces satisfying assumption $\A$, see 
Propositions \ref{isotopies-hyperbolic-case}, \ref{prop-1-1},
Theorem \ref{answer-codes},
and Tab. \ref{pics120}. 
In Section \ref{sect5}, we introduce and evaluate the groups $\Maps(X_\R)$ and $\MW_\R(X)$. Section \ref{sect6} is devoted to the proof of Theorem
\ref{MW-homomorphism} and a lattice description of $\Ker\Phi$, see Theorem \ref{details-on-MW-homomorphism}. 
Theorems \ref{line-number}, \ref{vanishing-classes} are proved in 
Section \ref{SectProofs}, while Section \ref{realizability} is devoted to proving Theorems
\ref{section-criterion-X}, and \ref{section-criterion-X1|1}. In Section \ref{sect8},
we perform a matrix description of the action of $\Maps(X_\R)$ in $H_1(X_\R)$ and apply it to proving Theorems \ref{real-action-matrix} and \ref{MW-sum-general}.

In the concluding remarks, we discuss a few related topics.
 We start with Proposition \ref{Eichler-Siegel}, which describes the $\MW$-action in $H_2(X)$  
in the complex setting.
Being an analogue of our 
Theorem \ref{real-action-matrix}, it demonstrates, 
however, a significant difference. Namely, the $\MW$-action on $H_2(X)$ restricts to identity on $K^\perp/K$, while the action of $\MW_\R$ is not identical on $K_\R^\perp/K_\R$
(although is identity modulo 2). In Section \ref{obstruction}, we give a coordinate expression for a $\Z/2$-obstruction for realizability of classes 
in $H_1(X_\R)$ by real lines. In Section \ref{decomposable-quintics}, we give an application of our count of tritangents to
a count of real conics tangent to a pair of real lines and a real cubic. In  Section \ref{5types-theta}, we discuss a relation between 5 types of tritangents and $\theta$-characteristics. In Section \ref{nonrational}, we address a question on non-rational real elliptic surfaces.
Finally, in Section \ref{puzzle10}, we indicate a puzzling persistence  phenomenon in counting real vanishing cycles on del Pezzo surfaces of various degrees.

\subsection{Notation and conventions}\label{conventions}
For complex algebraic varieties, we denote by the same letter the variety itself and its complex point set. 
If a complex variety $Z$ is defined over $\R$, then $\conj_Z :Z\to Z$ denotes the complex conjugation
and $Z_\R$ the {\it real locus}, $Z_\R=\operatorname {Fix} \conj_Z$.
 The same convention is applied to $\conj$-invariant subsets $V\subset Z$ (complex algebraic cycles, etc.).

Recall that according to our definition of lines, a line $L$ on a relatively minimal, without multiple fibers, rational elliptic surface $f: X\to \P^1$ 
(in particular, on any surface satisfying assumption $\A$) is the same as a holomorphic section of $f$.
We prefer to use the term ''line'', because the term "section" (as a rule, "smooth section")
will be reserved for
sections of the mapping
$f_\R: X_\R\to \P^1_\R$ that are smooth in the sense of differential topology.

Blowing down a line $L\subset X$ yields a nonsingular del Pezzo surface $Y$ of degree $K_Y^2=1$, which we call {\it subordinate}
to $X$. Conversely, by blowing up the basepoint of
 the pencil 
$\vert\!-\!K_Y\vert$
we obtain a relatively minimal, without multiple fibers, rational elliptic surface.

 This establishes a canonical correspondence between pairs $(X,L)$ and del Pezzo surfaces $Y$ as above. Under this correspondence,
 the linear system $\vert\!\!-\!\!K_X\vert$  and the map $f : X\to  \P^1$
 turn into, respectively, a proper transform of the linear system $\vert\!\!-\!\!K_Y\vert$  and a proper transform of the map $f_Y : Y \dashrightarrow \P^1$.

The anti-bicanonical linear system $\vert\!\!-\!\!2K_Y\vert$
gives rise to a standard model of $Y$ as a double covering $\pi : Y\to Q$ of a quadratic cone $Q\subset \P^3$ branched at the vertex $v\in Q$ and along a transversal intersection $C$ of $Q$ with a cubic surface.
This establishes a canonical correspondence between surfaces $Y$ and pairs $(Q,C)$.
Under this correspondence, the linear system $\vert\!-\!K_Y\vert$  and the map $f_Y : Y \dashrightarrow \P^1$ turn into, respectively,  a pull-back of the system of
generatrices of $Q$ and a pull-back of the projection map $f_Q : Q \dashrightarrow \P^1$ from $v$. The deck transformation of the covering $\pi : Y\to Q$
is called {\it Bertini involution} and denoted by $\beta$.

For a compact
differentiable 2-manifold $S$, we denote by $\Map(S)$ 
the mapping class group of  diffeomorphisms of $S$ 
that fix its boundary $\del S$ pointwise
and that are orientation-preserving, if $S$ is orientable.
One of our main tasks is the study of the image and kernel of the representation of the real Mordell-Weil group $\MW_\R(X)$ of real elliptic surfaces $X$ in $\Map(X_\R)$,
and, respectively, the study of the embedding of the real lines, $L_\R\subset X_\R$, up to {\it ambient} (not necessarily fiberwise) {\it isotopies} in $X_\R$.
For that, we introduce a topological analog of $\MW_\R(X)$ as a subgroup $\Maps(X_\R)$ of the mapping class group $\Map(X_\R)$
formed by isotopy classes of fiber-preserving diffeomorphisms $X_\R\to X_\R$ acting by group-shifts in each non-singular real fiber.

\subsection{Acknowledgements} 
An essential part of this work was accomplished during our Research in Residence visits at  the Centre International de Rencontres Math\'ematiques in Luminy in 2022-2023.
It took its final shape during our stay  at the Max-Planck Institute for Mathematics in Bonn in summer 2024. We thank both institutions for their hospitality and excellent working conditions.

We would also like
to thank the anonymous referees of this
paper for a number of remarks and suggestions that helped us to improve the presentation.

\section{Preliminaries}\label{prelim}

\subsection{Drawing figures on the cone $Q_\R$}
\label{swept}
On figures, we think of the quadratic cone $Q_\R\subset \P^3_\R$ as a vertically directed cylinder in an affine chart $\R^3\subset \P^3_\R$ (placing the  vertex $\v$ of $Q$ at infinity),
pick a real generatrix $\FFF\subset Q$, and then stretch
$Q_\R\sm \FFF_\R$ on a real plane $\R^2$. In particular,
this 
allows us to make "flat" sketches of the sextic $C_\R\subset Q_\R$ ({\it cf.} Fig. \ref{0type-examples}). 
We assume that this development of $Q_\R$ agrees
with the projection map $f_Q : Q \dashrightarrow \P^1$ in such a way that 
the map $f_Q$ reads in coordinates as $(x,y)\mapsto x$.
We suppose also that $\FFF$ does not intersect the ovals, so that they can be numerated consecutively $o_1,\dots,o_r$ with respect to the positive direction of axis $x$.

\subsection{Real loci of $C$, $Y$ and $X$}\label{real-loci} To fix
a correspondence between real sextics
$C$  on a  real quadratic cone $Q$ with a fixed orientation of real  generatrices, real del Pezzo surfaces
$\pi :Y\to Q$ of degree $1$, and real elliptic surfaces $f: X\to \P^1$ with a fixed real line,
we use the following convention. 

A real elliptic surface $X$ satisfying the assumption $\A$ and equipped
with a marked real line is identified with a real del Pezzo surface $Y$ blown up at the fixed point of the anti-canonical pencil $-K_Y$. 
Next, like in Introduction, we assume that  the real structure $\conj_Y : Y\to Y$ covers the standard complex conjugation involution $\conj_Q : Q\to Q$  
and $\pi(Y_\R)=Q_\R^+$.  Accordingly, we equip the generatrices
of $Q_\R$ with an orientation that is coherent with passing at the vertex $v$
from $Q_\R^+$ to $Q_\R^-=\operatorname{Cl} (Q_\R\sm Q_\R^+)$. In the opposite direction, a real sextic $C$ and an orientation of the  generatrices
of $Q_\R$ determine uniquely the half $Q_\R^+$ of $Q_\R$.

The three classifications stated below are well known (see \cite{DIK}[A3.6.1, 17.3] for the first two, while the third one
is a straightforward consequence of the second).

\begin{theorem}\label{deform-sextics}
There exist 11 deformation classes of non-singular real sextics $C\subset Q\sm\{\v\}$ on a  real quadratic cone $Q$ with a fixed orientation of the  generatices of $Q_\R$.
Each of the deformation classes is determined by the isotopy class of the embedding $C_\R \subset Q_\R\sm \{\v\}$. 
These isotopy classes have the following 11 codes:
$$ \la p\,|\,0\ra \ \text{ with }\ \ 0\le p\le4,\qquad  \la1\,|\,1\ra,
\qquad\la\,|\,|\,|\,\ra,
\qquad\la 0\,|\, q\ra \ \text{ with }\ \ 1\le q\le4. \qquad\qed$$
\end{theorem}

\begin{theorem}
There exist 11 deformation classes of real del Pezzo surfaces $Y$ of degree 1. These classes are distinguished by the topological types
of $Y_\R$, which are listed in the second row of Tab.\,\ref{C-Y-X-correspondence}.
\qed\end{theorem}

\begin{theorem}\label{11-X-classes}
Under deformations preserving the assumption $\A$, the  real rational
elliptic surfaces $X$ satisfying the assumption $\A$ form 11  deformation classes.
Each of the deformation classes is determined by the topological type of $X_\R$.
These topological types are listed in the third row of Tab.\,\ref{C-Y-X-correspondence}.
\qed\end{theorem}

\subsection{
Lines and  positive tritangents via roots of $E_8$}\label{L-to-l}

As is known, the orthogonal complement of $K_Y$ in $H_2(Y)$ is $K_Y^\perp=E_8$. 
On the other hand, the adjunction formula implies
$L\cdot K_Y=-1$ for any line $L\subset Y$.
The following fact is also well known (see \cite[Theorem 2.1.1]{TwoKinds} and references therein).

\begin{proposition}\label{delPezzo-line-root-correspondence}
Assume that $Y$ is a real del Pezzo surface of degree 1
with the canonical divisor class $K_Y$. Then:
\begin{enumerate}
\item Every homology class $h\in H_2(Y)$ with $h^2=-1$, $h\cdot K_Y=-1$ 
is realized by a line $L\subset Y$.
This establishes a one-to-one correspondence between
the set of lines in $Y$ and the set 
$\{h\in H_2(Y)\,|\,h^2=h\cdot K_Y=-1\}$. 
\item For every root $e\in E_8$ there exists a unique line $L_e$ that realizes the homology class $-K_Y-e$.
This establishes a one-to-one correspondence between the set of lines in $Y$ and the set of roots in $E_8$.
\item If $Y$ is real, then a line $L_e$ is real if and only if $e\in\L=K_Y^\perp\cap\ker(1+\conj_*)$.
\end{enumerate}
\qed\end{proposition}

Since the Bertini involution acts on $K_Y^\perp\subset H_2(Y)$ as multiplication by $(-1)$, we have also an analogous correspondence
for positive tritangents.

\begin{proposition}\label{one-to-one}
For each root $e\in E_8$, the Bertini involution interchanges the lines
$$
L_e=-K_Y-e, \quad L_{-e}=-K_Y+e,
$$
while the projection
$\pi: Y\to Q$ maps them to a tritangent. When $Y$ and $L_{\pm e}$ are real, $\pi(L_{\pm e})$ is a 
positive tritangent.
Conversely, each tritangent (resp. 
positive tritangent) is covered by a pair of lines (resp. real lines), which are permuted by the Bertini involution.
This gives a one-to-one correspondence between the set of pairs of opposite roots $\{\pm e\}\subset E_8$ {\rm (}resp., the set of pairs of opposite roots 
$\{\pm e\}\subset \L${\rm )} and the set of tritangents {\rm (}resp., the set of
positive tritangents{\rm )}.
\qed\end{proposition}

Note that each real tritangent $\Tri$ of $C\subset Q$, like any real hyperplane section of $Q$ not passing through the vertex $\v\in Q$, 
divides $Q_\R$ into 2 half-cones. 
The half-cone which contains the germ of $Q^+_\R$ at $v\in Q$ will be denoted by $\hat \Tri$.

\subsection{Positivity of intersection for totally real $\mathbf{conj}$-anti-invariant 2-cycles}\label{positivity}
By an {\it anti-invariant 2-cycle}
in a nonsingular complex surface $Y$ with a real structure $\conj_Y: Y\to Y$ we mean 
an embedded orientable smooth 2-submanifold $Z\subset Y$ such that $Z=\conj_Y Z$ and
$\conj_Y|_Z$ is orientation-reversing. We say that $Z$ is {\it totally real}, if the tangent space $T_pZ$ is not complex
(equivalently, if $T_pY=T_pZ+ J T_pZ$ where $J$ stands for the multiplication by  $\sqrt{-1}$) for each $p\in Z$.

If $Z$ is a totally real anti-invariant 2-cycle $Z$ and $p\in Z_\R$,
then there exists a real basis $v,w\in T_pY_\R$
such that $v$ and $Jw$ is a basis of
$T_pZ$. Moreover, such vectors $v$ and $w$ are unique up to rescaling.
 The local orientation of $Y_\R$ at $p$ given by $v\wedge w$ and
 the local orientation of $Z$ given by $v\wedge Jw$ are said to be
 {\it coherent}.
 
 A smooth arc $\gamma\subset Y$ is called {\it $\conj$-anti-invariant} (resp., {\it $\conj$-invariant}) if  $\conj\vert_\gamma(\gamma)=\gamma$ and $\conj\vert_\gamma$ is reversing orientation
 (resp., if $\conj\vert_\gamma =\id$).
 
\begin{proposition}\label{cycle-cycle-intersection}
Assume that $p\in Y_\R$ is a point of transversal intersection of totally real conj-anti-invariant 2-cycles $Z_1$ and $Z_2$.
Choose some local orientation of $Y_\R$ at $p\in Y_\R$ and 
coherent with it local orientations of $Z_1$ and $Z_2$.
 Assume that there exists a smooth real algebraic curve $C\subset Y$ which contains $p$ and locally at $p$ intersects 
 $Z_1$ along a smooth $\conj$-invariant arc, and 
$Z_2$ along a smooth $\conj$-anti-invariant one.
Then  the local intersection index of these cycles, $\ind_p(Z_1,Z_2)$, is equal to 1.
\end{proposition}

\begin{proof} Let $v_1, w_1$ be a pair of vectors providing coherent orientations, $v_1\wedge w_1$ of $T_pY_\R$ and 
$v_1\wedge Jw_1$ of $T_p Z_1$.
For a similar pair $v_2,w_2$ for $Z_2$, 
transversality of $Z_2$ with $Z_1$ implies $v_2= w_1 + \lambda v_1$, $\lambda\in\R$. From the conditions imposed on $C$ we have $Jv_1\in T_pZ_2$, which together
with coherence of the orientations implies $w_2= -v_1$.
Now, the result follows from $v_1\wedge Jw_1\wedge(w_1+\lambda v_1)\wedge 
J(-v_1)=v_1\wedge Jv_1\wedge
w_1\wedge Jw_1$.
\end{proof}

\subsection{Oval and bridge classes}\label{oval-bridge-classes}
Let $C_0\subset Q$ be a 6-nodal sextic which splits into 3 real hyperplane sections.
 Select once and for all the 5 perturbations constructed as is shown on Fig. \ref{perturbations}.
 This yields non-singular real sextics, $C_\e\subset Q$, of types  
$\la p\,|\,q\ra$ with $p>0$,
 which we call {\it smart}.

\begin{figure}[h!]
\caption{Construction of smart sextics}
\label{perturbations}
\hskip17mm$\la 4\,|\,0\ra$\hskip18mm $\la 3\,|\,0\ra$\hskip17mm $\la 2\,|\,0\ra$\newline
\includegraphics[height=3.8cm]{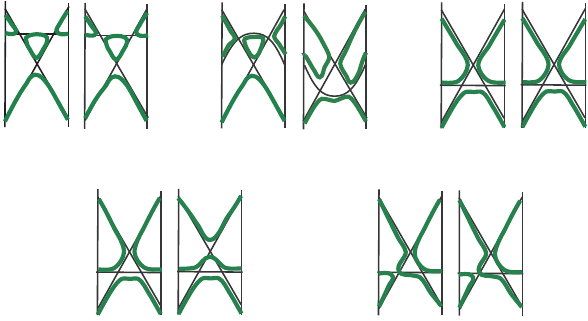}
\\
\hskip4mm $\la 1\,|\,0\ra$\hskip25mm $\la 1\,|\,1\ra$
\caption*{\footnotesize The cone $Q_\R$ is depicted as a pair of bands with 
the left-to-left and right-to-right
identification of sides in each pair.}\end{figure}

By passing to the double covering
we get a small real perturbation $Y_\e\to Q$ of a 6-nodal 
surface $Y_0$. 
Each of the 6 nodes in the case of types $\la p\,|\,0\ra$, $p=4,3,2$, 
and 5 nodes in the cases $\la 1\,|\,0\ra$ and $\la 1\,|\,1\ra$ (see Fig.\,\ref{perturbations}),
provides a $\conj$-anti-invariant
totally real vanishing cycle $B\subset Y_\e$ (well-defined up to isotopy preserving $\conj$-anti-invariance and total reality)
called a {\it bridge-cycle}. Its class in $\L\subset H_2(X)$
is denoted also by $B$ and called a {\it bridge-class}.

On the other hand, each of the $p$ positive ovals of $C_\e$ bounds a disc $D\subset Q^-_\R$, whose pull-back to $Y_\e$ is a 2-sphere which
represents a totally real $\conj$-anti-invariant cycle called an {\it oval-cycle} and denoted by $O$.
It realizes a class in $\L\subset H_2(Y_\e)$ (also denoted by $O$) called an {\it oval-class}.

Note that the real loci, $B_\R$ and $O_\R$, represent in $H_1(Y_{\e\R};\Z/2)$ the image of the bridge-class $B$ and
the oval-class $O$ under the Viro homomorphism (see \cite[Section 2.2]{TwoKinds}) 
$$\bhv : H^-_2(Y_\e)
\to H_1(Y_{\e\R};\Z/2),\quad H^-_2(Y_\e)=\ker(1+\conj_*)\subset  H_2(Y_{\e}).
$$

By construction, each bridge-class is incident
to two connected components of $C_{\e\R}$, which may coincide. When the positive ovals
of $C_\e$ are  numerated consecutively, 
the oval-class corresponding to the $i$-th oval is denoted by $O_i$,
a bridge-class  incident to the J-component and $O_i$ is denoted by $B_i$,
and a bridge-class incident to $O_i$ and $O_j$,
$j=i+1$, by $B_{ij}$.

\begin{figure}[h!]
\caption{Oval- and bridge-classes in the root-graph of $\L$}\label{oval-bridge-graphs}
\hbox{
$\L=E_8$\hskip3mm
\xymatrix@1@=15pt{
*+[Fo]{O_1}\ar@{-}[r]&B_{12}\ar@{-}[r]&*+[Fo]{O_2} \ar@{-}[r]&B_{23}\ar@{-}[r]&
*+[Fo]{O_3} \ar@{-}[r] \ar@{-}[d] &B_{34} \ar@{-}[r]&*+[Fo]{O_4}\\
&&&&{B_3} 
}
\hskip14mm
\xymatrix@1@=15pt{
{}\ar@{-}[r]&*+[Fo]{O_1}\ar@{-}[r] \ar@{-}[d]&*+[Fo]{O_2} \ar@{-}[d]\\
&*+[Fo]{O_4} \ar@{-}[r]&*+[Fo]{O_3}\ar@{-}[r]&{}
}}
\vskip3mm
\hbox{
$\L=E_7$\hskip3mm\xymatrix@1@=15pt{
B_{1} \ar@{-}[r]&*+[Fo]{O_1} \ar@{-}[r]&B_{12}\ar@{-}[r]& *+[Fo]{O_2} \ar@{-}[r] \ar@{-}[d] &B_{23} \ar@{-}[r]&*+[Fo]{O_3}\\
&&& B_2
}
\hskip25mm
\xymatrix@1@=15pt{
{}\ar@{-}[r]&*+[Fo]{O_1}\ar@{-}[r] \ar@{-}[dr]&*+[Fo]{O_2} \ar@{-}[d]\ar@{-}[r]&{}\\
&{} \ar@{-}[r]&*+[Fo]{O_3}
}}
\vskip3mm
\hbox{
$\L=D_6$\hskip3mm
\xymatrix@1@=15pt{
B_{1} \ar@{-}[r]&*+[Fo]{O_1} \ar@{-}[r]&B_{12}\ar@{-}[r]& *+[Fo]{O_2} \ar@{-}[r] \ar@{-}[d] &B_{2}\\
&&& B_2'
}
\hskip34mm
\raisebox{4mm}{\xymatrix@1@=15pt{
&&&&\\
{}\ar@{-}[r]&*+[Fo]{O_1} \ar@{-}[r] 
\ar@{-} `u[r] `[r]
\ar@{-}[d]&*+[Fo]{O_2} \ar@{-}[d]\ar@{-}[r] 
\ar@{-} `u[l]
&{}\\
&&
}}}

\vskip3mm
\hbox{
$\L=D_4+A_1$\hskip14mm
\xymatrix@1@=15pt{
B_{1} \ar@{-}[r]&*+[Fo]{O_1} \ar@{-}[r] \ar@{-}[d] &B_{1}'&B_{11}\\
& B_1''
}
\hskip24mm
\xymatrix@1@=15pt{
&&\raisebox{-2pt}{\text{---}}{}
&&\\
{}&\ar@{-}[r]&*+[Fo]{O_1}\ar@{-}[r]
\ar@{-} `l[u]
`[r] 
 \ar@{-} `r[u] `[l]
 \ar@{-}[dl]\ar@{-}[dr]&{}
 \\
&&&
}}

\vskip3mm
\hbox{
$\L=D_4$\hskip22mm
\xymatrix@1@=15pt{
B_{1} \ar@{-}[r]&*+[Fo]{O_1} \ar@{-}[r] \ar@{-}[d] &B_{1}'\\
& B_1''
}
\hskip34mm
\xymatrix@1@=15pt{
{}&\ar@{-}[r]&*+[Fo]{O_1}\ar@{-}[r]
\ar@{-}[dl]\ar@{-}[dr]&{}\\
&&&
}}
\end{figure}

Fig.\,\ref{oval-bridge-graphs} shows the incidence relations between the bridge- and oval-classes.
In the rightmost column
the oval-classes $O_i\in\L$ are depicted as circles and 
bridge-classes as line segments which either join two ovals, or join an oval with a J-component of $C_\R$ and depicted as
pendant line segments attached to ovals.
 
 In the middle
 column we present the graphs, where oval- and bridge-cycles are taken as vertices, while edges
 show the incidences between these cycles. More precisely, we indicate only a part of edges, to obtain Coxeter's graph of the lattice $\L$.
In the last three rows representing  $\L=D_6$, $D_4$, and $D_4+A_1$,
there are several pending bridges incident to $O_i$ and
we use beyond $B_i$ also notation $B_i'$, $B_i''$
(without a particular rule, just to distinguish).
If $\L=D_4+A_1$ there exists also a bridge-cycle $B_{11}$ double incident to $O_1$. 
It represents a separate vertex corresponding to $A_1$ in Coxeter's graph of $\L$
(see Proposition \ref{graph-adjacency}).

\subsection{Orientation of oval- and bridge-cycles}\label{CD-graphs} 
There exists a natural way to orient  oval- and bridge-cycles.
It is determined after fixing
a real generatrix $\FFF\subset Q$ as in Section \ref{swept}
(so that it does not intersect the ovals) 
and an orientation of the real part $Y^0_\R$ of
$Y^0=\pi^{-1}(Q\sm \FFF)$. 
The orientation of the oval- and bridge-cycles contained in $Y^0$ are chosen
coherently with the orientation of $Y_\R^0$ in the sense of Section \ref{positivity}.

\begin{proposition}\label{graph-adjacency}
Let us orient as specified above the bridge- and oval-cycles in the middle column of Fig. \,\ref{oval-bridge-graphs} that are  different from $B_{11}$
and choose any orientation for $B_{11}$.
Then their pairwise intersection indices 
 are $+1$ for cycles representing adjacent vertices and $0$ otherwise.
\end{proposition}

\begin{proof}
This positivity property is a direct consequence of Proposition \ref{cycle-cycle-intersection} when the bridge-cycle does not intersect $\pi^{-1}(\FFF_\R)$.

The only case to consider in addition is when $C_\R$ is of type $\la 1\,|\,0\ra$, since the bridge-cycle $B_{11}$
representing a single vertex on the graph is intersected by $\pi^{-1}(\FFF_\R)$.
This cycle is depicted by a loop in the rightmost column of Fig.\,\ref{oval-bridge-graphs} and contrary to all other chosen cycles, 
its intersection points with $O_1$ have intersection indices of opposite sign, as it follows from Proposition \ref{cycle-cycle-intersection}. 
Therefore, this bridge-class is orthogonal to $O_1$.
\end{proof}

\subsection{Lower and upper ovals}\label{lower-upper-section}
Assume that $C_\R$ has type $\la p\,|\,0\ra$, $1\le p\le4$, and consider
the ovals $o_i=O_{i\R}$, $i=1,\dots, p$,
with consecutive numeration.
If $p=4$ we suppose that $o_1$ and $o_3$ have bridges to the $J$-component and
call $o_1, o_3$ the {\it lower} and $o_2,o_4$ the {\it upper ovals}.

\begin{proposition}\label{lower-upper}
The distinction between lower and upper ovals in the case $p=4$ is well defined.
\end{proposition}

\begin{proof}
Since existence of a bridge is a property preserved by deformation, it is sufficient to check such an uniqueness for the sextic $C_0$
constructed in Section \ref{oval-bridge-classes}. There, we have already observed that under appropriate numeration the ovals $o_1$ and $o_3$
do have bridges to the $J$-component. Now, it remains to trace a hyperplane intersecting $o_1, o_3$, and $o_2$ (respectively, $o_4$) and to observe
that such a hyperplane separates $o_4$ (respectively, $o_2$) from the $J$-component, so that by B\'ezout no nodal degeneration connecting $o_4$ 
(respectively, $o_2$) with the $J$-component is possible.
\end{proof}

\subsection{Intersecting oval-classes by lines}\label{oval-line-intersection} Let a real line $L\subset Y$ be transversal to an oval-cycle $O\subset Y$ 
at a point $q$, or equivalently let the positive tritangent $\Tri=\pi(L)$ meet the positive oval $O_\R\subset Q_\R$ at the point $p=\pi(q)$ with simple tangency.

Our aim is to evaluate the intersection
index $\ind_q(L,O)$, where $O$ is oriented coherently with a chosen local orientation of $Y_\R$ along $O_\R$ as described in Section \ref{positivity}.
Note that there is unique up to rescaling a nonzero real vector field $w$ tangent to $Y_\R$ along $O_\R$ such that $Jw$ together with a nonzero real vector
field $v$ of vectors tangent to $O_\R$ generates the tangent spaces of $O$ along $O_\R$. In particular, due to transversality between $L$ and $O$ the vector
$w(q)$ can not be tangent to $L$. 
We can choose the field $v$ so that $v\wedge w$ defines the chosen local orientation of $Y_\R$.
Then $v\wedge Jw$ gives an orientation of
the oval-cycle $O$ coherent with that of $Y_\R$.

When drawing a piece of $Y_\R$ (as on Fig.\,\ref{lines&ovals}) we imagine it in a form of two sheets, 
permuted by Bertini involution, and choose as the front sheet the one whose orientation coincides with the right-hand (positive)
orientation.

\begin{lemma}\label{cycle-line-intersection} 
Let $L_\R$ be directed at the point $q$ by vector $av+bw$, $a,b\in\R$.
Then $\ind_q(L,O)$ is $-1$ if $ab>0$ and $1$ if $ab<0$.
\end{lemma}

\begin{proof} It follows from $v\wedge Jw\wedge (av+bw)\wedge (aJv+bJw)=- ab\,v\wedge Jv\wedge w\wedge Jw$.
\end{proof}

\begin{figure}[h!]
\caption{Detecting the intersection index of lines with oval-cycles
}
\label{lines&ovals}
\hbox{
\vtop{\hsize=50mm}{
\includegraphics[height=3.2cm]{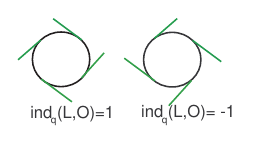}
 \hskip-60mm{\footnotesize Here,  the lines are shown  on the front sheet.}
}
\hskip-1mm
\vtop{\hsize=60mm}{\includegraphics[height=3cm]{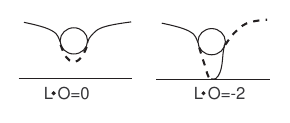}}
}
\end{figure}

The following two corollaries are straightforward consequences of Lemma \ref{cycle-line-intersection}.

\begin{cor}\label{tire-bouchon} Under the assumptions of Lemma \ref{cycle-line-intersection}
{\rm (}and the above surface-drawing convention{\rm )},
the intersection index $\ind_q(L,O)$ depends on the direction of $L_\R$
at $q$ as it is indicated on Fig. \ref{lines&ovals}.
\qed
\end{cor}

\begin{cor}\label{types2-3}
Let $L\subset Y$ be a real line that covers
a positive tritangent $\Tri$, and let $O\subset Y$ be
an oval-cycle.
Assume that $\Tri_\R$ meets the oval $O_\R\subset Q_\R$ at a pair of points, with simple tangency.
Then $L\cdot O=\pm2$ if these tangency points belong to the arcs of the oval separated by
the generatrix of $Q_\R$
traced through the tangency point with the  $J$-component, and $L\cdot O=0$ otherwise  {\rm (}see Fig. \ref{lines&ovals}{\rm )}.
\qed\end{cor}

\section{Arithmetic of real lines on del Pezzo surfaces}\label{arithmetic}

\subsection{Modulo 2 arithmetic of roots}\label{mod2arithmetic}
In this subsection we 
start with considering 
an arbitrary even negative definite lattice, which we denote by $\L$, 
and put $V=\L/2\L$. For any $e\in \L$, we denote by $[e]$ its image in $V$ under the quotient map.

\begin{lemma}\label{roots-mod2}
For any $e_1,e_2\in\L$, if $[e_1]=[e_2]$  and $e_1^2=e_2^2=-2$ then $e_2=\pm e_1$.
\end{lemma}

\begin{proof}
By triangle inequality  $\vert \frac{e_1-e_2}2\vert \le\frac{|e_1|+|e_2|}2=\sqrt2$. So, 
since $v=\frac{e_1-e_2}2$ belongs to $\L$ and the lattice $\L$ is even,
in the case of 
$v\ne0$ 
this inequality should be identity and thus, $e_2$ is collinear with $e_1$.
\end{proof}

Reducing the lattice product modulo 2 we obtain a $\Z/2$-valued bilinear form
$V\times V\to\Z/2$ and denote by $R$ its radical $R=\{v\in V\,|\,v \cdot V=0\}$.
We consider also a $\Z/2$-valued quadratic form 
$$
\q_0: V\to \Z/2, \q_0([v])=\frac{v^2}2\! \!\!\mod2
$$
associated with this bilinear form.
In $R$ we introduce another $\Z/2$-valued bilinear form  $b:R\times R\to \Z/2, b([v_1],[v_2])=\frac{v_1\cdot v_2}2\!\!\!\mod 2,$
and an associated with it $\Z/4$-valued quadratic form 
$$
\q: R\to\Z/4, \q([v])=\frac{v^2}2\! \!\!\mod4.
$$
Then, we put
$$
V_i=\q_0^{-1}(i), i\in \Z/2, \quad\text{and}\quad R_i=\q^{-1}(i), i\in \Z/4.
$$
In the same time, we let $\L^*=\{ x\in \L\otimes \Q: x\cdot L\subset\Z\} $ and consider the {\it discriminant group} $\DD(\L)=\L^*/\L$ of $\L$.

The following two lemmas are well known and straightforward from definitions.
\begin{lemma}\label{radical}
If the group
$\DD(\L)$ 
is 2-periodic, then the map
$$\DD(\L)=\L^*/\L\to V= \L/2\L,\quad x+\L\mapsto 2(x+\L)\in \L/2\L$$
is a well-defined monomorphism whose image is $R$.
 The quadratic form $\q$ in $V$ is identified with the discriminant form in $\DD(\L)$.
In particular, $\q$ is given by a matrix $[-1]$ for $\L=A_1$, $\q=[1]$ for $\L=E_7$, and 
$\begin{bmatrix}
2 & 1 \\
1 & -k 
\end{bmatrix}$
for $\L=D_{2k}$.
\qed\end{lemma}

\begin{lemma}\label{V1-R1}
For any $e\in\L$, $e^2=-2$, we have $[e]\in V_1\sm R_{1}$.
\qed\end{lemma}
 
A kind of opposite property, stated in the next proposition, does not hold for arbitrary even lattice $\L$ (for example, does not hold for $\L=n A_1$ with $n\ge5$)
but it holds for each of the lattices we need. 

 \begin{table}[h!] 
\caption{
Lattices $\L=\L(Y)$ versus the type of $C$}\label{lattices}
\boxed{
\begin{tabular}{c|c|c|c|c|c|c|c}
$C_\R$&$\la 4\,|\,0\ra $&$\la 3\,|\,0\ra $&$\la 2\,|\,0\ra $&$\la 1\,|\,0\ra $&$\la1\,|\,1\ra$&$\la\,|\,|\,|\,\ra$&
$\la0\,|\,q\ra,0\le q\le 4$\\
\hline
$\L$&$E_8$&$E_7$&$D_6$&$D_4+A_1$&$D_4$&$D_4$&$(4-q)A_1$\\
\hline
$|V|=2^{\rk\L}$&256&128&64&32&16&16&$2^{4-q}$\\
$|R|=2^{\rk R}$&1&2&4&8&4&4&$2^{4-q}$\\
\hline
$|V_1|$&120&64&32&16&12&12&$2^{3-q}$\\
$|R_1|$&0&1&2&3&0&0&$\binom{4-q}3$\\
\hline
$|V_1\sm R_1|$&120&63&30&13&12&12&$2^{3-q}-\binom{4-q}3=4-q$
\end{tabular}}
\end{table}

 \begin{proposition}\label{lift-to-root}
Any element of $V_1\sm R_1$ is realized by some root $e\in \L$ as soon as $\L$ is one of the lattices
$\L$ from Tab.\,\ref{lattices}.
\end{proposition}

 \begin{lemma}\label{downstairs} For lattices $\L=\L(Y)$ associated with real sextics $C\subset Q$, the following holds:
 \begin{enumerate}
 \item The cardinalities $|V_1|$, $|R_1|$, and
 $|V_1\sm R_1|$
 depend on the type of a sextic $C$ as it is indicated in Tab.\,\ref{lattices}.
\item In the cases $\la r\,|\,0\ra $ and $\la0\,|\,r\ra$, $0\le r\le 3$
we have $|R_1|+|R_3|=|R_0|+|R_2|=\frac12|R|=2^{3-r}$.
\end{enumerate}
 \end{lemma}
 
 \begin{proof} To count $|V_1|$ we apply the rule saying that a quadratic function on a  non-degenerate
 quadratic space (case $R=0$) takes value one $2^{g-1}(2^g+(-1)^{\Arf+1})$ times where $\Arf$ is its $\Arf$-invariant and $g$ its symplectic rank (half dimension of the underlying vector space), and that in the case $R\ne 0$ a quadratic function takes value 1 the same number of times as value 0, if the function does not vanish on $R$. For computing $\Arf$-invariants we use explicit symplectic bases.
 To check vanishing/nonvanishing of $\q_0$ on $R$ we use the congruence $\q_0\vert_R=q\!\!\!\mod 2$ and Lemma \ref{radical}.
 The computations of $R_i$ are also straightforward 
 from Lemma \ref{radical} and give, for $i=0,1,2,3$, the following values of $|R_i|$:
 \[\begin{aligned}
&|R_i|=\begin{cases} \binom{4-r}i+\binom{4-r}{i+4}, \text{ for types }\ \la r\,|\,0\ra \\
\binom{4-r}{4-i}+\binom{4-r}{8-i}, \, \text{ for types }\ \la 0\,|\,r\ra\\
\end{cases}\\
 &|R_0|=1,\ R_2=3,\ R_1=R_3=0 \hskip2mm \text{ for type } \la 1\,|\,1\ra
\hskip45mm \qedhere
\end{aligned}
\]
 \end{proof}

\begin{proof}[Proof of Proposition \ref{lift-to-root}]
Since the half of the number of roots in $\L$ is equal to the number 
of elements 
in $V_1\sm R_1$ found in Lemma \ref{downstairs} and shown in Tab.\, \ref{lattices}, the result stated
follows from Lemmas \ref{roots-mod2} and \ref{V1-R1}.
\end{proof}

Now, when we consider real sextics $C\subset Q$, we can, following Proposition \ref{lift-to-root} and Proposition \ref{one-to-one}, associate with each $v\in V_1\sm R_1$  a unique positive tritangent $\Tri_v=\pi(L_e)$ where $e$ is a (unique up to sign) root in $\L$ with $[e]=v$.

\begin{lemma}\label{intersection-criterion}
Consider a positive tritangent $\Tri_v$,
$v\in V_1\sm R_1$,
and one of the geometric vanishing classes, an oval-cycle $O$
or a bridge-cycle $B$.
\begin{enumerate}\item
$v\cdot [O]=1$ 
if and only if $\Tri_v$ has odd tangency with the oval.
\item 
$v\cdot [B] =1$ 
if and only if $\Tri_v$ separates the components of $C_\R$ incident to the bridge.
\end{enumerate}
\end{lemma}

\begin{proof} By definition, the tritangent $\Tri_v$ is covered by a line $L_e$ where $e$ is a root in $\L$ with residue $[e]=v$.
Note also that an oval of $C$ has
an odd tangency with $\Tri_v$
if and only if  in $Y_\R$ the oval has an odd intersection with $L_{e\R}$.
Since in $Y_\R$ the oval represents the image $O_\R\in H_1(Y_\R;\Z/2)$ of $O\in H_2(Y)$ by Viro homomorphism,
we have
$v\cdot [O]=L_e\cdot O\!\!\mod 2=L_{e\R}\cdot O_\R$, which gives the claim (1). 
Analogously, a separation of two components of $C_\R$,
incident to a given bridge-class $B$, by $\Tri_v$ is equivalent to an odd intersection of $L_{e\R}$ with $B_\R\in H_1(Y_\R;\Z/2)$, which gives 
the claim (2), since $L_{e\R}\cdot B_\R=e\cdot B\mod 2=v\cdot [B].$
\end{proof}

\subsection{Oval/bridge classes decomposition}\label{decomposition}
Here, we develop an approach for describing the real lines on a del Pezzo surface via reduction 
modulo 2 of the geometric root bases formed by oval- and bridge-classes that are specified, for a smart sextic $C_\e$, 
in Section \ref{oval-bridge-classes}.  The crucial role here is 
played by a direct sum decomposition $V=V^o+V^b$,
where $V^o$ is generated by residues of the vanishing oval-classes 
and $V^b$ by the residues of the bridge-classes that are shown at the rightmost column of Fig. \ref{oval-bridge-graphs} (we used here notation "$+$", since sign "$\oplus$", 
is reserved for 
orthogonal direct sums,
while $V^o$ and $V^b$ are isotropic).
Verification and principal properties of this decomposition are discussed in the next proposition.

\begin{proposition}\label{natural-decomposition}
Assume that $C_\e$ is a smart real sextic of type $\la p|q\ra$ where either $p=q=1$ or $q=0, 1\le p\le 4$.
Then:
\begin{enumerate}
\item $V=V^o+V^b$ is a direct sum decomposition.
\item $\dim V^o=p$ and the residues $[O_1],\dots,[O_p]\in V$ of the oval-classes form a basis of $V^o$.
\item $\dim V^b=4-q$ and the residues of
the bridge-classes that take part of the Coxeter-Dynkin diagrams in Fig. \ref{oval-bridge-graphs}
form a basis of $V^b$.
\item The radical $R\subset V$ has dimension $4-p-q$ and is contained in $V^b$.
\item Subspaces $V^o$ and $V^b$ are isotropic with respect to the $\Z/2$-pairing in $V$ inherited from $\L$,
and the induced pairing $V^o\times (V^{b}/R)$ is non-degenerate.
\end{enumerate}
\end{proposition}
\begin{proof} The classes involved
in the Coxeter-Dynkin diagrams in Fig. \ref{oval-bridge-graphs}
form a basis of $\L$. This implies claims (1), (2), and (3). 
Since oval- and bridge-classes alternates in the Coxeter-Dynking graphs, the subspaces $V^o$ and $V^b$ are isotropic.
To prove claims (4) and (5), it is sufficient to notice that for any collection of oval-classes there exists a bridge-class in
the Coxeter-Dynkin diagrams which has odd number of incidences with the chosen collection of oval-classes. 
\end{proof}

According to Proposition \ref{natural-decomposition}, the spaces  $V^o$ and $V^b$ have specific bases
formed respectively
by the residues $\o_i$ of oval-classes $[O_i]$, $i=1,\dots,p$ and by the residues of $4-q$ bridge-classes which are indicated on Fig.\,\ref{oval-bridge-graphs}
and which we will denote $\bb_1,\dots, \bb_{4-q}$ (with random enumeration).
These bases of $V^o$ and $V^b$ will be called {\it geometric bases}.

Any vector $v\in V$ is decomposed as $v=v^o+v^b$ in accord with the direct sum decomposition  $V=V^o+ V^b$.
By o-length $|v|_o$ and b-length $|v|_b$ of $v$
we mean the number of non-zero coordinates 
of $v^o$ and $v^b$ in the geometric bases $\o_1,\dots, \o_p, \bb_1,\dots, \bb_{4-q}$ fixed above.

\begin{lemma}\label{q-calculation}
For any $v\in V$, $\q_0(v)=|v|_o+|v|_b+v^o\cdot v^b\mod2$.
\end{lemma}

\begin{proof}
Since  $\q_0$ takes value $1$ on each of the basic elements $\o_1,\dots, \o_p, \bb_1,\dots, \bb_{4-q}$,
the relations $\q_0(v^o)=|v|_o$ and $\q_0(v^b)=|v|_b$
follow from the linearity of the restrictions $\q_0|_{V^o}$ and $\q_0|_{V^b}$ ({\it cf.} Proposition \ref{natural-decomposition}(5)).
Applying quadraticity of $\q_0$ to $v=v^o+v^b$, we obtain the required relation.
\end{proof}

\subsection{Internal and tangent ovals with respect to a tritangent}\label{Sin-Stan}
With each positive tritangent $\Tri$
we associate two index sets $S_{in}, S_{tan}\subset\{1,\dots,p\}$.
Namely, $i\in S_{in}$ if and only if the oval 
with number $i$ is contained in 
$\hat \Tri$ (defined in Subsection \ref{L-to-l}), 
and $i\in S_{tan}$ if and only
if $\Tri$ has odd tangency with this oval.

Consider also a {\it boundary homomorphism} $\delta: V^b\to V^o$ 
that sends a basic class $\bb_j\in V^b$ to the sum of the residues $\o_i=[O_i]$
of those oval-classes $O_i\in \L\subset H_2(Y)$ that are incident to the bridge underlying the class $\bb_j$. More precisely, we put, by definition, 
\be\label{delta-f}
\delta(\bb_j)=\sum_{i=1,\dots,p} (\o_i\cdot \bb_j)\o_i.
\ee
Note, that 
\be\label{ker-delta}
\ker\delta=R
\ee
as it follows immediately from Proposition \ref{natural-decomposition}.

\begin{proposition}\label{oval-bridge-components}
Assume that $C_\e$ is a smart real sextic of type $\la p|0\ra$ with $p\ge 1$ and $l_v, v\in V_1\sm R_1$, is a positive tritangent
with the associated index sets $S_{in}, S_{tan}\subset\{1,\dots,p\}$. Then:

\begin{enumerate}
\item
$v^o=\sum_{i\in S_{in}}\o_i$ and in particular 
$|v|_o=|S_{in}|$.
\item
$\delta v^b=\sum_{i\in S_{tan}}\o_i$ and in particular
$\o_i$-class in $V$ is a summand in $\delta v^b$ if and only if
 $\o_i$ has odd tangency with $\Tri_v$.
\item
$v^o\cdot v^b=|S_{in}\cap S_{tan}|\mod2$.
\end{enumerate}
\end{proposition}

\begin{proof} Due to non-degeneracy of the pairing $V^o\times (V^b/R)\to\Z/2$ (see Proposition \ref{natural-decomposition}),
the component $v^o$ of $v$ is determined by the intersection indices $v^o\cdot\bb_j= v\cdot\bb_j$ with 
$1\le j\le 4$.
On the other hand, by Lemma \ref{intersection-criterion}, $v\cdot\bb_j\in \Z/2$ does not vanish if and only if $j$-th
bridge-class 
is incident to one and only one oval lying
in $\hat\Tri_{v\R}$. Since the same non-vanishing property holds for  $(\sum_{i\in S_{in}}\o_i)\cdot \bb_j$, we obtain the claim (1). 

Due to (\ref{delta-f}) and linearity of $\delta$, we have $\delta v^b=\sum_{i=1,\dots,p} (\o_i\cdot  v)\o_i$. Thus, to get claim (2) there remains
to notice that, due to Lemma \ref{intersection-criterion}, $o_i\cdot v=1$ if and only if $\Tri_v$ is tangent to 
$i$-th oval.

Finally, we deduce from claim (1) and Lemma \ref{intersection-criterion} that $v^o\cdot v^b=\sum_{i\in S_{in}} \o_i\cdot v$ counts the number of $i\in S_{in}$ for which
$i$-th oval
is tangent to $\Tri_v$, that is the number of elements in $S_{in}\cap S_{tan}$.
\end{proof}

\subsection{Pairs $(S_{in},S_{tan})$ for sextics of type $\la4\,|\,0\ra$} Here we use numeration of ovals fixed in Sect. 
\ref{lower-upper-section} which is distinguishing lower and upper ovals (see Prop.  \ref{lower-upper}).

\begin{lemma}\label{M-case-relation}
If a smart sexic $C=C_\e$ is of type $\la4\,|\,0\ra$, then, for any positive tritangent $\Tri_v$
$$|v|_b=|S_{tan}\cap\{1,3\}|\mod2.$$
\end{lemma}

\begin{proof}
Each bridge is incident either to $o_1$ or to $o_3$ (but not to both). Therefore, 
$|v|_b$ has the same parity as 
$v\cdot \o_1 + v\cdot \o_3$, which, due to Lemme \ref{intersection-criterion}, has the same parity as
the total number
of tangencies of $\Tri_v$ with $o_1$ and $o_3$.
\end{proof}

\begin{lemma}\label{arithm}
For any sextic $C\subset Q$ of type $\la4\,|\,0\ra$ and any positive tritangent $\Tri$,
$$|S_{in}\sm S_{tan}|+|S_{tan}\cap\{1,3\}|\qquad \text{ is odd}.$$

Conversely, for any pair of sets $S_1,S_2 \subset\{1,2,3,4\}$ with odd sum
$|S_{1}\sm S_{2}|+|S_{2}\cap\{1,3\}|$,
there exists a positive tritangent $\Tri$ for which $S_1=S_{in}$ and $S_2=S_{tan}$.
\end{lemma}

\begin{proof}
The number and type of positive tritangents are preserved under deformation of $C$. So, 
it is enough to prove the statement for
a smart sextic $C=C_\e$ of type $\la4\,|\,0\ra$.

For any tritangent $\Tri_v$, $v\in V_1\sm R_1$, we have $\q_0(v)=1$, while
Proposition \ref{oval-bridge-components} with Lemmas \ref{M-case-relation} and \ref{q-calculation}
imply that 
\be\label{q0-formula}
\begin{aligned}
\q_0(v)=&\q_0(v^o)+\q_0(v^b)+v^o\cdot v^b=|S_{in}|+|S_{tan}\cap\{1,3\}|+|S_{in}\cap S_{tan}|=\\
&|S_{in}\sm S_{tan}|+|S_{tan}\cap\{1,3\}|\mod2.
\end{aligned}
\ee

To prove the converse statement, we put $v=v^o+v^b$, $v^o=\sum_{i\in S_1}\o_i$
and $v^b=\delta^{-1}(\sum_{i\in S_{2}}\o_i)$, where the inverse map $\delta^{-1}$ is well defined, since 
in the case of type $\la4\,|\,0\ra$ the homomorphism $\delta:V^b\to V^o$ is an isomorphism,
as it follows from $\ker\delta=R$ (see (\ref{ker-delta})) and $R=0$, $\dim V^b=\dim V^o$ (see Proposition \ref{natural-decomposition}).
With such a choice, we have $q(v^0)=|S_{1}|\!\!\mod 2$, while due to $\sum_{i\in S_{2}}\o_i=\delta(v^\delta)= \sum_i(\o_i\cdot v^b)\o_i$
(see (\ref{delta-f})) we get
$\q(v^b)=(v^b, \o_1+\o_3)=|S_{2}\cap\{1,3\}|$ and $v^o\cdot v^b=\sum_{i\in S_1}\cdot v^b= |S_{1}\cap S_{2}|$. Therefore,
$\q(v)= |S_{1}|+|S_{2}\cap\{1,3\}|+|S_{1}\cap S_{2}|=|S_{1}\sm S_{2}|+|S_{2}\cap\{1,3\}|=1$. Propositions \ref{lift-to-root} and \ref{one-to-one},
now, imply existence of a tritangent $\Tri_v$.  Proposition  \ref{oval-bridge-components} shows finally that 
$S_{in}=S_1$ and $S_{tan}=S_2$ for this $\Tri_v$.
\end{proof}

\begin{proposition}\label{15-8}
Assume that 
$C$ is a sextic of type $\la4\,|\,0\ra$. Then:
\begin{enumerate}
\item A subset $S\subset\{1,2,3,4\}$ can be realized as $S_{tan}$
of a positive tritangent $\Tri$ if and only if $S\ne\{1,2,3,4\}$.
\item For each of the 15  subsets 
$S_2\varsubsetneq\{1,2,3,4\}$ 
there  are precisely $8$ subsets 
$S_1\subset\{1,2,3,4\}$ for which
there exists a positive tritangent $\Tri$ with $S_{in}=S_1, S_{tan}=S_2$,
and this positive tritangent is uniquely determined by  $(S_1,S_2)$.
\end{enumerate}
\end{proposition}

\begin{proof}
It is trivial to observe that for $S_2\subset\{1,2,3,4\}$
there exists $S_1\subset\{1,2,3,4\}$ with odd
$|S_{1}\sm S_{2}|+|S_{2}\cap\{1,3\}|$ if and only if
$S_2\ne\{1,2,3,4\}$, which together with
Lemma \ref{M-case-relation} implies (1).
It is also trivial that for $S_2\varsubsetneq\{1,2,3,4\}$ 
there exists precisely 8 such subsets $S_1$.
Totally it gives $15\times 8=120$
pairs $(S_1,S_2)$ matching the 
known number of
tritangents (see, e.g., 
\cite[Corollary 5.3]{Russo}). 
Thus, each pair $(S_1,S_2)$ may be represented by a unique
positive tritangent. 
\end{proof}

The following Corollary interprets
the results of Lemma \ref{arithm} and Proposition \ref{15-8} in terms of types $T_i$, $T_0^*$.

\begin{cor}\label{Pairs-Mtype} 
For every
sextic $C$ of type $\la4\,|\,0\ra$,
a pair $(S_1,S_2)$ of subsets $S_1, S_2\subset\{1,\dots,4\}$ 
is realized as
 $(S_{in},S_{tan})$ 
 of some positive tritangent 
 if and only if $|S_1|\le3$ and $S_2\sm S_1$
satisfies the  criteria for $|S_{in}\sm S_{tan}|$ pointed in 
the table below.
\qed
\end{cor}
\hskip-5mm\scalebox{0.82}{
\boxed{
\begin{tabular}{c|c|c|c|c|c}
&$T_0$&$T_0^*$&$T_1$&$T_2$&$T_3$\\
\hline
$|S_{in}\sm S_{tan}|=$
&$3$&$1$&
$|S_{tan}\cap\{1,3\}|+1\ \rm{mod}\,2$&
$|S_{tan}\cap\{1,3\}|\ \rm{mod}\,2$&
$\begin{cases}0,&\!\text{if } \{2,4\}\subset S_{tan} \\ 
1,&\!\text{if }   \{1,3\}\subset S_{tan} 
\end{cases}$\\
\end{tabular}}}

\subsection{Pairs $(S_{in},S_{tan})$ for sextics of type $\la p\,|\,0\ra$ with $p\le 3$}

\begin{proposition}\label{tan-in-p}
Assume that $C$ is a sextic of type $\la p\,|\,0\ra$, $0\le p\le3$. Then:
\begin{enumerate}
\item For any pair of subsets $S_1,S_2\subset \{1,\dots,p\}$, except $S_1=S_2=\varnothing$ for $p\in\{2,3\}$,
there exists a positive tritangent with $S_1=S_{in}$ and $S_2=S_{tan}$.
\item If  $S_1=S_2=\varnothing$, then there exist precisely $2^{3-p}-(4-p)$ \rm{(}that is four for $p=0$,  one for $p=1$, and zero for $p\in\{2,3\}$\rm{)} such realizations.
 \end{enumerate}

Any other pair 
$(S_1,S_2)$  is realized by precisely $2^{3-p}$ positive tritangents.
\end{proposition}

\begin{proof}
Once more we refer to invariance of positive tritangents under deformation of $C$, pick a smart sextic $C_\e$ of type 
$\la p\,|\,0\ra$, and prove the statement for $C=C_\e$.

According to Proposition \ref{oval-bridge-components}, for every positive tritangent $\Tri_v$ with given $S_{in}=S_1, S_{tan}=S_2$ we should have
$v=v^o+v^b$ with $v^o=\sum_{i\in S_1}\o_i$ and $v^b\in\delta^{-1}(\sum_{i\in S_2}\o_i)$. Thus, the component $v^o$ is determined uniquely,
while $v^b$ varies in a given $R$-coset and thus can be chosen in $|\ker\delta|=|R|=2^{4-p}$ ways (see Proposition \ref{natural-decomposition}). In the opposite direction, according to 
Lemma \ref{V1-R1} and Proposition \ref{lift-to-root}, $v=v^o+v^b$ does correspond to a positive tritangent if and only if
$v\in V_1-R_1$, and such a tritangent is unique, if exists.

Since $\q(v)= \q(v^a)+\q(v^b) + v^a\cdot v^b = |S_{1}|+\q(v^b)+|S_{1}\cap S_{2}| =
\q(v^b)+|S_1-S_2| \mod 2 $ ({\it cf.} Proof of Proposition \ref{15-8}), to achieve
$v\in V_1$ we need to achieve $q(v^b)=1+|S_1-S_2| \mod 2$. Now, note that, 
for $\L=E_7,D_6, D_4+A_1, 4A_1$ corresponding to $p=3,2,1,0$ (see Tab.\,\ref{C-Y-X-correspondence}),
precisely a half of elements $v^b$ of $R$ has
$\q(v^b)=0$ and a half has $\q(v^b)=1$, which follows from linearity of $\q$ on $R$  and existence 
in $R$ of elements with $\q=1$.
This proves that the number of tritangents representing $(S_{in}=S_1,S_{tan}=S_2)$ is $\frac12|R|= 2^{3-p}$ as soon as
$S_1\ne\varnothing$ or $S_2\ne\varnothing$. Indeed, if $S_1\ne\varnothing$ then $v^o\ne 0$ and thus $v\notin R_1$, while if 
$S_2\ne\varnothing$ then $v^b$ belongs to a $R$-coset distinct from $R=\ker\delta$, and thus never belongs to $R_1$. 

If both $S_1$ and $S_2$ are empty, then
$v^o=0$ and $v^b\in R$.
In cases $p\in\{2,3\}$, we have $\L= D_6$ and $\L=E_7$. Since for these lattices $R\cap V_1\subset R_1$,
Proposition \ref{lift-to-root} implies that no positive tritangent exists in these cases. In cases $p=0,1$, we have $\L=4A_1$
and
$\L=D_4+A_1$, where $(R\cap V_1)\sm R_1$ is nonempty and consists of, respectively, 4 and 1 elements.
\end{proof}

\subsection{Pairs $(S_{in},S_{tan})$ for sextics of type $\la1\,|\,1\ra$}
\begin{proposition}\label{tan-in-1-1}
Assume that 
$C$ is a sextic of type $\la 1\,|\,1\ra$. 
Then, for 
a pair of subsets $S_{1},S_{2}\subset \{1\}$ there exists a positive tritangent with $S_{in}=S_1$ and $S_{tan}=S_2$
if and only if $(S_1,S_2)\ne (\varnothing,\varnothing)$.
Each of the remaining 3 pairs $(S_1,S_2)\ne (\varnothing,\varnothing)$
is realized precisely by 4 positive tritangents.
\end{proposition}

\begin{proof}
The proof is analogous to that of Proposition \ref{tan-in-p}. Here, $\L=D_4$, the radical $R$ is or dimension 2, and
$\q_0$ is identically zero on $R$, so that $V_1=\varnothing$.
 The emptyness of $V_1$ exclude the case $v^o=0, v^b\in R, \q(v)=1$ (which is equivalent to $(S_1,S_2)=(\varnothing,\varnothing)$).
In its turn, from $\dim R=2$ and $\q_0\vert_{R}=0$ it follows that in each of the cases $(S_1,S_2)\ne (\varnothing,\varnothing)$ there are precisely 4 choices
of $v^b$ in the R-coset $\delta^{-1}(\sum_{i\in S_{tan}}\o_i)$ for which $\q_0(v)=1$.
\end{proof}

\subsection{Proof of Theorem \ref{5types}}

\subsubsection{The case of $C$ of type $\la 4\,|\,0\ra$}
By definition, for the type $T_k$, $0\le k\le3$ and $T_0^*$ the cardinality $\vert S_{tan}\vert $ is $k$ and $0$ respectively.
Thus, applying Proposition \ref{15-8} we conclude that the number of tritangents is
\begin{itemize}
\item $\binom40\times8=8$ for the types $T_0$ and $T_0^*$ counted together,
\item $\binom4k\times 8$ for the type $T_k$, that is 32, 48 and 32 for 
$k=1,2,3$ respectively,
\end{itemize}
where $\binom4k$ indicates a choice of a subset $S_{tan}\subset \{1,\dots,4\}$.

To finish the proof we separate the types $T_0$ and $T_0^*$ by means of the following criterium.
\begin{proposition}\label{separate-23}
The 4 cases with $|S_{in}|=3, S_{tan}=\varnothing$ and 4 cases with  $|S_{in}|=1, S_{tan}=\varnothing$ represent
the tritangents $\Tri$ of type $T_0$ and $T_0^*$, respectively. 
\end{proposition}

\begin{proof}
According to Proposition \ref{oval-bridge-components},  a positive tritangent $l_v$ has $S_{in}=\{i\}, S_{tan}=\varnothing$, if an only if $v=\o_i$.
In such a case, it is the vanishing oval-class $O_i$ that lifts $v$ to $\L=E_8$. Thus, by Proposition \ref{one-to-one} $|L_v\cdot O_i|=|(-K\pm O_i)\cdot O_i|=
|O_i^2|=2$, and by Corollary \ref{types2-3}
 $L_v$ should have two tangency points with oval $O_i$ separated by tangency with the $J$-component.
 
 Similarly, a positive tritangent $l_v$ has $S_{in}=\{i,j,k\}, S_{tan}=\varnothing$ if and only if $v=\o_i+\o_j+\o_k$. The corresponding roots
 $e\in E_8$ are given by the following linear combinations of the basic geometric vanishing classes
 indicated in  Fig. \ref{oval-bridge-graphs}
(to point the position of the oval-classes, we encircle their multiplicities).
\newline
$\begin{matrix}
\circled{0}\,0\,\circled{1}\,2\,\circled{3}\,2\,\circled{1}\\
\phantom{00a)}2
\end{matrix}$
\hskip8mm
$\begin{matrix}
\circled{1}\,2\,\circled{3}\,4\,\circled{4}\,2\,\circled{1}\\
\phantom{00a)}2
\end{matrix}$
\hskip8mm
$\begin{matrix}
\circled{1}\,2\,\circled{2}\,2\,\circled{3}\,2\,\circled{1}\\
\phantom{00a)}2
\end{matrix}$
\hskip8mm
$\begin{matrix}
\circled{1}\,2\,\circled{3}\,4\,\circled{4}\,2\,\circled{1}\\
\phantom{00a)}2
\end{matrix}$
\newline
For each of these 4 roots $e$, the product $e\cdot O_i$ with each of the oval-classes $O_i$ vanishes.

\end{proof}

\subsubsection{The case of $C$ of type $\la p\,|\,0\ra$, $0\le p\le 3$}\label{372}
In the case of tritangents of types $T_k$, $k=1,2,3$, 
Proposition  \ref{tan-in-p} gives $8\binom{p}k$, 
tritangents, where $8$ appears as the product of $2^{3-p}$ with the number $2^p$ of subsets $S_{in}$.
\begin{table}[h!]
\caption{}\label{roots}
\boxed{
\begin{tabular}{c|c|c|c|c|c}
p&$\L$&root&root&root&root\\
\hline
3&$E_7$&
$\begin{matrix}
1\,\circled{1}\,1\,\circled{1}\,0\,\circled{0}\\
\phantom{a)}1
\end{matrix}$
&
$\begin{matrix}
0\,\circled{1}\,2\,\circled{3}\,2\,\circled{1}\\
\phantom{a)}2
\end{matrix}$
&
$\begin{matrix}
1\,\circled{2}\,3\,\circled{3}\,2\,\circled{1}\\
\phantom{a)}1
\end{matrix}$
&
$\begin{matrix}
1\,\circled{1}\,1\,\circled{2}\,2\,\circled{1}\\
\phantom{a)}1
\end{matrix}$
\\
\hline
2&$D_6$&
$\begin{matrix}
1\,\circled{1}\,1\,\circled{1}\,0\\
\phantom{aaa)}1
\end{matrix}$
&
$\begin{matrix}
1\,\circled{1}\,1\,\circled{1}\,1\\
\phantom{aaa)}0
\end{matrix}$
&
$\begin{matrix}
0\,\circled{0}\,0\,\circled{1}\,1\\
\phantom{aaa)}1
\end{matrix}$
&
$\begin{matrix}
0\,\circled{1}\,2\,\circled{2}\,1\\
\phantom{aaa)}1
\end{matrix}$
\\
\hline
1&$D_4+A_1$&
$\begin{matrix}
1\,\circled{1}\,0\hskip5mm0\\
\hskip-7mm1
\end{matrix}$
&
$\begin{matrix}
1\,\circled{1}\,1\hskip5mm0\\
\hskip-7mm0
\end{matrix}$
&
$\begin{matrix}
0\,\circled{1}\,1\hskip5mm0\\
\hskip-7mm1
\end{matrix}$
&
$\begin{matrix}
0\,\circled{0}\,0\hskip5mm 1\\
\hskip-7mm0
\end{matrix}$
\\  
\hline
0&$4A_1$&
1\hskip4mm0\hskip4mm0\hskip4mm0&
0\hskip4mm1\hskip4mm0\hskip4mm0&
0\hskip4mm0\hskip4mm1\hskip4mm0&
0\hskip4mm0\hskip4mm0\hskip4mm1
\end{tabular}}
\end{table}
If we count together tritangents of types $T_0$ and $T_0^*$,  Proposition  \ref{tan-in-p} gives
$2^{3-p}\cdot 2^p - (4-p)= 4+p$ tritangents,
among which $p$
corresponding to $|S_{in}|=1$ represent case $T_0^*$ 
as it follows from by Corollary \ref{types2-3} as above,
and the remaining $4$ tritangents represent $T_0$-case.
The corresponding $4$ roots are described in Tab.\,\ref{roots}.
In Tab.\,\ref{roots},
the 4 positive roots $e\in\L$ representing the 4 tritangents $\Tri_v$, $v=[e]$ of type $T_0$ for $C$ of type $\la p\,|\,0\ra$.
For each of these roots $e$ we have $e\cdot O_i=0$ for each (encircled) oval-root.
An oval $o_i$ lies above $\Tri_v$ if the corresponding (encircled) coefficient is odd.

\subsubsection{The case of $C$ of type $\la 1\,|\,1\ra$}\label{tritangents-to-1,1}
By Proposition \ref{tan-in-1-1} the type $T_1$
is represented by $8\times\binom11=8$ tritangents,
among which 4 correspond to $S_{in}=\varnothing$ and 4 to $S_{in}=\{1\}$,
while the types $T_0$ and $T_0^*$ together are represented by $4$ tritangents corresponding to $S_{tan}=\varnothing$ and $S_{in}=\{1\}$.
\begin{table}[h!] 
\caption{}\label{roots1-1}
\boxed{
\begin{tabular}{cccc|c|c|c}
root&$e$\ \ \  in&$\L=D_4$&&
type of $\Tri_{[e]}$&$(|S_{in}|,|S_{tan}|)$&\# roots\\
\hline
$\begin{matrix}
0\,\circled{1}\,0\\
\hskip1mm0
\end{matrix}$
&&&&
$T_0^*$&(1,0)&1\\
\hline
$\begin{matrix}
1\,\circled{1}\,1\\
\hskip1mm0
\end{matrix}$
&
$\begin{matrix}
1\,\circled{1}\,0\\
\hskip1mm1
\end{matrix}$
&
$\begin{matrix}
0\,\circled{1}\,1\\
\hskip1mm1
\end{matrix}$
&&$T_0$&(1,0)&3\\
\hline
$\begin{matrix}
1\,\circled{1}\,0\\
\hskip1mm0
\end{matrix}$
&
$\begin{matrix}
0\,\circled{1}\,1\\
\hskip1mm0
\end{matrix}$
&
$\begin{matrix}
0\,\circled{1}\,0\\
\hskip1mm1
\end{matrix}$
&
$\begin{matrix}
1\,\circled{1}\,1\\
\hskip1mm1
\end{matrix}$
&
$T_1$&(1,1)&4\\
\hline
$\begin{matrix}
1\,\circled{0}\,0\\
\hskip1mm0
\end{matrix}$
&
$\begin{matrix}
0\,\circled{0}\,1\\
\hskip1mm0
\end{matrix}$
&
$\begin{matrix}
0\,\circled{0}\,0\\
\hskip1mm1
\end{matrix}$
&
$\begin{matrix}
1\,\circled{2}\,1\\
\hskip1mm1
\end{matrix}$
&
$T_1$&(0,1)&4\\
\end{tabular}}
\end{table}
Among the latter four, only one represents type $T_0^*$, since there is only one oval above the J-component
(which follows again from Corollary \ref{types2-3} applied in a similar way).
The roots indicating the corresponding pairs $(S_{in},S_{tan})$
are shown in Tab.\,\ref{roots1-1}.

The types $T_3$ and $T_2$ are not represented by tritangents since $\binom13=\binom12=0$.

\subsubsection{The case of $C$ of type $\la |\,|\,|\ra$}
Absence of ovals implies that all tritangents are of type $T_0$.
Their number is the half of the number of roots in $\L=D_4$, that is $12$.

\subsubsection{The case of type $\la0\,|\,q\ra$, $q\ge1$}
In this case, $\L=qA_1$. Such a lattice has precisely $q$ pairs of opposite roots.
So, according to Proposition \ref{one-to-one}, in this case we have precisely $q$ positive tritangents.
Since all ovals of $C_\R$ bound disc-components of $Q_\R^+$, all these tritangents are of type $T_0$.

\section{Descriptive topology of positive tritangents to sextics and
 of real lines on del Pezzo surfaces}\label{topology}

As before, we consider a real nonsingular sextic $C\subset Q$ and
a real del Pezzo surface $Y$ obtained as the double covering $\pi : Y\to Q$ branched along $C$ and at the vertex $\v\in Q$.
Our goal here is to describe the isotopy types of positive real tritangents to $C_\R\subset Q_\R$,
which provides an isotopy classification of real lines on $Y_\R$.
Throughout this section (except for Propostition \ref{elimin-zigzag}), for simplicity of presentation, we assume that $C_\R$ has only simple tangencies with
the generatrices of $Q_\R$. This additional assumption does not change the isotopy classification of real tritangents and that of real lines. Indeed, both
the set of real tritangents and the set of real lines vary bijectively and continuously under real deformations of $C$ and $Y$, while this assumption can always be achieved through a suitable small real variation of $C$.

\subsection{Removable pairs of tangencies}\label{removable-sec}
It will be convenient to consider a more flexible, topological, version of tritangents in $Q_\R^+$ and 
lines on $Y_\R$.
Namely, by a {\it  loose section} we will mean a smoothly embedded circle $r\subset Q_\R^+$ that:
\begin{itemize}\item
meets each real 
generatrix
of $Q_\R$ transversely at one point, 
\item
has intersection $r\cap C_\R$ at one or three simple tangency points.
\end{itemize}
In its turn, by a {\it pseudo-line}  we will mean a smoothly embedded circle $R\subset Y_\R$ which
meets the real locus $F_\R$
of each real anti-canonical effective divisor $F\in \vert -K\vert$ transversely at one point (such
divisors $F$ are nothing but pullbacks of the 
generatrices
of $Q$).

\begin{lemma}\label{lifting}
For any  loose section $r\subset Q^+_\R$ its pull-back $\pi^{-1}(r)\subset Y_\R$ splits into a union $R\cup R'$ 
of  pseudo-lines  $R'=\beta(R)$. These pseudo-lines intersect each other transversally over the tangency points of $r$,
and both $\pi|_R$ and $\pi|_{R'}$ are diffeomorphisms.  
\qed\end{lemma}

In what follows we call a {\it Bertini-pair} each pair of pseudo-lines $R, R'$ like in Lemma \ref{lifting}.

The $J$-component is said to have a {\it zigzag}
over an interval $[a,b]\subset \P^1_\R$
if, first, $a,b$ are critical values
of the projection $f_Q\vert_J: J\to\P^1_\R$ 
and, second, for intermediate points $a<t<b$,
the 
preimages $f_Q\vert_J^{-1}(t)$ are  3-point subsets of $J$, as is shown on
Fig.\,\ref{zigzag-move}.
Respectively, we say that a real del Pezzo surface $Y$ contains a zigzag, if there exists a zigzag on the $J$-component of the associated sextic $C_\R\subset Q_\R$.

\begin{proposition}\label{elimin-zigzag}
Each of the 11 deformation classes of non-singular real sextics (resp., non-singular real del Pezzo surfaces) contains representatives without zigzags and non simple tangency to
the generatrices of $Q$ (resp., without zigzags and cuspidal fibers).
\end{proposition}

\begin{proof} For the deformation classes different from $\la1\,|\, 1\ra$,
 such representatives are provided by smart sextics, see Section \ref{oval-bridge-classes}.
For $\la1\,|\, 1\ra$, 
it is sufficient to pick a real quartic $B_\R\subset Q_\R$ with 2 ovals (say, an intersection of $Q_\R$ with a thin cylinder)
and to consider
a small real perturbation of $B\cup H$, where $H\subset Q$ is a real plane section with $H_\R$ separating the ovals of $B_\R$ in $Q_\R$.
\end{proof}

A {\it strong isotopy} of loose sections, $r_t$, is defined as an isotopy formed by loose sections which 
moves the tangency point set $r_t\cap C_\R$ by an isotopy on $C_\R$.
By a {\it fiberwise isotopy} of pseudo-lines, $R_t$, we mean an isotopy in $Y_\R$ formed by pseudo-lines.

Two loose sections, $r_0$ and $r_1$, are said to be {\it ambient isotopic} if there exists
a continuous family of diffeomorphisms $\phi_t : Q_\R^+\to Q_\R^+$, $0\le t\le 1$, with $\phi_0=\id$ and $\phi_1(r_0)=r_1$.
Such isotopies allow to perform {\it zigzag moves} of loose sections like the one shown on Fig.\,\ref{zigzag-move}.
The following version of Lemma \ref{lifting} for families is also straightforward.

\begin{figure}[h!]\caption{Zigzag move}\label{zigzag-move}
\includegraphics[height=2.5cm]{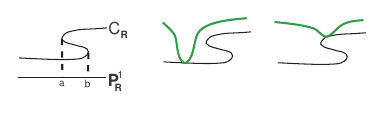}
\end{figure}

\begin{lemma}\label{lift-isotopy}
Ambient isotopies of a loose section $r$, as well as its strong isotopies,  are lifted to isotopies of each of the pseudo-lines $R, R'$ 
in the Bertini-pair
arising as pull-back of $r$. \qed
\end{lemma}

For loose sections $r$ having several tangency points 
with the same connected component $\gamma$ 
of $C_\R$, we define also a {\it simplification move}.
Namely, such a move can be performed if a pair of points $a,b\in r\cap\gamma$
is {\it removable},
which means that there exists a topological disc $D\subset Q^+_\R$ 
whose boundary is formed by two arcs, $D\cap r$ and $D\cap\gamma$, connecting $a$ and $b$
without passing through other tangency points.
Then, a {\it simplification move of $r$ guided by $D$ and supported near $\delta=D\cap r$} slightly pushes
the arc $\delta=D\cap r$ out of $D$ 
and preserves $r$ unchanged outside a small neighborhood of this arc.
As a result, we obtain
a loose section $\til r\subset Q_R^+\sm D$,
whose number of tangencies with $C_\R$ is dropped by 2.

\begin{figure}[h!]\caption{
Removable and not removable pairs of tangent points}\label{double-figure}
\includegraphics[height=0.9cm]{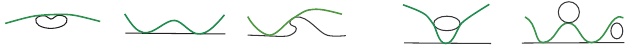}
\newline
\phantom{AAAAAAA}removable\hskip50mm not removable
\end{figure}

\begin{lemma}\label{removable} 
Assume that a positive real tritangent $\Tri$ to a real non-singular sextic $C\subset Q$ of type different from $T_0^*$
has more than one tangency point with a connected component of $C$. 
Then at least one pair of these tangency points is removable.
\end{lemma}
\begin{proof} By B\'ezout, every real 
generatrix
of $Q$ intersects $\Tri_\R$ at
a unique real point and meets each oval of $C$ in at most 2 real points. Therefore, if $a$ and $b$ are two consecutive points of tangency of $\Tri_\R$
with an oval $o$, we consider that arc $a b$ of $\Tri_\R$ which does not contain the third tangency point. If $\Tri$ is not of type $T_0^*$, the real 
generatrices
of $Q$ passing through the points of this arc trace on $o$ two arcs. One of them has $a, b$ as extremities and forms together with $ab\subset \Tri_\R$ a circle bounding 
in $Q^+$ a disc formed by intervals of the above real 
generatrices
(see the leftmost sketch of Fig.\,\ref{422-figure}). This proves the statement in the case of tangencies with an oval.

Next, assume that $\Tri$ has 3 
tangency points with the J-component: $a, b$, and $c$. Then, $\Tri_\R\cup J$ form 3 topological circles and, if neither of them bounds a disc in $Q^+$,  inside each of these circles there is an oval. Now intersecting $l$ with a plane section
$h\subset Q$ intersecting each of these
3 ovals we observe at least $6>2$ intersection points (see at the center of Fig.\,\ref{422-figure}), which is in contradiction with the B\'ezout theorem. 

Finally, assume that $a$ and $b$ are 2 tangency points of $\Tri$ with the $J$-component, and $c$ is a tangency point of $\Tri$ with an oval $o$.
Then, $\Tri_\R\cup J$ form 2 topological circles. One of them contains $c$. If the other circle does not bound a disc in $Q^+$, then inside it there is an oval, $o'$.
Now, intersecting $\Tri$ with a plane section $h\subset Q$ passing through the points $b$, $c$ and any point on  the oval $o'$,
we obtain a contradiction with the Bezout theorem
applied to $h\cap \Tri$ and $h\cap C$. Namely, to avoid the third intersection point with $\Tri$ besides $b$ and $c$, the section $h$ must
pass below the  point $a$ as is shown on the rightmost sketch of Fig.\,\ref{422-figure}.
This leads to $>2$ intersection points with the $J$-component and thus, $>6$ with $C$.
\end{proof}

\begin{figure}[h!]\caption{To the proof of Lemma \ref{removable}}\label{422-figure}
\includegraphics[height=2.7cm]{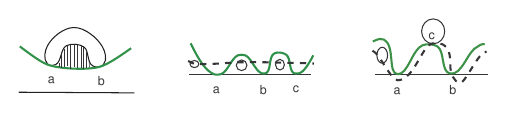}
\end{figure}

We say that a loose section (and in particular, a tritangent) is {\it simple}
 if 
either
it is tangent to each connected component of $C_\R$ not more than once
or it is of type $T_0^*$.

\begin{proposition}\label{stretching}
Every positive tritangent
is either
simple itself, or can be made simple 
tritangent
by a simplification move.

Each of the pseudo-lines in the
Bertini-pair  that covers a loose section obtained by a simplification move 
is isotopic to a line in the
Bertini-pair that covers the initial tritangent.
If for the disc $D$ that guides the simplification move, the arc $D\cap\gamma$ is a part of an oval, or a part of a $J$-component
with no zigzag in $D\cap\gamma$, the isotopy can be made fiberwise.

\end{proposition}

\begin{proof}
The first part follows directly from Lemma \ref{removable}.
For the second part, consider a loose section $r_2$ obtained by  a simplification move
of $r_0$
and note that,
due to a disk $D$ guiding
the move,
$r_2$ can be obtained by a continuous family $r_t$, $t\in[0,2]$, such that:
\begin{itemize}
\item it performs a 
an isotopy for $t\in[0,1)$
so that
the removable tangency points $a,b\in r_0\cap \gamma$
move towards each other along $\gamma$
and merge into 
a double tangency point of $r_1\cap\gamma$;
\item while for
$t\in[1,2]$, 
it performs shifting of this double tangency from $\gamma$ to obtain $r_2$.
\end{itemize}
If $\gamma$ is an oval, or 
a $J$-component with the arc $D\cap\gamma$ not containing zigzag,
then the disc $D$ is sliced in intervals by the 
generatrices
of $Q$ 
(see the leftmost sketch on Fig.\,\ref{422-figure}), and by this reason in such a case the above isotopies can be made fiberwise.

For
every $t\in [0,2]$, the pull-back $\pi^{-1}(r_t)\subset Y_\R$ splits into a Bertini-pair of pseudo-lines  $R_t$ and $R_t'$
(see Lemma \ref{lift-isotopy}).
Each of these two families of pseudo-lines forms
an isotopy 
(at moment $t=1$ these pseudo-lines are just tangent to each other, see Fig. \ref{simplification-fig}).
\end{proof}
\begin{figure}[h!]\caption{
A family $r_t$ connecting a tritangent $r_0$ with its simplification $r_2$ (upper row) and the covering isotopy of Bertini-pairs
(lower row)}\label{simplification-fig}
\includegraphics[height=2.5cm]{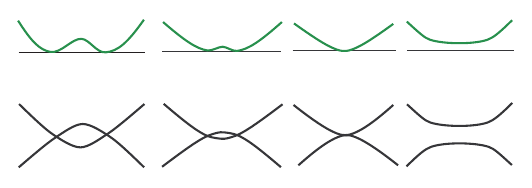}
\end{figure}

In what follows
by
a {\it simplified tritangent} we mean a tritangent itself if it is simple, or a loose section obtained from the tritangent by a simplification move.

\subsection{The simplest case: Sextics $C$ of type $\la 0\,|\,q\ra$}\label{q-0}
Absence of positive ovals implies that 
all positive tritangents in this case are of type $T_0$.
By Theorem \ref{5types}, 
their number is $4-q$, and
in particular, there are no positive tritangents if $q=4$ and
no real lines on the corresponding $Y_\R=\Rp2\+4\SSS^2$.

If $q\le 3$, each 
positive tritangent is isotopic to the $J$-component, and
each real line on the corresponding $Y_\R=\Rp2\+q\SSS^2$ is isotopic to the (unique) lift of the $J$-component to $Y_\R$, and, in particular, all 
real lines are isotopic to each other.
If the $J$-component contains no zigzag,
then the isotopies between the real lines can be performed fiberwise,
while the tritangents becomes strongly isotopic after simplification moves (see Proposition \ref{stretching}).

\subsection{Sextics $C$ of type $\la |||\ra$} 
In this case $C_\R$ has three $J$-components and
$Q_\R^+$ has two connected components: a disc containing the vertex of $Q$ and a band,
which are covered in $Y_\R$ by $\Rp2$ and $\Kl$, respectively. 
The components of $C_\R$ will be
denoted by $J_1$, $J_2$, $J_3$
so that $J_1$ bounds the disc-component of $Q_\R^+$,
while $J_2$ and $J_3$ bound the band-component
and $J_2$ lies between $J_1$ and $J_3$ on $Q_\R$.

For the same reason as in the previous case, all positive tritangents are
of type $T_0$,
each of the tritangents is isotopic either to $J_1$, or to $J_2$, or to $J_3$,
and
each real line on $X_\R=\Rp2\+\K$ is isotopic to the lift of 
a corresponding $J$-component.

\begin{lemma}\label{4-J-bridges}
There exist 4 geometric bridge-classes $B_1,\dots, B_4$
between components $J_2$ and $J_3$, and any 3 of these four classes 
together with the class $B_0=-\frac12(B_1+\dots+B_4)$
form a root 
basis of the $D_4$-lattice $\L$, wherein $B_0$
represents the central vertex of the $D_4$-graph and the 3 other chosen classes the pendant vertices.
\end{lemma}

\begin{figure}[h!]\caption{A sextic $C$ of type $\la |||\ra$ with 4 geometric bridge-classes}\label{III-fig}
\includegraphics[height=1.7cm]{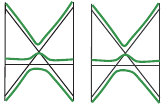}\hskip10mm
\includegraphics[height=1.8cm]{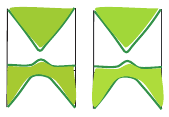}
\end{figure}

\begin{proof} For existence of 4 bridge-classes, see Fig. \ref{III-fig}, where the sextic is obtained by a small real perturbation of 3 real hyperplane sections. 
A divisibility of their sum by 2 follows from comparison of the discriminants of $D_4$ and 
$4A_1$.
\end{proof}

\begin{proposition}\label{12=3x4}  If $C\subset Q$ is of type $\la |||\ra$, then, for each of the components $J_i$, $i=1,2,3$,
there exist precisely 4 positive tritangents having odd tangency with it.
\end{proposition}

\begin{proof} By Theorem \ref{5types}, the total number of positive tritangents is 12. Among them there are 4 corresponding to the geometric bridge-classes $B_1, \dots, B_4$ and 8 to the 8 pairs of opposite roots $\frac12(\pm B_1\pm\dots\pm B_4)$. The tritangents $\pi(L_{B_i})$ ($i=1,\dots,4$) are contained in the
disc-component, while the tritangents $\pi(L_e)$ with $e=\frac12(\pm B_1\pm\dots\pm B_4)$ belong to the band-component, as it follows
from $L_{B_i}\cdot B_j= - B_i\cdot B_j=0\! \mod 2$ 
and $L_e\cdot B_j= - e\cdot B_j=\mp\frac12B_j^2=1\!\mod 2$, for every $1\le i,j\le4$ (cf. Proposition \ref{graph-adjacency}).

To conclude, we notice that according to Theorem 4.2.2 in \cite{TwoKinds} the positive tritangents tangent to $J_1$ and $J_3$ are hyperbolic, while those tangent to $J_2$ are elliptic, and that according to Theorem 1.1.2 in \cite{Combined} the number of hyperbolic tritangents minus the number of elliptic is equal to 4.
\end{proof}

\begin{proposition}\label{isotopies-hyperbolic-case}
If $C\subset Q$ is of type $\la |||\ra$, then the 12 positive tritangents split in 3 groups by 4 tritangents isotopic to the same $J$-component. For each of the 3 groups, all the 8 real lines in
the 4 covering Bertini-pairs
are fiberwise isotopic to each other, while the 4 tritangents themselves becomes strongly isotopic after simplification moves.
\end{proposition}

\begin{proof} Existence of fiberwise isotopies for lines, as well as that of strong isotopies for tritangents, follows from absence of zigzags on sextics of type $\la |||\ra$.
\end{proof}

\subsection{Sextics $C$ of type $\la 1\, |\,1\ra$} 
Proposition \ref{stretching} together with
Table \ref{roots1-1}, which lists possible combinations of $(S_{in}, S_{tan})$, and
Theorem \ref{5types},  which provides the number of positive tritangent of each type,
 can be summarized in the  following  description.

\begin{proposition}\label{prop-1-1}
For any nonsingular sextic $C$ of type $\la 1\, |\,1\ra$, up to ambient isotopy in $Q_+$ 
the simplified positive tritangents are as shown on Fig. \ref{4pics}.
\begin{figure}[h!]\caption{}\label{4pics}
\includegraphics[height=1.7cm]{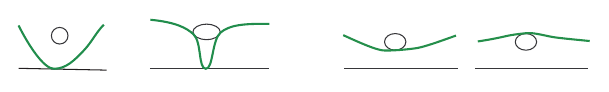}
\newline
Type $T_0$ \hskip14mm Type $T_0^*$ \hskip35mm  Type $T_1$\phantom{AAAAAAaa}
\end{figure}
The leftmost type, $T_0$, is represented by 3 distinct tritagents, the next type, $T_0^*$, by 1, and each of the remaining ones {\rm (}both $T_1${\rm)} by 4.
The
real lines covering the tritangents of the same isotopy type are isotopic. If $C$ has no zigzags,
then the isotopies between 
these lines can be performed fiberwise, while the 
tritangents themselves become strongly isotopic after simplfiication moves.
\qed\end{proposition}

\subsection{Encoding of the isotopy types}\label{encoding}
If a simplified tritangent,
$r\subset Q_\R^+$, goes below (resp., above)
a positive oval  without  tangency,
we say that $r$ {\it underpasses} (resp., {\it overpasses}) this oval, 
and use the symbol \upp \  (resp., \bott ) to encode such mutual position.
If $r$ goes below (resp. above)  an oval with one simple tangency
we use the symbol \uptan \ (resp., \bottann) and say  that $r$ is an {\it undertangent} (resp., an {\it overtangent}).
When we wish to underline that both, undertangent and overtangent, positions are realizable,
we put the {\it ambivalent symbol}\, \upbottan.
In the case of a tritangent of type $T_0^*$ 
we introduce an additional symbol
 {\VV} for the fragment with two tangencies to an oval and an intermediate tangency to the J-component (see Section \ref{onside}).

Note that the fiber
$f_Q^{-1}(t)\subset Q_\R$, $t\in \P^1_\R$,
containing a J-tangency point cannot
intersect an overpassed or overtangent oval of $C_\R$.
 It follows that $t$ belongs to the complement
$I_r^\circ\subset \P^1_\R$ of the projection of the union of such ovals 
with the set of undertangent tangency points.
We let $I_r\subset I_r^\circ$ be obtained by removing from $I_r^\circ$ the projection of the undertangent ovals
and put $J_r=f_Q^{-1}(I_r)\cap J$.

\begin{lemma}\label{moving-to-I}
Assume that $r$ is a  simple loose section with a $J$-tangency point. Then:
\begin{enumerate}\item
A strong isotopy of  $r$ does not change the sets $I_r$ and $J_r$.
\item
If $r$ is not of type $T_0^*$, then it can be moved
by a strong isotopy so that the fiber $f_Q^{-1}(t), t\in \P_\R^1$ that contains the J-tangency point
will not intersect the ovals of $C_\R$.
The connected component of $I_r$ containing the point $t$ obtained,
as well as the component of $J_r$ containing the resulting J-tangency point,
 do not depend on such isotopy.
\end{enumerate}
\end{lemma}

\begin{proof}
Claim (1) is straightforward. 
Claim (2) is trivial, if the J-tangency point is not under undertangent oval.
Otherwise, for proving (2),
we just need to move to $I_r$ the J-tangency point,
which 
can be obviously done 
by a strong isotopy,
and to notice that presence of an undertangency point shows that the direction of this moving is uniquely defined.
\end{proof}

For a full description of $r$ up to an ambient isotopy in $Q_+$ we need to enrich the codes introduced above by an
information about the location of J-tangency points (if any). For that, we push the J-tangency points to $I_r$ using Lemma \ref{moving-to-I}
and specify the connected components of $I_r$ containing new positions of J-tangencies.

It turns out 
that such information is required only if $r$ has type $T_2$, while in the other cases the question about J-tangencies does not rise.
For type $T_0$,
it is because the interval $I_r$ is connected.
For type $T_0^*$,
the position of the J-tangency is prescribed by the definition of $T_0^*$.
For type
$T_1$,
 the simplification procedure allows to remove the J-tangencies due to
Lemma \ref{removable}, 
and for type
$T_3$,
there is no J-tangencies at all.

In the case of 
type $T_2$,
we distinguish the component of $I_r$ containing the J-tangency
by means of delimiters $\la$ and $\ra$ that mark the endpoints
of this component. 
Note that 
some number of symbols \upp\,  may be enclosed by the delimiters.

For example,
the code \ \uptan $\lla$ \upp \ \upp \ $\rra$ \uptan\ \ 
refers to $C_\R$ with 4 ovals and a tritangent $r$ underpassing the second and the third ovals and
undertangent the first and the fourth ovals.
The brackets indicate presence of
a J-tangency between the first and the fourth ovals.
As additional examples,
the 4 
 tritangents 
shown on Fig.\,\ref{4pics} can be encoded respectively as\ \ 
 \bott,\  \hskip2mm {\VV},\hskip2mm  \uptan,\, and\hskip2mm  \bottan.

\begin{lemma}\label{codes-determine}
The code of a simple loose section $r$ determines it uniquely up to an ambient isotopy in $Q_+$.
\end{lemma}

\begin{proof}
If the J-component contains no zigzags, then
loose sections with the same code can be connected by a strong isotopy.
To connect loose sections in presence of zigzags it is enough to perform zigzag moves.
\end{proof}

\subsection{Restrictions on the position of $J$-tangencies}\label{J-tangencies}
These restrictions concern the tritangents of type $T_2$, and only in the cases  $\la 4\,|\,0\ra$ and $\la 3\,|\,0\ra$.

\begin{proposition}\label{restriction-tangencies}
Let $\Tri$ be a positive tritangent of type $T_2$ to a real sextic $C$ of type $\la 4\,|\,0\ra$ or $\la 3\,|\,0\ra$.
\begin{enumerate}
\item If $C$ is of type $\la 4\,|\,0\ra$ and the pair of tangent to $\Tri_\R$ ovals include precisely one of the ovals $O_1$ and $O_3$, then
the non-tangent ovals lie both above or both below $\Tri_\R$ while the J-tangency point belongs to that interval of 
$\Tri$
delimited by projections of the tangency points with ovals which contains the projection of non-tangent ovals 
in the 
case ``both above" and does not contain the projection of non-tangent ovals in the 
case `both below".
\item If $C$ is of type $\la 4\,|\,0\ra$ and the two tangent ovals are either $O_1$ and $O_3$ or $O_2$ and $O_4$, then
one of the non-tangent ovals lie above and one lie below $\Tri_\R$, while the J-tangency point belongs to that interval of 
$\Tri$
delimited by projections of the tangency points with ovals which contains the projection of the non-tangent oval lying above $\Tri_\R$.
\item If $C$ is of type $\la 3\,|\,0\ra$, then the J-tangency point belongs to that interval of $\Tri$
delimited by projections of the tangency points with ovals which contains the projection of the non-tangent oval if the latter one lies above $\Tri_\R$,
and does not contain the projection of the non-tangent oval otherwise.
\end{enumerate}
\end{proposition} 
\begin{proof}
To justify each of the above restrictions on the position of the $J$-tangency points on the $J$-component
we assume the contrary and 
trace an auxiliary plane section 
$h$
that contradicts to the B\'ezout theorem applied to $h\cap\Tri$.  
Typical examples are shown on Fig.\,\ref{byBezout}. 
The upper row shows the correct location of the J-tangency points
in each of the 4 chosen examples.
In the bottom row we demonstrate why another location
 is forbidden.
For that, in each example we present
a section $h$  (the dotted curve) with $|h\cap\Tri|>2$, which contradicts to the B\'ezout theorem.
This section $h$ is chosen to intersect the 3 ovals chosen as indicated on Fig.\,\ref{byBezout}.
If any of these 3 ovals is tangent to $\Tri$, then $h$ is chosen to intersect it at the tangency point.
Note also that 
$h$ 
passes
above the J-component (shown as the bottom line segment), because by the Bézout theorem $h\cap C$
cannot have more than 6 points.
\end{proof}

\begin{figure}[h!]
\caption{
Examples of realizable and not realizable $J$-tangencies}\label{byBezout}
\includegraphics[height=2.8cm]{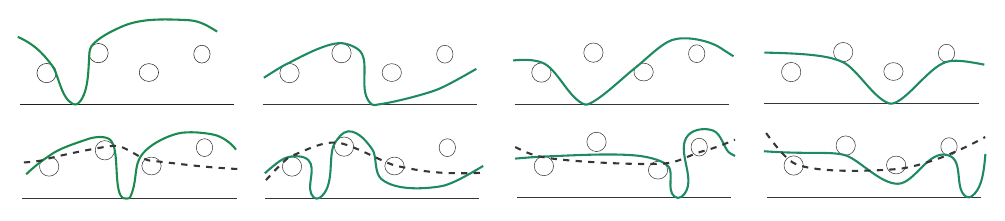}
\end{figure}

\begin{proposition}\label{flipping-oval}
A real sextic has a positive tritangent of type \cod{x}{\uptan}{y} if and only if it has a positive tritangent of type \cod{x}{\bottan}{y}.
\end{proposition}

\begin{proof}
In terms of
homology classes of
 real lines on $Y$ that cover a tritangent, switching from one type to another is equivalent to adding the oval-class $e$ to $v^o$ in the oval/bridge decomposition, as it follows from Proposition \ref{oval-bridge-components}.
\end{proof} 

\subsection{Sextics of type $\la p\,|\,0\ra$}
We say that the code of a sextic of type $\la p\,|\,0\ra$, $p<4$, is a derivative of a code of type $\la 4\,|\,0\ra$, if the first is obtained from the second by dropping 
$4-p$ symbols of types \upp \, and \bott \,.

\begin{theorem}\label{answer-codes}
The 
120 
positive tritangents to a real sextic $C\subset Q$ of type $\la 4\,|\,0\ra$ in their simplified forms {\rm (}as loose sections{\rm )} have the codes listed in Tab. \ref{pics120}. For sextics of type $\la p\,|\,0\ra$ with $p<4$, the codes 
are exactly the derivatives of the above ones.
Respectively, for every real del Pezzo surface $Y$ of degree 1
with $Y_\R=\Rp2\#p\T^2$, $p\le 4$, a bijection between the set of codes of positive tritangets
to a real sextic of type $\la p\,|\,0\ra$ and the set of isotopy types of real lines on $Y$ is given by passing from a code to the isotopy type of pseudo-lines covering the loose sections given by the code.
\end{theorem}

\begin{proof} Passage to a simplified form for positive tritangents to $C$, 
and to 
pseudo-lines on $Y$ that cover them, is justified by Proposition \ref{stretching} and  Lemmas \ref{lift-isotopy}, \ref{codes-determine}. Corollary \ref{Pairs-Mtype} and Proposition \ref{tan-in-p}
give us the list of pairs $(S_{in},S_{tan})$
for sextics of type $\la 4\,|\,0\ra$ and those of type $\la p\,|\,0\ra$, $0\le p\le3$, respectively.
Proposition \ref{restriction-tangencies} determines, for sextics of type $\la 4\,|\,0\ra$ and $\la 3\,|\,0\ra$,
the positions of the J-tangency point in the case of type $T_2$.
For other types, the codes do not contain any information on J-tangencies, and the latter is not needed for determining
the isotopy type of the simplifications as it follows from  Lemma \ref{codes-determine}. 
\end{proof}

\subsection{Comments on 
Table \ref{pics120} and 
its derivatives}
For a fixed 
real sextic $C\subset Q$ of type $\la 4\,|\,0\ra$, 
we pick a real generatrix of $Q_\R$ not intersecting the ovals of $C$, and thus, after we orient the base 
$\P^1_\R$ of the cone $Q_\R$, we obtain an ordering of ovals of $C$.
We choose
a reference generatrix and an orientation of $\P^1_\R$ so that the first oval is a lower one.
The symbols of the codes 
appear in the table  in the corresponding order, so that
the leftmost symbol in each code refers to that lower oval.
If a code contains
$k\ge0$ ambivalent symbols \upbottan, it represents $2^k$ different codes obtained by independent replacements of each \upbottan\  by
\uptan\ or \bottan.
Accordingly, such code is counted with multiplicity $2^k$
in the
rightmost column of the table, where the number of tritangents represented in each row is given.
The J-tangency points are determined in accordance with Proposition \ref{restriction-tangencies}
and we indicate their location intervals with delimiters
$\lla\ \rra$. It is done only for the $T_2$-type, since in the other case this location is 
as indicated in Section \ref{encoding}.

When we pass from type $\la 4\,|\,0\ra$ to type $\la p\,|\,0\ra$ with $p<4$, the
number of derivative codes decreases, since the
same code may be obtained by simplification of different codes 
presented in
Table \ref{pics120}.
For example, 
the code \upp\, \upp\, \bott\  can be obtained from
\upp\, \upp\, \upp\, \bott\  by dropping any of its 3 symbols \upp, or by dropping the last symbol
in the code \upp\, \upp\, \bott\, \upp. By contrary, since it is forbidden to drop symbol \upbottan, the given code
 \upp\, \upp\, \bott\  is not a derivative of\ \upbottan \,\upp\, \upp\, \bott\ (and for the same reason, not a derivative
 of any other codes in rows labeled $T_1$, $T_2$ and $T_3$).

 \begin{table}[h!] 
\caption{The codes realizable by 120 positive simplified tritangents 
(the rightmost column shows the number of tritangents per row)}\label{pics120}
\begin{tabular}{c|c|c|c|c|c|c}
$T_0$&\upp\, \upp\, \upp\, \bott&\upp\, \upp\, \bott\, \upp& \upp\, \bott\, \upp\, \upp&\bott\, \upp\, \upp\, \upp&$4$\\
\\
$T_0^*$&\VV\, \bott\, \bott\, \bott &\bott\, \VV\, \bott\, \bott&\bott\, \bott\, \VV\, \bott&\bott\, \bott\, \bott\, \VV&$4$\\ \\
$T_1$&\upbottan \,\upp\, \upp\, \bott&\upbottan\, \upp\, \bott\, \upp&\upbottan\, \bott\, \upp\, \upp&\upbottan\, \bott\, \bott\, \bott&$8$\\ \\
$T_1$&\upp\, \upp\,  \upbottan\, \bott&\upp\, \bott\, \upbottan\, \upp&\bott\, \upp\, \upbottan\, \upp& \bott\, \bott\, \upbottan\, \bott&$8$\\ \\
$T_1$&\upp\, \upbottan\, \bott\, \bott&\bott\, \upbottan\, \bott\, \upp&\bott\, \upbottan\, \bott\, \upp&\upp\, \upbottan\, \upp\, \upp&$8$\\ \\
$T_1$&\upp\, \bott\, \bott\, \upbottan&\bott\, \bott\, \upp\, \upbottan&\bott\, \upp\, \bott\, \upbottan&\upp\, \upp\, \upp\, \upbottan&$8$\\ \\
$T_2$&\upbottan\, \upbottan\,$\lla$\,\upp\, \upp\,$\rra$&\upbottan\,$\lla$\,$\rra$\,\upbottan\,\,\bott\, \bott\, &
\upp\, $\rra$\,\upbottan\, \upbottan\,$\lla$\,\upp&\bott\, \upbottan\,$\lla$\,$\rra$\,\upbottan\, \bott&$16$\\ \\
$T_2$&$\la$\,\upp\, \upp\,$\ra$\,\upbottan\, \upbottan&\bott\, \bott\, \upbottan\,$\la$\,$\ra$\,\upbottan&
\upbottan\,$\lla$\,\upp\, \upp\,$\rra$\,\upbottan& $\rra$\,\upbottan\, \bott\, \bott\, \upbottan\,$\lla$&$16$\\ \\
$T_2$&\upbottan\,$\lla$\,\upp\,$\rra$\,\upbottan\, \bott&\upbottan\, \bott\, \upbottan\,$\lla$\,\upp \,$\rra$&
$\lla$\,\upp\,$\rra$\,\upbottan\, \bott\, \upbottan&\bott\, \upbottan\,$\lla$\,\upp\,$\rra$\,\upbottan&$16$\\ \\
$T_3$&\upbottan\, \upbottan\, \upbottan\, \upp \ & \upbottan\, \upbottan\, \bott\, \upbottan\ &
\upbottan\, \upp\, \upbottan\, \upbottan\  & \bott\, \upbottan\, \upbottan\, \upbottan&$32$\\
\end{tabular}
\end{table}

\section{Preliminaries on rational elliptic surfaces}\label{sect5}
Throughout this section we assume that $f\,: X\to \P^1$ is an elliptic surface satisfying assumption $\A$.

\subsection{Lines on a rational elliptic surface}\label{lines-elliptic-surface}
As is known, the lines in $X$ can be distinguished by their homology classes in $H_2(X)$ as follows.

\begin{proposition}\label{line-root-correspondence}
Let $L\subset X$ be a line.
Then:

$(1)$\ $\la K,L\ra^\perp\subset H_2(X)$ is isomorphic to $E_8$.

$(2)$\ 
There is a natural 1--1 correspondence between the lines in $X$ and elements of $E_8=\la K,L\ra^\perp$ that 
associates with each $v\in E_8$ the line which is uniquely determined
by its 
homology
class
$L_v=L+\frac{v^2}2K+v$.

$(3)$\ 
A  line $L_v$ is real if and only if $v\in\L=E_8\cap\ker(1+\conj_*)$.
\end{proposition}

\begin{proof} For items (1)--(2) see \cite{shioda}, while (3) follows from functoriality of the correspondence in (2).
\end{proof}

Contraction of a line $L\subset X$ gives a del Pezzo surface $Y=X/L$ of degree $K^2_Y=1$.
The following relation between lines on $Y=X/L$ and  lines on $X$ is straightforward from the definition of lines given in Section \ref{conventions}.

\begin{proposition} Each of the lines in $Y=X/L$ lifts to one and only one line in $X$, which establishes a bijection between the set of lines in $Y$ and the set of lines in $X$ disjoint from $L$. If $X$ and $L$ are real, then this induces a bijection between the set of real lines in $Y$ and the set of real lines in $X$ disjoint from $L$.
\qed\end{proposition}

Note also that the homomorphism $\phi_* : H_2(X)\to H_2(Y)$ induced by the contraction $\phi : X\to Y=X/L$ establishes an isomorphism
between $E_8=\la K,L\ra^\perp$, $K=K_X$, and $E_8= K_Y^\perp\subset H_2(Y)$. Using this canonical isomorphism, we will omit $\phi_*$ and $\phi_*^{-1}$, as 
long
as it does not lead to a confusion.

\subsection{Fibers of a real elliptic fibration}\label{52}
In subsections \ref{52} -- \ref{Di} we restrict ourselves with the connected case $X_\R=\Kl\#p\T^2$, and 
later on
apply the same conventions to the component
$\Kl\#p\T^2$ if $X_\R=\Kl\#p\T^2\+q\SSS^2$.

\begin{lemma}\label{critical}

$(1)$ The  
mapping
$f_\R:X_\R=\Kl\#p\T^2\to \P^1_\R$
 has an even number, $2r\ge 0$, of non-degenerate critical points and
 the same number 
 of 1-nodal singular fibers. The degenerate critical points correspond to cuspidal fibers.

$(2)$
Non-singular fibers $f_\R^{-1}(x)$, $x\in\P^1_\R$, have either 1 or 2 connected components, and 
these numbers alternate as $x$ 
varies passing a
non-degenerate critical value.
Furthermore,
we can cyclically in $\P^1_\R$ enumerate the non-degenerate critical values as $x_1,\dots, x_{2r}$ so that
the  fibers $f_\R^{-1}(x)$ are nonsingular and  have 2 components on the intervals $]x_{2i-1},x_{2i}[$, $1\le i\le r$,
while on the intervals $]x_{2i},x_{2i+1}[$, $1\le i\le r$, they are 1-component and either non-singular or cuspidal. In particular, the number of connected components does not alternate and equal to 1, when 
$x$ varies passing a degenerate critical value.

$(3)$ 
If 
$r=0$, then all fibers are connected.
\end{lemma}

\begin{proof}
Since the Euler characteristic of real nonsingular fibers and that of cuspidal fibers  is $0$, while for the singular 1-nodal fibers it is equal to $\pm 1$, the parity of the number of non-degenerate singular fibers follows from the parity of the Euler characteristic
of $X_\R$. Due to assumpation $\A$, all degenerate singular fiber are cuspidal. This implies Claim (1).
The number of connected components of $f_\R^{-1}(x)$ is $\ge 1$,
since the mapping $f_\R $ has a section,
and $\le2=\frac12b_*(\T^2)$ due to Harnack's inequality.
Alternation of the number of connected components when $x$ varies passing a nondegenerate critical value follows from orientability
of a real fiber neighborhood. Thus, to complete proving Claim (2), it remains to notice that cuspidal fibers are arising from generatrices of $Q$ intersecting the sextic with multiplicity 3.
Claim (3) 
holds
due to connectedness of $X_\R$ and existence of 
a section for $f_\R$.
\end{proof}

\begin{lemma}\label{fragments}

$(1)$ The complement $X_\R\sm f_\R^{-1}(x)$ of any connected fiber is a connected orientable surface of genus $p$ with 2 holes.

$(2)$ If the critical values $x_1,\dots,x_{2r}$ are enumerated as in Lemma \ref{critical}, then
there exist precisely $p$ pairs $x_{2i-1},x_{2i}$ with the Morse indices 1 for each value.

$(3)$ For each of the above $p$  pairs $x_{2i-1},x_{2i}$ and every $0<\e<\!\!\!<1$,
the part
$N_i=f_\R^{-1}[x_{2i-1}-\e,x_{2i}+\e]$ of $X_\R$
is a torus with two holes bounded by circles
$c_i=f_\R^{-1}(x_{2i-1}-\e)$ and $d_i=f_\R^{-1}( x_{2i}+\e)$. Each of 
$A_i=f_\R^{-1}[x_{2i}+\e, x_{2i+1}-\e]$
is homeomorphic to a cylinder.
\end{lemma}

\begin{proof}
Connectedness in (1) is due to connectedness of $X_\R$ and existence of a section 
for $f_\R$,
while orientability is due to that $w_1(X_\R)$ is dual to a fiber.

It follows from Lemma \ref{critical}(1-2) that
a fragment $f_\R^{-1}[x_{2i-1}-\e,x_{2i}+\e]\subset X_\R$ is
a torus with 2 holes
if both $x_{2i-1}$ and $x_{2i}$ have index 1. Moreover, otherwise the indices differ by 1
and form a pair of critical points which can be mutually canceled in the sense of Morse theory.
This implies (2) and the first part (3). The cylindricity property follows from the above cancelation rule and from connectedness of the fibers in a neighborhood of cuspidal ones (see Lemma \ref{critical}(2)).
\end{proof}

\subsection{A system of cuts}\label{cut-section}
In accordance with Lemma \ref{fragments}, let
us  cut $X_\R$ along a 1-component fiber $c_1$
to obtain a compact surface $N$ (a compactification of $X_\R\sm c_1$),
which is connected, oriented and
projects to the interval $I_N$ obtained by cutting $\P^1_\R$ at a point.
Next, we cut $N$ along one-component real fibers
$c_i$, $d_i$ to split it into subsurfaces $N_1, A_1, \dots, N_p, A_p$ called {\it fragments} of $N$ (see Fig.\,\ref{cutting}).
The projection $f|_{N_i}$
has precisely two critical points separated by a 2-component fiber $a_i\cup b_i$.
The fragments $A_i$ are cylindrical, but the projection $f|_{A_i}$ may be not a fibration.

\begin{figure}[h!]
\caption{A system of cuts: $N_i$, $A_i$, $c_i$, $d_i$, $a_i$, $b_i$}\label{cutting}
\vskip2mm
\includegraphics[height=2.4cm]{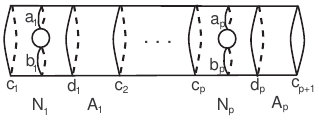}
\end{figure}

We pick one of the two orientations of $N$ for all further considerations, and consider right-handed
Dehn twists $t_x\in\Map(N_i)\subset\Map(N)$
about curves $x=c_{i},d_i,a_i,b_i$ (for injectivity of $\Map(N_i)\to\Map(N)$ see, for example,  \cite[Theorem 3.18]{farb}).
Here,
$t_{c_{i}}$ and $t_{d_i}$ are the {\it boundary Dehn twists},
that is the Dehn twists about curves obtained by a shift of $c_i, d_i$ inside $N_i$.

\begin{lemma}\label{cuts}
For $p\ge0$, 
the Dehn twists $t_{c_i}, t_{a_i}, t_{b_i}\in \Map(N)$, $1\le i\le p$ and $t_{c_{p+1}}$ 
form a basis of
a free abelian subgroup of rank $3p+1$ in  $\Map(N)$.
The image of this group in $\Map(X_\R)$ is
obtained by adding one relation $t_{c_1}t_{c_{p+1}}=1$.
In particular, for $k>0$ this image is a free abelian group of rank $3p$, while
for $p=0$ this image is $\Z/2$ generated by the image of
$t_{c_1}$.
\end{lemma}

\begin{proof} This is a straightforward consequence of \cite[Lemma 3.17]{farb} in what concerns $\Map(N)$ and \cite[Theorem 3.6]{stukow} in what concerns $\Map(X_\R)$.
\end{proof}

\subsection{Fiberwise mapping class groups}\label{fiber-groups}
If $f_F: F\to I_F$ is a fragment
of 
$f:X_\R\to \P^1_\R$ 
fibered over some
segment $I_F=[x_-,x_+]\subset I_N$  (like $N_i,A_i$, and $N$ itself),
we denote by $\til G(F)$ the group 
formed by fiberwise diffeomorphisms $F\to F$ that act as a
group translation
in each fiber of $f_F: F\to I_F$.
The subgroup $G(F)\subset\til G(F)$ is formed by the diffeomorphisms whose restriction to $\partial F$ is
the identity.
The image of the natural 
homomorphism, $\pi_0(G(F))\to\Map(F)$,
will be denoted by $\Maps(F)$
and the 
image in $\Maps(F)$ of elements $g\in G(F)$
by $[g]\in\Maps(F)$.

We include the whole fibration $f:X_\R\to \P^1_\R$  in the list of fragments, and apply to it the same definitions and notation as above
(with replacement $F$ by $X_\R$).

\begin{proposition}\label{commutativity}
For any fragment $F$, groups $\til G(F), G(F)$ and  $\Maps(F)$ are abelian.
\end{proposition}

\begin{proof} It is an immediate consequence of commutativity of group shifts.
\end{proof}

\begin{lemma}\label{torus-group}
For $i=1,\dots, p$,
$\Maps(A_i)\cong\Z$ with a generator $t_{c_{i+1}}$.
\end{lemma}

\begin{proof}
It follows from 
realizability of boundary
Dehn twists  $t_{c_{i+1}}$ in $\Map(A_i)\cong\Map(S^1\times[0,1])=\Z$
by fiberwise 
group-shift diffeomorphisms identical on $\partial A_i$.
\end{proof}

Given two  smooth
sections $\ll_i:I_F\to F$, $i=1,2$, 
let $\gg{\ll_1}{\ll_2}\in \til G(F)$
denote the uniquely defined element of $\til G(F)$ that sends
$\ll_1$ to $\ll_2$.
If we assume in addition that $\ll_1$  coincides with $\ll_2$ at the endpoints of $I_F$, 
then 
$\gg{\ll_1}{\ll_2}\in G(F)$.

A smooth section $\lambda_0:I_F\to F$ being fixed, we define $\Sec(F,\l_0)$
to be the space of smooth sections
$\lambda:I_F\to F$ satisfying the boundary condition $\lambda(x_\pm)=\ll_0(x_\pm)$. By $\Sec(X_\R)$ we denote the space
of all smooth sections $\lambda : \P^1_\R\to X_\R$. 

\begin{lemma}\label{section-action}
For any fixed smooth section $\lambda_0:I_F\to F$, the mapping $\Sec(F,\l_0)\to G(F)$
assigning to $\lambda\in\Sec(F,\l_0)$
the diffeomorphism $\gg{\lambda_0}{\lambda}\in G(F)$, is a homeomorphism 
with respect to the natural topology.
 This defines a natural epimorphism from $\pi_0(\Sec(F,\l_0))=\pi_0(G(F))$ to $\Map^s(F)$,
 which in its turn induces a bijection between $\Map^s(F)$ and the set of equivalence classes of smooth sections
 $\lambda\in \Sec(F,\l_0)$ under ambient isotopies whose restriction to $\partial F$ is the identity.
\end{lemma}
\begin{proof} Homeomorphism property and surjectivity  of the induced map are straightforward. To prove injectivity, assume that two diffeomorphisms,
$g_1=\gg{\lambda_0}{\lambda_1}$ and $g_2=\gg{\lambda_0}{\lambda_2}$, represent the same element of $\Map^s(F)$. Then, $g_2 g_1^{-1}$ represent the 
identity
element, which means
that there exists a continuous family of diffeomorphisms $h_t : F\to F, t\in [0,1],$ such that $h_0=\id$ and $h_1=g_2 g_1^{-1}$. This gives an isotopy between $h_0(\lambda_1)=\lambda_1$
and $h_1(\lambda_1)=\lambda_2$.
\end{proof}

\begin{lemma}\label{simulteneous}
Assume that $F\to I_F$ is a connected 
fragment, and one of its boundary fibers, $\del_- F$ or $\del_+F$, has two connected components, $a$ and $b$.
Then the mapping class
$t_a^mt_b^n$, $m,n\in\Z$, belongs to $\Maps(F)$ if and only if $m=n$.
\end{lemma}

\begin{proof}
Without loss of generality we may suppose that $\del_+ F=a\cup b$.
Since $F$ is connected, there exist 
smooth
sections $\ll_0$ and $\ll_1$ of $F$ intersecting the fiber $\del_+F$ at some points
of $a$ and $b$ respectively, and the fiber $\del_- F$ both at the same point.
Then the diffeomorphism $h=\gg{\ll_0}{\ll_1}\in\til G(F)$ interchanges $a$ and $b$ and preserves (any chosen) orientation of $F$.

If $t_a^mt_b^n$ belongs to $\Maps(F)$, then, by
Lemma \ref{section-action}, it is a class of a diffeomorphism $g=\gg{\ll_0}{\ll_2}$ for some
$\ll_2\in\Sec(F,\l_0)$.
So, due to Lemma \ref{commutativity}, we have $h^{-1}g h=g$.
On the other hand, $[h^{-1}g h]=t_a^nt_b^m$, since $h$ permutes $a$ and $b$ and is orientation preserving.
This implies $m=n$,  since $t_a, t_b$ generate a free abelian subgroup (see  \cite[Lemma 3.17]{farb}).

It remains to notice that, for $\ll_0, \ll_2$ shown on Fig. \ref{pants-s},
 $g=\gg{\ll_0}{\ll_2}$ gives either $t_at_b$ or $t_at_b^{-1}$, and that the second option is eliminated by the previous argument.
\end{proof}

\begin{figure}[h!]
\caption{Pair of smooth sections representing $t_at_b$.}\label{pants-s}
\includegraphics[height=2.5cm]{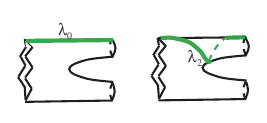}
\end{figure}

By a {\it pair-of-pants fragment}
$f_F: F\to I_F$ we mean a fragment diffeomorphic to a pair-of-pants for which $f_F$ is a Morse function with only one critical point.

\begin{lemma}\label{pants-group}
Assume that $F$ is a pair-of-pants fragment with $c$ and $a\cup b$ as boundary fibers.
Then $\Maps(F)=\la t_at_b,t_c\ra\cong\Z^2$ and $\pi_0(G(F))\to\Maps(F)$ is an isomorphism.
\end{lemma}

\begin{proof}
As is  well-known (cf., Lemma \ref{cuts})
$\Map(F)=\la t_a,t_b,t_c\ra\cong\Z^3$,
where each of the boundary twists involved
can be realised by a fiber-preserving map.
Moreover,
in $\Maps(F)$, 
the Dehn twists $t_a$ and $t_b$ can only be applied simultaneously,
by Lemma \ref{simulteneous},
and any such a simultaneous twist can be realized
by an element in $\Maps(F)$.
The kernel of $\pi_0(G(F))\to\Maps(F)$ it trivial since the map $\pi_0(\Sec(F,\l_0))\to H_1(F)$ induced by assigning to 
$\ll\in \Sec(F,\l_0))$ the element of $H_1(F)$ realized by the loop $\ll *\ll_0^{-1}$ is injective (the latter is a trivial consequence of
uniqueness of the critical fiber in $f_F : F\to I_F$).
\end{proof}

\subsection{Exact sequence for adjacent fibration fragments}
We say that
fragments $F_1\to I_{F_1}$ and $F_2\to I_{F_2}$ are {\it adjacent} if
the intervals $I_{F_1}$ and $I_{F_2}$ intersect at one point,
so that $F=F_1\cup F_2$ is fibered over an interval $I_F=I_{F_1}\cup I_{F_2}$.

\begin{lemma}\label{exact-seq}
{\rm(1)} For the union of adjacent 
fragments $F=F_1\cup F_2$ intersecting along a connected curve $\alpha= F_1\cap F_2$, 
we have the exact sequence
\[
0 \to \Z \to \Maps(F_1)\oplus \Maps(F_2) \xrightarrow{\theta} \Maps(F) \to 0 \]
with the kernel $\Z$
generated by $t_{\alpha_1}\oplus t_{\alpha_2}^{-1}$, where $\alpha_i$ is a copy of $\alpha$ in $F_i$.

{\rm (2)} In the case of 2-component intersection $\a\cup \b=F_1\cap F_2$ of connected
fragments $F_1$ and $F_2$
we have
the exact sequence
\[
0 \to \Z \to \Maps(F_1)\oplus \Maps(F_2) \xrightarrow{\theta} \Maps(F) \to\Z/2 \to  0, 
\]
where 
the kernel $\Z$ is generated by $s_1\oplus s_2^{-1}$,
$s_i=t_{\a_i}t_{\b_i}$ 
($\a_i$, $\b_i$ being copies of $\a, \b$ in $F_i$),
and permutation of the components $\a$ and $\b$ by elements of $\Maps(F)$ defines the projection to $\Z/2=\operatorname{Sym}(\a, \b)$.

{\rm (3)} If for  $F_i$, $i=1,2$ from items {\rm (1)} or {\rm (2)} the epimorphisms $\pi_0(G(F_i))\to\Map^s(F_i)$ are isomorphisms,
then so is $\pi_0(G(F))\to\Map^s(F)$.
\end{lemma}

\begin{proof} 
We treat below the case (2) and skip (1) as a similar and simpler case.

Connectedness of $F_1$ and $F_2$ implies existence of 
smooth
sections $\ll_0$ and $\ll_1$ of $F$
that intersect $a$ and $b$ respectively.
Then $\gg{\ll_0}{\ll_1}\in\Maps(F)$ permutes $a$ and $b$,
which gives surjectivity of $\Maps(F)\to\Z/2$.

For proving the exactness at   $\Maps(F)$, it is sufficient to notice, first, that $Im(\theta)$
preserves the components $\a$ and $\b$ invariant, and second, that any diffeomorphism $g\in G(F)$ preserving $\a$ and $\b$ invariant
can be made identical on these components by twisting via an isotopy in $G(F)$.

On the other hand, $\ker\theta\subset \Ker \{ \Map(F_1)\oplus \Map(F_2) \to \Map(F)\} =\Z^2=\la t_{\a_1}\oplus t_{\a_2}^{-1},
t_{\b_1}\oplus t_{\b_2}^{-1}\ra$ (see, e.g., \cite[Theorem 3.18]{farb}).
 So, if $g_1\oplus g_2\in\Ker\theta$, then $g_i=t_{\a_i}^{k_i}t_{\b_i}^{l_i}$, $i=1,2$, where $k_1+k_2=l_1+l_2=0$,
 while, by Lemma \ref{simulteneous},  $k_i=l_i$.
 Also by Lemma  \ref{simulteneous} the elements $s_1\oplus s_2^{-1}\in \Map(F_1)\oplus \Map(F_2)$ do belong to $\Maps(F_1)\oplus \Maps(F_2)$.
 
 The part (3) follows
 from the following commutative diagrams 
 \[\minCDarrowwidth25pt
 \begin{CD}
 0@>>> \Z @>>> \Maps(F_1)\oplus \Maps(F_2) @>{\theta}>> \Maps(F) @>>> 0 \\
&&@|  @| @AAA  \\
0 @>>> \pi_1(S^1) @>>> \pi_0(G(F_1))\oplus \pi_0(G(F_2)) @>>> \pi_0(G(F)) @>>> 0 
\end{CD}\]
under assumptions of (1), and the following diagram under assumptions of (2)
 \[\minCDarrowwidth15pt
 \begin{CD}
 0@>>> \Z @>>> \Maps(F_1)\oplus \Maps(F_2) @>{\theta}>> \Maps(F) @>>>\Z/2  @>>> 0 \\
&&@|  @| @AAA @| \\
0 @>>> \pi_1(S^1) @>>> \pi_0(G(F_1))\oplus \pi_0(G(F_2)) @>>> \pi_0(G(F)) @>>> \Z/2 @>>> 0 
\end{CD}\]
The lower rows come from the exact homotopy sequence of the group-quotient fibration $G(F_1)\times G(F_2)\to G(F)\to G(F_1\cap F_2)$
where by $G(F_1\cap F_2)$ we understand the group of shifts of the fiber $F_1\cap F_2$ over the point $I_{F_1}\cap I_{F_2}$, which is 
 $S^1$ for a 1-component and $S^1\times\Z/2$ for a 2-component fiber.
\end{proof}

\begin{proposition}\label{no-zigzag-isomorphism}
Let $f_F: F\to I_{F}$ (respectively, $f_\R : X_\R \to \P^1_\R$) 
be  a fragment (respectively, the whole real elliptic surface) not containing zigzags
and cuspidal fibers,
and $\l_0:I_{F}\to F$ (respectively, $\l_0 : \P_\R^1\to X_\R$)
a
smooth
section.
Then, the natural epimorphism from $\pi_0(\Sec(F,\l_0))=\pi_0(G(F))$ to $\Maps(F)$ is an isomorphism.
\end{proposition}

\begin{proof} If $F$ is a cylinder or a Klein bottle, then it is evident (cf., Lemma \ref{torus-group}).
Otherwise, absence of zigzags and cuspidal fibers guaranties that $F$ admits 
a decomposition in pair-of-pants fragments. Then the required claim follows immediately from Lemmas
\ref{pants-group} and \ref{exact-seq}(3) if $F$ is a proper fragment. If $F=X_\R$, then it follows from the case $F=N$ by means of \cite[Theorem 3.6]{stukow}.
\end{proof}

\subsection{The elements $\boldsymbol{\D_i\in\Maps(N_i)}$}
\label{Di}
Let $N_i$ be one of the fragments introduced in Section \ref{cut-section}.
Consider the interval $f(N_i)=[y_i,z_i]\subset I_N$, pick a smooth section $\ll_0:[y_i,z_i]\to N_i$, and 
equip the fibers of $f_{\vert_{N_i}}$ with a group structure  for which $\ll_0$ is the 
zero section. With respect to this group structure, the fixed point set of the  fiberwise group involution $g\mapsto g^{-1}$ 
consists of
$\ll_0$, an oval and a smooth section $\ll_1:[y_i,z_i]\to N_i$ as shown on Fig.\,\ref{Delta-section}.
If $\ll_0$ is a part
of some real line $L\subset X$, then $\ll_1$ and the oval are parts of the sextic $C_\R$ associated with $L$.
 
 Let us consider also a smooth section $\ll:[y_i,z_i]\to N_i$ that represents a {\it half of a Dehn twist about a fiber} (in accord with a fixed orientation of $N_i$)
 above each of two small small intervals  $[y_i,y_i+\e]$, $[z_i-\e,z_i]$ and coincides with $\ll_0$ (resp.\,$\ll_1$) on $\partial[y_i,z_i]$ (resp. 
 $[y_i+\e, z_i-\e]$), see Fig. \ref{Delta-section}).
\begin{figure}[h!]
\caption{Smooth section $\ll$ representing $\D_i\in\Maps(N_i)$}\label{Delta-section}
\includegraphics[height=2.5cm]{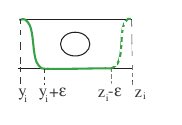}
\caption*{\footnotesize The upper and bottom segments depict the smooth sections $\ll_0$ and $\ll_1$, respectively. 
}
\end{figure}

\begin{proposition}\label{Delta-lemma}
The smooth section $\ll$ introduced above is well-defined up to isotopy 
fixed at the boundary,
 the corresponding to it
element $\D_i=\gg{\ll_0}{\ll}\in\Maps(N_i)$ satisfies the following  properties
\begin{enumerate}
\item  $t_{a_i}\D_i=\D_it_{b_i}$,  $t_{b_i}\D_i=\D_it_{a_i}$,
\item $\D_i^2=t_{c_{i}}t_{c_{i+1}}$,
\end{enumerate}
where  $a_i,b_i,c_i$  are the cut-curves introduced in Section \ref{cut-section}.
Moreover, any $\D\in\Maps(N_i)$, satisfying
$t_{a_i}\D=\D t_{b_i},  t_{b_i}\D=\D t_{a_i},
\D^2=t_{c_{i}}t_{c_{i+1}}$ coincides with $\D_i$.
\end{proposition}

\begin{proof}
Note that for any element $\D\in\Maps(N_i)$ the relations (1) are equivalent to that $\D$ interchanges the curves $a_i$ and $b_i$, 
which is obviously true for $\D=\D_i$.
It follows also from the definition of $\ll$ that $\D_i^2$ 
performs the Dehn twists $t_{c_i}$ and 
$t_{d_i}=t_{c_{i+1}}$
on the intervals $[y_i,y_i+\e]$ and $[z_i-\e,z_i]$ respectively and the identity in $[y_i+\e,z_i-\e]$, so,
(2) is also satisfied.

To show the required uniqueness of $\D_i$, suppose that $\D\in\Maps(N_i)$ is another element satisfying (1) and (2).
Then, the property (1) implies that
 $\D_i$ and $\D$ are both
cross-sections of the epimomorphism $\Maps(N_i)\to\Z/2$ of Lemma \ref{exact-seq}.
Therefore, by Lemma \ref{exact-seq}, we have $\D_i\D^{-1}=\theta(\delta_0\oplus\delta_1)$, $\delta_j\in\Maps(F_j)$, where $N_i=F_0\cup F_1$ is the pair-of-pants decomposition
obtained by cutting $N_i$ along the fiber $a_i\cup b_i$.
Since, according to Lemma \ref{pants-group}, we have $\delta_0=s_i^{k_0}t_{c_i}^{m_0}$ and $\delta_1= s_i^{k_1}t_{c_{i+1}}^{m_1}$,
 the relations (2) for $\D_i$ and $\D$
 give
$$
\theta(\delta_0^2\oplus\delta_1^2)=\D_i^2(\D^{-1})^2=1.
$$
By Lemma \ref{exact-seq}, this implies existence of $k\in\Z$ such that $s_i^{k_0}t_{c_i}^{m_0}=s_i^k$ and $s_i^{k_1}t_{c_{i+1}}^{m_1}=s_i^{-k}$.
Finally, applying Lemma  \ref{pants-group} we conclude that $m_0=m_1=k_0+k_1=0$, which in its turn gives $\D_i\D^{-1}=1$.
\end{proof}

\begin{cor}\label{handle-group}
$\Maps(N_i)\cong\Z^3$ 
with a basis
$t_{c_i}$,
$s_i=t_{a_i}t_{b_i}$ and $\D_i$.
\end{cor}

\begin{proof}
It follows immediately from Proposition \ref{pants-group} and Lemmas \ref{exact-seq}, \ref{Delta-lemma}.
\end{proof}

\subsection{Computation of the group $\Maps(X_\R)$}\label{topologicalMW}
\begin{proposition}\label{maps-spherical}
If $X_\R= \K\# p\T^2\+q\SSS^2$,
then
$\Maps(X_\R)=\Z^{2p}+\Z/2$ 
is generated by the elements
$t_{c_1}$, $s_i$, $\D_i$, $1\le i\le p$ 
with the only relation
$$
t_{c_1}^\epsilon \prod_{1\le 2i+1\le p}\D_{2i+1}^2=\prod_{1\le 2i\le p}\D_{2i}^2 \qquad \text{where}\ \ \epsilon=1+(-1)^p.
$$
\end{proposition}

\begin{proof} After we skip the spherical components and cut the component
$\K\# p\T^2$ in the same way as in Sec. \ref{cut-section}, we obtain a surface $N$.
From Lemma \ref{exact-seq} and 
Corollary \ref{handle-group} it follows that
 $\Maps(N)\cong \Z^{2p+1}$ with a basis formed by $t_{c_1}$ and $s_i$, $\D_i$, $1\le i\le p$.
According to 
\cite[Theorem 3.6]{stukow} the group $\Maps(X_\R)$ is obtained from $\Maps(N)$ by adding  the relation
\be\label{tc-relation}t_{c_{p+1}}=t_{c_1}^{-1}.\ee
Finally, the relation required follows from (\ref{tc-relation}) and 
Proposition \ref{Delta-lemma}(2).
\end{proof}

For computation of $\Maps(X_\R)$
in the remaining case, $X_\R=\Kl\+\Kl$, note that our elliptic fibration $f_\R:X_\R\to\P^1_\R$ cannot have critical points. 
So, it is a nonsingular fibration with a fiber $S^1\+S^1$.
The restriction of $f_\R$ to
 each copy of $\K$ admits a pair of disjoint sections, which we denote
$\ll_1^1$, $\ll_2^1$ for one copy and $\ll_1^2$, $\ll_2^2$ for another.

\begin{figure}[h!]
\caption{Sections $\ll_i^j$}\label{lambda-ij}
\includegraphics[height=3cm]{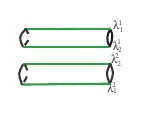}
\end{figure}

\begin{proposition}\label{maps-hyperbolic-case}
If $X_\R=\Kl\+\Kl$, then $\Maps(X_\R)$ is isomorphic to $\Z/2\oplus\Z/2$ and formed by the elements $\{\gg{\ll_1^1}{\ll_i^j}\}_{i,j\in\{1,2\}}$.
\end{proposition}

\begin{proof} Let us choose as $\ll^1_1$ one of real lines and as $\ll^1_2, \ll^2_1, \ll^2_2$ the 3 
smooth
sections
that form together with $\ll^1_1$ the fixed point set of the fiberwise hyperelliptic involution determined by the choice of $\ll^1_1$ as zero
(see Fig. \ref{lambda-ij}). They do form 4 disjoint 
smooth
sections, since in each fiber the fixed points are the points of period 2, and since the monodromy acts identically on the 
$\conj$-invariant part of the period lattice of a fiber and as multiplication by $-1$ on its anti-invariant part, which provides pairwise distinction between the 
points of period 2. Under fiberwise addition these 4 
smooth
sections form a  group $\Z/2\oplus \Z/2$. Thus, there remain to notice that any 
smooth
section of $X_\R$ is isotopic to one of the four $\ll_i^j$, and to apply Proposition \ref{section-action}.
\end{proof}

\subsection{Element $\boldsymbol{\delta}$ of order 2  in $\boldsymbol{\Maps(X_\R)}$}\label{order-2-elements}
If $X_\R\cong \K\# p\T^2\+q\SSS^2$
with $p\ge 1$, then, in accordance with Proposition \ref{maps-spherical},
the unique element of order 2 in $\Maps(X_\R)$ is  
$$\delta=
t_{c_1}^{\frac{1+(-1)^p}2}
\prod_{1\le 2i+1\le p}\D_{2i+1}\prod_{1\le 2i\le p}\D_{2i}^{-1} .
$$
From the presentation of 
the
elements $\D_i$ by smooth sections (see Fig. \ref{Delta-section}) it follows that 
$\delta=\gg{\ll_0}{\ll}$, where $\ll_0$ is 
a section given by a chosen real line
and $\ll$ is a smooth section whose image in $\Q_\R$
goes below ovals $o_i$, $1\le i\le p$ and above the zigzags (if any), without intersection of $\ll_0$,
as it is shown on
Fig.\,\ref{delta-tritangent}.

\begin{figure}[h!]
\caption{The smooth section $\ll=\delta(\ll_0)$ defining the element $\d\in
\Maps(X_\R)$ of order 2 (the top segment depicts $\ll_0$)}
\label{delta-tritangent}
\includegraphics[height=1.2cm]{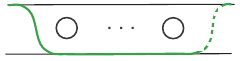}\hskip15mm
\includegraphics[height=1.25cm]{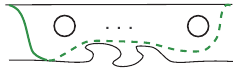}\newline
\hskip5mm without zigzags          \hskip35mm with zigzags \phantom{AAA}
\end{figure}

\subsection{Mordel-Weil group}
Recall that the Mordel-Weil group of an elliptic surface $f:X \to \P^1$ can be defined as a subgroup $\MW (X)$ of the automorphism group $\Aut (X)$ formed by those automorphisms that preserve the fibers of $f$ and act
as a translation in each nonsingular fiber. The Mordel-Weil group acts freely and transitively on the set of
lines,
so that  the latter becomes a torsor over $\MW (X)$. This definition applies
to surfaces over any field. We keep notation $\MW(X)$ for the 
Mordel-Weil group of elliptic surfaces $X$
defined over $\C$, while when $X$ is a real elliptic surface, we notate by $\MW_\R(X)$ the subgroup of $\MW(X)$ formed by the elements 
$g\in \MW(X)$ preserving the real structure. In the latter case, it is the set of real lines in $X$ that becomes a torsor over $\MW_\R(X)$.

Thus, if we fix a line $L\subset X$ (respectively, a real line $L\subset X$) then we can interpret $\MW(X)$ (respectively, $\MW_\R(X)$) as a group structure
on the set of lines in $X$ (respectively, the set of real lines in $X$) by associating with each line $L'\subset X$ (respectively, each real line $L'\subset X$)
an element of $\MW(X)$ (respectively, of $\MW_\R(X)$) that transforms $L$ into $L'$, and which we denote by $\gg{L}{L'}$.

Furthermore, by passing from lines to their homology classes and applying 
the natural correspondence $v\in E_8=\la K,L\ra^\perp\mapsto L_v$  described in Proposition \ref{line-root-correspondence}, one gets the next, well-known, result
(see \cite{shioda}).

\begin{proposition}\label{lattice-MW}
Assume that $X$ is a 
rational relatively minimal
elliptic surface with a fixed line $L\subset X$ and that $f$ has only 
reduced irreducible fibers.
Then the compositions 
$$\begin{aligned}
v\in E_8&\mapsto L_v\mapsto \gg{L}{L_v}\in \MW(X)\\
v\in E_8\cap\ker(1+\conj_*)&\mapsto L_v\mapsto \gg{L}{L_v}
\in \MW_\R(X) \ \text { if } \ X \ \text{ is real }
\end{aligned}$$ 
are group isomorphisms.

In particular, $\MW(X)$ is a free abelian group naturally isomorphic to $E_8$, while $\MW_\R(X)$ is a free abelian group naturally isomorphic to 
$\L=E_8\cap\ker(1+\conj_*)$.
\qed\end{proposition}

By definition each element of $\MW_\R(X)$ preserves the real fibers and act on them by translation. Thus, considering its restriction to $X_\R$ we get
a well defined, natural, homomorphism to $\Maps(X)$, which we denote by $\Phi : \MW_\R(X)\to \Maps(X)$.

\section{Proof of Theorem \ref{MW-homomorphism}}\label{sect6}
Let us fix a real line $L\subset X$ and set $g_v=\Phi \gg{L}{L_v}\in \Maps(X)$,
for every $v\in\L=E_8\cap\ker(1+\conj_*)$ (see Proposition \ref{lattice-MW}).
Recall our convention to use the canonical identification of $\L\subset H_2(X)$ with the isomorphic to it $\L\subset H_2(Y)$ (see Section \ref{lines-elliptic-surface})
as identity, and, in particular, to treat (when it does not lead to a confusion) the oval- and bridge-classes of $Y$ as elements of both $\L\subset H_2(Y)$
and $\L\subset H_2(X)$.

\subsection{Preparation}

\begin{proposition}\label{L-and-g}
The smooth sections $\lambda_0, \lambda$ with $\lambda\in \Sec(F,\lambda_0)$
on the  fragments
$F\to I_F$ of $f: X_\R\to\P^1_\R$ which are
depicted on Fig. \ref{local-figure} represent the elements $g=\gg{\lambda_0}{\lambda} \in\Maps(F)$ that are indicated under the corresponding fragment.
\end{proposition}

\begin{figure}[h!]
\caption{
Examples of [section 
$\mapsto\Maps(F)$-element] correspondence}\label{local-figure}
\includegraphics[height=1.25cm]{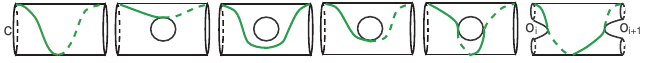}
\newline
\phantom{AAA} $t_c$\hskip17mm $s$\hskip17mm  $\D t_c^{-1}$\hskip13mm  $\D s^{-1}$\hskip11mm  $\D s^{-2}$\hskip10mm  $s_is_{i+1}t_c^{-1}$
\caption*{\footnotesize The top segment depicts the section $\ll_0$, while $\ll$ is drawn in green. By convention, the depicted fragments are equipped with an orientation induced from a fixed orientation of $N$ and, on drawings, this  is the right-hand orientation of the front side.}
\end{figure}

\begin{proof}
By definition of $s, \D,t_c\in\Maps(F)$, each of the elements indicated 
in the bottom of Fig. \ref{local-figure}
is  presented by a  diffeomorphism $g\in G(F)$ sending $\ll_0$ to $\ll$ depicted.
\end{proof}

\begin{lemma}\label{MW-basic}
If 
$X_\R=\Kl\#p\T^2\+q\SSS^2$,
the subgroup $\Im\Phi\subset \Maps(X_\R)$
contains the following elements:
\begin{enumerate}\item
$\D_is_i^{-2}=g_e$ for $e=O_i$ with $i=1,\dots, p$.
\item
$s_is_{i+1}t_{c_{i+1}}^{-1}=g_e$ for $e=B_{i\,i+1}$ with $i=1,2,\dots, p-1$.
\item
$s_i=g_e$ for 
$e= B_i$ 
with $i=1,3$ if $p=4$, $q=0$ and with $i=1,\dots, p$ if $p<4$.
\item
$t_{c_1}=g_e$ if $p=q=0$ and $e$ is any root of $\L=4A_1$.
\end{enumerate}
\end{lemma}

\begin{proof}
Up to changing $e$ by $-e$, (or equivalently, changing  $g_e$ by $g_e^{-1}$), the relations (1)--(3) follow from Lemma \ref{intersection-criterion}, Proposition \ref{oval-bridge-components},
and the correspondence between smooth
sections and elements of $\Maps(F)$
in Proposition \ref{L-and-g} (see  Fig.\,\ref{local-figure}).
The correct choice of sign ($e$ rather than $-e$)
is determined by
$L_{e}\cdot O_i= -e\cdot O_i$
(see Corollary \ref{tire-bouchon} ).

To prove (4), note that  $\Maps(X_\R)=\Z/2$ if $p=q=0$, so, we  just need to check that $g_e\ne 0$ for any root $e\in\L=4A_1$.
For each such $e\in\L=4A_1$, the line $L_e$ projects into a positive real tritangent  
of type $T_0$ (see Section \ref{372}), which implies that $L_{e\R}$ is not isotopic to $L_\R$ and thus, $g_e\ne0$.
\end{proof}

\begin{lemma}\label{MW-Klein2}
Let $X_\R=\K\+\K$, and $\{B_1, B_2, B_3, B_4\}$ be the bridge-classes from Lemma \ref{4-J-bridges}.
Then $g_e\in\Im\Phi\subset\Maps(X_\R)$ leaves each of the two connected components of $X_\R$ invariant if $e\in\{B_1, B_2, B_3, B_4\}$, while it
interchanges them if $e=-\frac12(B_1+B_2+B_3+B_4)$.
\end{lemma}

\begin{proof}
For $e,e'\in\{B_1, B_2, B_3, B_4\}$, the intersection number $L_e\cdot e'=-e\cdot e'$
is even, and therefore $L_{e\R}$ does not intersect
the real loci of these four bridges.
Therefore, $L_{e\R}=g_e(L_\R)$ belongs to the same component of $X_\R$ as  $L_\R$, and, thus, $g_e$ does not interchange 
the components of $X_\R$.

For $e=-\frac12(B_1+B_2+B_3+B_4)$ we have $L_e\cdot B_i=-e\cdot B_i=-1$ (for each $i=1,\dots,4$) which implies that $L_{e\R}=g_e(L_\R)$ intersects 
the real locus of
$B_i$ and, thus, $L_{e\R}$ and $L_\R$ belong to different components of $X_\R$.
\end{proof}

\subsection{Case-by-case proof of Theorem \ref{MW-homomorphism}}\label{case-by-case}
Below,
for any $h\in \Maps(X_\R)$ 
we denote by
$[h]\in\Maps(X_\R)/\Phi(\MW_\R)$ its coset.

\begin{proposition}\label{index2}
If $X_\R=\Kl\#4\T^2$, then: 
\begin{enumerate}
\item
$\Ker\Phi=0$ and $\Im\Phi\cong\Z^8$ has index 2 in $\Maps(X_\R)\cong\Z^8\oplus\Z/2$.
\item
The elements  $s_1,s_3, s_2^2,s_4^2, \D_i$, $i\in\{1, \dots, 4\}$, belong to the group $\Im\Phi$ and generate it.
\item
The quotient $\Maps(X_\R)/\Im\Phi=\Z/2$
is generated by the classes $[s_2]=[s_4]=[t_{c_i}]$.
\end{enumerate}

\end{proposition}
\begin{proof}
The relations $(1)$ of Lemma \ref{MW-basic} give $[\D_i]=[s_i]^2$ for $i\in\{1, \dots, 4\}$.
The relation $(3)$ gives $[s_i]=1$ for $i=1,3$ and hence $[\D_i]=1$ for $i=1,3$. Therefore, the relations (2)
imply $[s_2]=[s_1][s_2]=[t_{c_{2}}]$, $[s_2]=[s_2][s_3]=[t_{c_{3}}]$, and $[s_4]=[s_3][s_4]=[t_{c_{4}}]$, which together with 
$[t_{c_{i}}][t_{c_{i+1}}]=[\D_i]^2$ (see Lemma \ref{Delta-lemma}) gives 
$[t_{c_1}][t_{c_2}]=1$, $[s_2]^2=[t_{c_{2}}t_{c_{3}}]=[\D_2]^2=[s_2]^4$, and $[t_{c_3}][t_{c_4}]=1$.
This implies $[s_2]^2=1$, $[s_2]=[t_{c_{2}}]=[t_{c_{3}}]$, $[s_2]= [s_2]^{-1}=[t_{c_{2}}]^{-1}=[t_{c_1}]$ and $[s_4]=[t_{c_4}]=[t_{c_3}]^{-1}=[s_2]^{-1}=[s_2]$.
In accordance with Proposition \ref{maps-spherical}, 
this also shows that  element $[s_2]=[s_4]=[t_{c_i}]$, $1\le i\le4$, generates
$\Maps(X_\R)/\Phi(\MW_\R)$.

Finally, it remains to notice that $\Maps(X_\R)\ne \Phi(\MW_\R)$, since $\Maps(X_\R)=\Z^{8}+\Z/2$  requires $>8$ generators contrary to $\MW_\R=\Z^8$.
\end{proof}

\begin{proposition}\label{Phi-p}
If $X_\R = \Kl\#p\T^2$ with $0\le p\le3$, then $\Phi(\MW_\R)=\Maps(X_\R)$ and $\Ker \Phi$ is isomorphic to $\Z^{4-p}$.
\end{proposition}

\begin{proof} 
Under the assumption $1\le p\le3$, the bridge-classes $B_i$ exist for every $i=1,\dots,p$, see Fig. \ref{oval-bridge-graphs}.
Applying Lemma \ref{MW-basic} to $g_e$ with $e=O_i$ and $e=B_i$, we get relations 
$[s_i]=1$
and  $[\D_i]=1$
for every $i=1,\dots,p$ (like in the case $p=4$ for $i=3$).

If $p=2,3$, we apply Lemma \ref{MW-basic} to $g_e$ with $e=B_{12}$ and get
$[t_{c_{2}}]=[s_1][s_2]=1$, which implies $[t_{c_{1}}]=[\D_1]^2[t_{c_{2}}]^{-1}=1$ (see Lemma \ref{Delta-lemma}).
If $p=1$, then 
we deduce $[t_{c_{1}}]=1$ from $g_e=t_{c_{1}}$
for $e=B_{11}$ (see, for example, Tab. \ref{roots}).
If $p=0$, then 
we deduce $[t_{c_{1}}]=1$ from $g_e=t_{c_{1}}$ for any of the roots $e\in \L=4A_1$ (see Lemma \ref{MW-basic}).

According to Proposition \ref{maps-spherical} 
the above computation shows surjectivity of $\Phi$. The latter implies $\Ker \Phi\cong \Z^{4-p}$, since $\MW_\R$ is a free abelian group of rank $4+p$, while $\Maps(X_\R)=\Z^{2p}+\Z/2$.
\end{proof}

\begin{proposition}\label{T-and-S}
If $X_\R=\Kl\#\T^2\+\SSS^2$, then:
\begin{enumerate}
\item  $\Phi(MW_\R)$ is isomorphic to $\Z+\Z/2$
and
 generated by $s_1$ and $\D_1$, while $\Maps(X_\R)/\Phi(\MW_\R)=\Z$ is generated by
 $[t_{c_1}]$.
 \item $\Ker \Phi$ is isomorphic to $\Z^3$.
 \end{enumerate}
\end{proposition}

\begin{proof}
Like in the case 
$X_\R = \Kl\#\T^2$
we obtain the relations $[s_1]=[\D_1]=1$ by applying Lemma \ref{MW-basic} to $g_e$ with $e=O_1$ and $e=B_1$.
Since in the case
$X_\R=\Kl\#\T^2\+\SSS^2$ 
the group $\MW_\R$ is generated by $g_e$ with $e=O_1, B_1, B_1',$ and $B_1''$
(see Fig. \ref{oval-bridge-graphs}), to prove item (1) there remains to notice that
$g_e=s_1$ for both $B_1'$ and $B_2'$, and to apply Proposition \ref{maps-spherical}.
Since the only remaining generator, $t_{c_1}\in\Maps(X_\R)$ is not involved, its coset $[t_{c_1}]$ generates the quotient.
Since $\MW_\R$ is a free abelian group of rank $4$, from $\Phi(\MW_\R)=\Z+\Z/2$ it follows that $\Ker \Phi$ is isomorphic to $\Z^3$.
\end{proof}


\begin{proposition}\label{2K-case}
If $X_\R=\Kl\#\Kl$,  then $\Phi(\MW_\R)=\Maps(X_\R)$ and $\Ker\Phi$ is isomorphic to $\Z^4$.
\end{proposition}

\begin{proof}
Surjectivity of $\Phi$ follows from Proposition \ref{maps-hyperbolic-case} and the possibility to realize the 4 disjoint sections involved by real lines (the latter follows, for example, from Proposition 
\ref{12=3x4}). Since $\MW_\R$ is a free abelian group of rank $4$, from $\Phi(\MW_\R)=\Z/2+\Z/2$ it follows that $\Ker \Phi$ is isomorphic to $\Z^4$.
\end{proof}

\begin{proposition}\label{spherical}
If 
$X_\R=\Kl\+q\SSS^2$ with $0<q<4$,  then $\Phi(\MW_\R)=\Maps(X_\R)$ and $\Ker\Phi$ is isomorphic to $\Z^{4-q}$.
\end{proposition}

\begin{proof} By Proposition 
\ref{maps-spherical}
$\Maps(X_\R)= \Z/2$ with the only nontrivial element $t_{c_1}$. Thus, there remains to notice
that $t_{c_1}=\gg{L}{L'}$ for any pair of disjoint real lines $L,L'\subset X_\R$, and that $\MW_\R$ is a free abelian group of rank $4-q$.
\end{proof}

\subsection{Addendum: Lattice description of $\Ker\Phi$}
Our goal here is to give an explicit expression 
for
$\Ker\Phi$ in terms of standard geometric generators of 
$\L=E_8\cap\ker(1+\conj_*)$,
each generator being a root of $\L$ and represented either by an oval- or a bridge-class
(see Fig. \ref{oval-bridge-graphs} and Proposition \ref{graph-adjacency}).
In the following theorem we consider {\it $n$-chains} formed by sequences of $n$ roots in  $\L$
that have pairwise intersection $1$ if consecutive
and $0$ otherwise. The notation for roots is like in  Section \,\ref{oval-bridge-classes} and Lemma \ref{4-J-bridges}.
For instance, in the case $X_\R=\Kl\#3\T^2$, for which $\L=E_7$, we consider a 7-chain
\raisebox{0.7mm}{\scalebox{0.8}{
\xymatrix@1@=15pt{
B_{1} \ar@{-}[r]&*+{O_1} \ar@{-}[r]&B_{12}\ar@{-}[r]& *+{O_2} \ar@{-}[r]  &B_{23} \ar@{-}[r]&*+{O_3}\ar@{-}[r]&B_3}}}
obtained from the 6-chain on the standard diagram of $E_7$
on Fig.\,\ref{oval-bridge-graphs} by adding the root $B_3$, which extends $E_7$-diagram to $\til E_7$. This implies that
$B_3=-(B_1+2O_1+3B_{12}+4O_2+3B_{23}+2O_3+2B_2)$.
In the case $X_\R=\Kl\+\Kl$, for which $\L= D_4$, we consider 
the 4 bridge-classes $B_1, B_2, B_3, B_4$ and their combination $B_0=-\frac12(B_1+B_2+B_3+B_4)$ (see Lemma \ref{4-J-bridges}).

In the case $X_\R=\Kl\+q\SSS^2$, $0\le r\le4$, we have $\Lambda=(4-q)A_1$, but have no oval- or bridge-classes in the sense of Section \ref{oval-bridge-classes}.
However, for uniformity of notation we will denote by $B_i$, $0\le i\le 4-q$, the elements of a root basis of $\Lambda$ (chosen arbitrarily).

\begin{theorem}\label{details-on-MW-homomorphism}
If $X$ satisfies the assumption $\A$, then
$\Ker\Phi\subset\L$ can be expressed as follows:
\begin{enumerate}
\item If $X_\R=\Kl\#4\T^2$, then $\Ker\Phi=0$.
\item If $X_\R=\Kl\#3\T^2$, then 
\scalebox{0.9}{$\Ker\Phi=\{ a(B_1+O_1+B_{12}+O_2+B_{23}+O_3+B_3) \,|\, a\in2\Z\}.$}
\item If $X_\R=\Kl\#2\T^2$, then 
$$\Ker\Phi=\{a_1(B_1+O_1+B_{12}+O_2+B_{2})+a_2(B_1+O_1+B_{12}+O_2+B_{2}')\,|\, a_1+a_2\in2\Z\}.$$
\item If $X_\R=\Kl\#\T^2$, then 
\[\Ker\Phi=\{a_1(O_1+B_1+B_1')+a_2(O_1+B_1+B_1'')+a_3(O_1+B_1'+B_1'')\,|\, a_1+a_2+a_3\in2\Z\}. \]
\item If $X_\R=\Kl\+q\SSS^2$, $0\le q\le4$, then
$\Ker\Phi=\{\sum_{i=1}^{4-q} a_iB_i \,\,|\,\,\sum_{i=1}^{4-q} a_i\in2\Z\}$.
\item If $X_\R=\Kl\+\Kl$, then
$\Ker\Phi= \{
\sum_{i=0}^{3} a_iB_i \,|\,a_0\in2\Z \, \text{and} \,\sum_{i=1}^{3} a_i\in2\Z\}$.
\end{enumerate}
\end{theorem}

\begin{lemma}\label{delta}
If $X_\R\cong\K\#p\T^2$ with $0\le p\le 3$, then
the order 2 element $\delta\in\Maps(X_\R)$
is as follows:
\begin{itemize}\item
If $p=3$, then $\d=\D_1\D_2^{-1}\D_3$ is represented as $g_e$, where $e$ is a 7-chain root
$$e=B_1+O_1+B_{12}+O_2+B_{23}+O_3+B_3
\in \L=E_7.$$
\item
If $p=2$, then $\d=\D_1\D_2^{-1}t_{c_1}^{-1}$ is represented as $g_e$,
where $e$ is a 5-chain root 
 $$e=B_1+O_1+B_{12}+O_2+B_2
 \in \L=D_6.$$
 \item
If $p=1$, then $\delta=\D_1$ is represented as $g_e$, where $e$ is 
a 3-chain root
$$e=B_1+O_1+B_1'
\in \L=D_4+A_1.
$$
\item
If $p=0$, then $\delta=t_{c_1}$ is represented as $g_{e}$, where $e$ is 
any root in $\L=4A_1$.
\end{itemize}
\end{lemma}

\begin{proof}
The 
specified
expressions for $\delta$ through the generators of $\Maps(X_\R)$ follow directly from Proposition 
\ref{maps-spherical}.
In the case $p=3$ we apply this expression, use Lemma \ref{MW-basic} and  Proposition \ref{Delta-lemma}(2),
and get
$$\begin{aligned}
g_e=g_{B_1}g_{O_1}g_{B_{12}}g_{O_2}g_{B_{23}}g_{O_3}g_{B_3}=
&s_1(\D_1s_1^{-2})(s_1s_2t^{-1}_{c_2})(\D_2s_2^{-2})(s_2s_3t^{-1}_{c_3})(\D_3s_3^{-2})s_3=\\
&\D_1\D_2\D_3t^{-1}_{c_2}t^{-1}_{c_3}=\D_1\D_2\D_3\D_2^{-2}=\delta.
\qquad\qquad\qquad\end{aligned}
$$
In the case $p=2$, we get similarly
$$\begin{aligned}
g_e=&g_{B_1}g_{O_1}g_{B_{12}}g_{O_2}g_{B_{2}}=
s_1(\D_1s_1^{-2})(s_1s_2t^{-1}_{c_2})(\D_2s_2^{-2})s_2=\D_1\D_2t^{-1}_{c_2}=\\
&\D_1\D_2^{-1}(t_{c_2}t_{c_3})t^{-1}_{c_2}=\D_1\D_2^{-1}t^{-1}_{c_1}=\d.
\end{aligned}$$
In the case $p=1$, we get
$$g_{e}=g_{B_1}g_{O_1}g_{B_{1}'}=s_1(\D_1s_1^{-2})s_1=\D_1=\d.
$$
For $g_e=t_{c_1}$ in the case $p=0$ see Lemma \ref{MW-basic}(4).
\end{proof}

\begin{lemma}\label{primitivity}
Assume that $\LL\subset E_8$ is a
root lattice and 
$\LL'\subset \LL$ is
generated by some pairwise orthogonal roots $e_1,\dots,e_n$,  $n\le3$.
Then $\LL'$ is primitive in $\LL$.
\end{lemma}

\begin{proof}
If $n\le 3$, then  $\DD(\LL')=n[-\frac{1}2]$ has no isotropic element, which implies primitivity of the embedding $\LL'\subset\LL$.
\end{proof}

\begin{proof}[Proof of Theorem \ref{details-on-MW-homomorphism}] In the case $p=3$, we have $\ker\Phi\cong\Z$ (see Proposition \ref{Phi-p}) and the result follows from Lemma \ref{delta} combined with Lemma \ref{primitivity}.

In the case $p=2$, in addition to the 5-chain $e$ from Lemma \ref{delta},
we have $\delta=g_{e}=g_{e'}\in\Maps(X_\R)$ for another 5-chain $e'$ (a subgraph of $\L=D_6$) obtained by replacing $B_2$  by $B_2'$.
It does give the same element of $\Maps(X_\R)$, since, according to Lemma \ref{MW-basic}(3), $g_{B_2}=g_{B_2'}$.
Thus, the combinations $a_1e+a_2e'$ 
with $a_1+a_2\in2\Z$
give a subgroup $\Z^2$ of $\ker\Phi\cong\Z^2$ (see Proposition \ref{Phi-p}) 
and applying Lemma \ref{primitivity} we conclude that the kernel
should coincide with this subgroup.

In the case $p=1$, by Lemma \ref{delta} we have $\delta=g_{e}$, where $e$ is presented by a 3-chain
 $B_1+O_1+B_1'$ (a subgraph of $\L=D_4+A_1$).
For the same reasons as in the case 
$p=2$,
the element  $\delta$ is presented by 
two
other 3-chains in the summand $D_4$ of $\L$.
From here, the sublattice $\mathcal K\subset \L$ formed by integer combinations of these three 3-chains
with coefficients $a_1,a_2,a_3$  satisfying $a_1+a_2+a_3\in 2\Z$ is contained in
the $\ker\Phi$. Since the rank, $3$, of $\mathcal K$ is the same as that of $\ker\Phi$ (see Proposition \ref{Phi-p})
and the lattice of all integer combinations of these three 3-chains is primitive
due to Lemma \ref{primitivity}, we conclude $\mathcal K=\ker\Phi$.

In the case $p=q=0$, the proof follows the same lines, using Lemma \ref{delta} and Proposition \ref{Phi-p}
(in this case $\L=4A_1$ and the primitivity argument becomes trivial).

The case $X_\R=\K\+q\SSS^2$, $0<q<4$, is analogous to the case 
$p=q=0$ and differs only in the rank of $\L=(4-q)A_1$ and 
usage of Proposition \ref{spherical} instead of Proposition \ref{Phi-p}.

In the cases $X_\R= (\K\#\T^2)\+\SSS^2$ and $X_\R= \K\+\K$ we 
make use of Propositions \ref{T-and-S}, \ref{2K-case} and provide matrices of the homomorphism $\L\to\Maps(X_\R)$
(see Tab. \ref{Phi-matrices}),
which are calculated 
using Lemmas \ref{MW-basic} and \ref{MW-Klein2}.
The kernels 
claimed in
Theorem \ref{details-on-MW-homomorphism} are then found from these matrices.
\end{proof}

\begin{table}[h!]
\caption{Matrices of  $\Z^4=\L\to\Maps(X_\R)$, $v\mapsto g_v$
}\label{Phi-matrices}
\hskip-5mm\scalebox{0.85}{(1) $\L=D_4\to \Maps(\K\#\T^2\+\S^2)=\Z^2+\Z/2$}\hskip10mm
 \scalebox{0.85}{(2) $\L=D_4\to \Maps(\K\+\K)=\Z/2+\Z/2$}
\vskip2mm

\hskip-4mm\begin{tabular}{c|ccccc}
&$O_1$&$B_1$&$B_1'$&$B_1''$\\
\hline
$s_1$&-2&1&1&1\\
$\D_1$&1&0&0&0\\
$t_{c_1}$&0&0&0&0
\end{tabular}
\hskip12mm
\begin{tabular}{c|cccc}
&$B_0$&$B_1$&$B_2$&$B_3$\\
\hline
$\gg{\ll_1^1}{\ll_2^1}$&$*$&1&1&1\\
$\gg{\ll_1^1}{\ll_1^2}$&1&0&0&0\\
\end{tabular}
\vskip1.5mm
\caption*{\footnotesize In the second matrix $*$ stands for $0$ or $1$ depending on orientations chosen for $B_1,B_2,B_3,B_4$.
By $\gg{\ll_1^1}{\ll_2^1}\in\Maps(X_\R)$ we denote an element 
preserving the  components of $\K\+\K$,
while $\gg{\ll_1^1}{\ll_1^2}$ denotes an element 
which interchanges them (see Lemma \ref{MW-Klein2}).}
\end{table}

\section{Proof of Theorems \ref{line-number} and \ref{vanishing-classes}
}\label{SectProofs}
Here, we follow the setting and notation of Sections \ref{sect5}, \ref{sect6}. In particular, we fix 
a real elliptic surface $f:X\to \P^1$ satisfying the assumption $\A$ and a real line $L\subset X$.

\subsection{Proof of Theorem \ref{line-number}}
The possible topological types of $X_\R$
are listed in Tab.\,\ref{C-Y-X-correspondence}, see Theorem \ref{11-X-classes}.

If $X_\R=\K\#p\T^2$ with $p\ge 1$, 
then $s_1^n\in
\Phi(\MW_\R)$ for any $n\in\Z$ 
(see Proposition \ref{Phi-p} if $1\le p\le3$ and Proposition \ref{index2}  if $p=4$).
 On the other hand, $g(L)\subset X$ is 
 a real line, for any $g\in\MW_\R$. Since 
 $L_{\R}$ intersects the fiber $a_1\cup b_1$ at a point of $a_1$, the homology class 
$[s_1^n(L_{\R})]\in H_1(X_\R)$ is $[L_{\R}]+n[a_1]$ and we obtain $\Nlin=\infty$, since $[a_1]\in H_1(X_\R)$ has infinite order.

If $X_\R=\K\#\T^2\+\SSS^2$, the arguments are literally the same, except that we refer to Proposition \ref{T-and-S}(1) to justify that $s_1^n\in\Phi(\MW_\R)$.

If $X_\R=\K$, then there exist only two classes in $H_1(X_\R)$ realizable by 
smooth
sections: $[L_{\R}]$ and $[L_{\R}]+[c_1]$, where $[c_1]$
is the order 2 element of $H_1(X_\R)$.
The class $[L_{\R}]+[c_1]$ is realized by the line $t_{c_1}(L_\R)$ 
(see Proposition \ref{Phi-p} applied to $p=0$), so $\Nlin=2$.

If $X_\R=\K\+q\SSS^2$ with $0<q\le3$,  we refer to Proposition \ref{spherical}, and the same arguments as for $X_\R=\K$
give $\Nlin=2$.

If $X_\R=\K\+4\SSS^2$, then $\L=0$, which implies $\Nlin=1$ (see Proposition \ref{line-root-correspondence}).

If $X_\R=\K\+\K$, then $H_1$ of each component contains only
2 classes realizable by 
smooth
sections. Finally, Proposition \ref{2K-case}
and the same arguments as for $X_\R=\K$ imply that all 4 are realizable by real lines, which gives $\Nlin=4$.
\qed

\subsection{Proof of Theorem \ref{vanishing-classes}}\label{proof-infinity-vanishing}
\begin{lemma}\label{line-through-node} Let $C\subset Q$ be a real nonsingular sextic $C\subset Q$ of type $\la p\,|\,q\ra$ with $p\ge 1$. Then, 
for any of the 
$p$ positive
ovals $o_i, 1\le i\le p$, of $C_\R$,
there exists a degeneration of $C$ to a real 1-nodal sextic $C_0$ contracting $o_i$  to a solitary real point. Furthermore, 
for any such degeneration there exists a real hyperplane section passing through
the node of $C_0$ and tangent to $C_0$ at two other (not necessarily real) points.
\end{lemma}
\begin{proof}
The first part of the statement is a straightforward consequence of existence of deformations contracting simultaneously
 all the ovals, see \cite[Lemma 2.3.4]{TwoKinds}.

To prove the second part, recall that for any point $x\in Q_\R$, the projection from $x$ establishes a real birational isomorphism
$\pi : Q\dashrightarrow \P^2$. This isomorphism can be decomposed into the blowup at $x$ followed 
by the blowup of the vertex of $Q$, then contracting the proper image  of the generatrix $h_x$ passing through $x$, and finally, contracting the proper  image of the exceptional divisor of the second blowup. The 
exceptional divisor $E_x$ of the first blowup is mapped by $\pi$ to a
real  line $\pi(E_x)\subset \P^2$ passing through the point
$y=\pi( h_x)$.  Respectively, $\pi^{-1} : \P^2\dashrightarrow Q$ consists in blowing up the real point $y\in \P^2$ followed by the blowup of the intersection point between the exceptional divisor $I_y$ 
and the proper image of the real line $\pi(E_x)$, and 
finally, contracting the proper image of $I_y$.

If $x$ coincides with the node of a 1-nodal sextic $C_0\subset Q$, then the image of 
$C_0$ is a nonsingular quartic $A\subset \P^2$, and $\pi(E_x)$ is its tangent at the point $y=\pi(h_x)\in A$. 
Due to the above description of $\pi^{-1}$, the
requested hyperplane section is provided by lifting of 
any of the real double tangents to $A$. The existence of the latter ones is well known.
\end{proof}

\begin{proof}[Proof of Theorem \ref{vanishing-classes}]
Due to the stability of  real vanishing classes under deformation,
and to the deformation classification of real
nonsingular relatively minimal rational elliptic surfaces containing a real line, it is sufficient to prove the statement using a single example. So, we pick a real nonsingular sextic, observe that each of its positive ovals represents a vanishing class on
the associated real elliptic surface, and check that the  $MW_\R$-orbit of such a vanishing class
consists of  an
infinite number of elements.

More precisely, following Lemma  \ref{line-through-node} we consider a generic,   invariant under complex conjugation, one-parameter 
complex-analytic family $C_\tau, 0\le \vert \tau\vert\le \e, \tau\in \C,$ of sextics such that $C_\tau$ with $\tau\in\R, \tau>0,$
are real nonsingular sextics of type 
$\la p\,|\,q\ra$ while $C_0$ is a real 1-nodal sextic with a solitary node obtained from contracting a positive oval  $o_i\subset C=C_\tau, \tau>0$.
The
associated complex analytic
family $X'_\tau$ of rational elliptic surfaces inherits a complex conjugation such that $X'_\tau$ is real with $(X'_\tau)_\R=\K\# p\T^2\+ q\SSS^2$ for $\tau\in \R, \tau>0$. 
Then, by a base change $\tau=t^2$ followed by Atiyah's smoothing construction (see \cite{Atiyah}), we obtain a smooth complex analytic family of
surfaces, $X_t$, such that $X_t=X'_{t^2}$ for $t\ne 0$ while $X_0$ is the minimal resolution of $X'_0$. Furthermore, since the nodal degeneration $X'_{\tau>0}\to X'_0$ is contracting a circle $o_i\subset (X'_\tau)_\R$ (case of signature 1 in terminology of \cite{IKS}), the real structure on $X_0$ lifted from $X'_0$
and that real structure on $\{X_{t\ne 0}\}$ lifted from the real structure on $\{X'_{\tau\ne 0}\}$ for which $(X_t,\conj)=(X'_{t^2}, \conj)$, they fit
together and define a real structure on the total space of the Atiyah family $\{X_{t}\}$ (see \cite{IKS} for details). In particular, this shows that $[o_i]\in H_1((X_{\sqrt{\tau}})_\R)$ is a real vanishing class for any choice of orientation on $o_i$.

Next, due to stability of $(-1)$-curves (see \cite{Kodaira}), 
any of two real lines
$L'\subset X_0$ 
covering the hyperplane section 
provided by Lemma \ref{line-through-node}
extends, at least for small values of $t \in \C$, to an analytic family of 
lines $L'_t\subset X_t$. Due to the uniqueness  of this extension, and since $L'_0=L'$ is real, the family $\{L'_t\}$ is also real, so that, for each small real $t$ 
the line $L'_t$ is also real.
Having also a real family
of zero sections $L_t\subset X_t$, we may reparametrize the family $X_t$ via $g^n_t\in \MW(X_t)$, $n\in \Z$, $g_t=\gg{L_t}{L'_t}$,
and thus deduce that
$g^n_\e(o_i)$ is a real vanishing cycle for any $n\in \Z$.

The intersection index of $L'=L'_0$, and hence of $L'_t$ for any $t$,  with the vanishing class $[O_i]\in H_2((X_{\sqrt{\tau}})_\R)=H_2((X'_{\tau})_\R))$ is equal to $\pm 1$. Thus, applying Corollary \ref{tire-bouchon}
and Proposition \ref{L-and-g}, 
we conclude that $\Phi(g_\e)\vert_{N_i}$ is equal either to $(s_i)^{\pm 1}$ or $(\Delta_i s_i^{-1})^{\pm 1}$.
Therefore, $\Phi(g^n_\e\vert_{N_i})(o_i)$ reduces to iteration of Dehn twists and, as a result, is equal to $\pm(o_i\pm n(a_i-b_i))$
(with respect to orientations shown on Fig. \ref{ab-orientation}), 
which gives us an infinite number of 
real vanishing classes in $H_1((X_\e)_\R)=H_1((X'_{\tau>0})_\R) $.
\end{proof}

\begin{figure}[h!]
\caption{Preferred orientations}\label{ab-orientation}
\includegraphics[height=2.5cm]{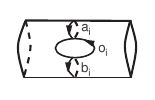}
\end{figure}

\section{Realizability of a smooth section by a line.}\label{realizability}
\subsection{The case without obstruction}
\begin{theorem}\label{lines-vs-sections}
If $X$ satisfies assumption A, and $X_\R$ is different from
$\K\#4\T^2$, $\K\#\T^2\+\SSS^2$, and $\K\+4\SSS^2$,
a class in $H_1(X_\R)$ is realized by a smooth section $\P^1_\R\to X_\R$ if and only if it is realized by a real line.
\end{theorem}

\begin{proof}
According to Theorem \ref{MW-homomorphism} the homomorphism $\Phi: \MW_\R(X)\to \Maps(X_\R)$ is epimorphic in the indicated cases. The orbit of the class $[L_\R]\in H_1(X_\R)$ 
of an arbitrary chosen real line $L$
under the induced $\Im(\Phi)$-action in $H_1(X_\R)$ is formed by the classes realizable by
real lines. 
The epimorphism $\pi_0(\Sec(F,\l_0))\to\Map^s(F)$
(see Lemma \ref{section-action}) implies that
a similar orbit of the whole group $\Maps(X_\R)$ by Proposition \ref{no-zigzag-isomorphism}
is formed by the classes realized by smooth real section.
\end{proof}
\subsection{Criterion for a smooth section to be realizable by a line}\label{criterion}
The fixed line $L\subset X$ determines (as any other line on $X$) a fiberwise involution $\b_L:X\to X$
that preserves $L$. It is the lift of the Bertini involution $\beta : Y\to Y$,
and its fixed point set is $L\cup C$ where we 
identified the sextic $C\subset Q$ and its lift to $X$.

Let $X_\R=\Kl\#4\T^2$.
Then $C$ considered as a sextic in $Q$ is of type $\la 4\,|\,0\ra$,
and we
numerate the ovals of $C_\R$ 
so that $o_1, o_3$ are the lower ovals and $o_2, o_4$ are the upper ones (see Subsection \ref{lower-upper-section} and Proposition \ref{lower-upper}).

Considering a generic real smooth section $\l:\P^1_\R\to X_\R$ of $f$ 
we let $l=\l(\P_\R^1)$ and define
 subsets $S_{in},S_{tan}\subset\{1,2,3,4\}$ associated with $\l$, extending our previous definition given in the case $L_\R\cap l=\varnothing$
 (as it was given in the context of 
 real tritangents to $C\subset Q$, see Section \ref{Sin-Stan}).
Namely, we observe that $l\cup \b_L(l)$ divides $X_\R$ into two 
closed regions with $l\cup 
\b(l)$ as a common boundary,
\begin{figure}[h!]
\caption{
Region $F_l$ is shaded}\label{Fl}
\includegraphics[height=2.5cm]{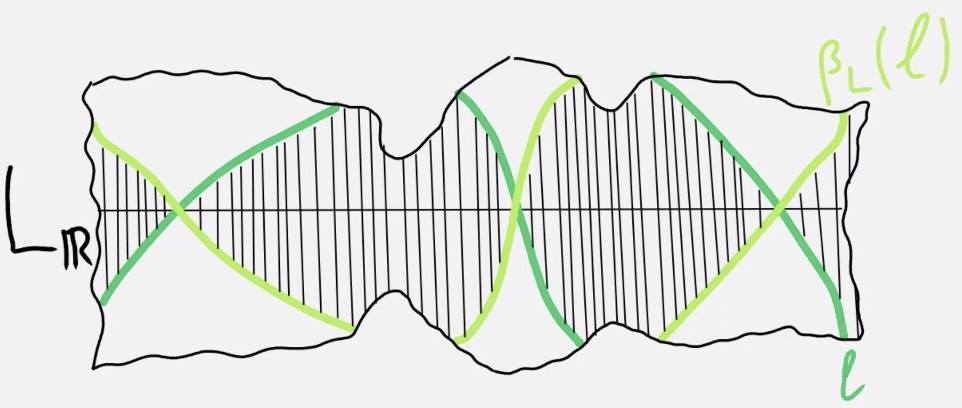}
\end{figure}
denote by $F_l\subset X_\R$ the 
region containing $L_\R$
(see Fig. \ref{Fl})
and set
$$S_{tan}(l)=\{i\,|\,o_i\circ l=1\!\mod2\}\quad\text{and}\quad
S_{in}(l)=\{i\,|\,o_i\subset F_l\}.
$$ 
When working with the sets $S_{tan}(l), S_{in}(l)$ in concrete situations, we descend from $X$ to $Y$ and apply
the terminology and encoding introduced in Sec. \ref{encoding}.

\begin{lemma}\label{simultaneous}
 The residue $r=|S_{in}(l)\sm S_{tan}(l)|+|S_{tan}(l)\cap\{1,3\}|+l\circ L_\R\!\mod2$ is preserved under replacement
$l$ by $l'=g(l)$ with $g\in \Map^s(X_\R)$ if and only if $g$ belongs to $\Phi(\MW)$.
\end{lemma}

\begin{proof}
It is enough to check that for the generators of $\Phi(\MW)$, 
$g\in\{s_1,s_3, s_2^2,s_4^2$, $\D_i \,|\,  i=1,\dots, 4\}$ (see Proposition \ref{index2}(2)),
the residue $r$ does not change, while for 
$g=t_{c_i}\in \Map^s(X_\R)$
which represents the generator
of $\Map^s(X_\R)/\Phi(\MW)=\Z/2$
(see Proposition \ref{index2}(3)), the residue $r$ changes.

Table \ref{change-t-s-D} shows 
how $r$ and each of its summands, $|S_{in}(l)\sm S_{tan}(l)|$, $|S_{tan}(l)\cap\{1,3\}|$, $l\circ L_\R\!\mod2$, vary
under the action of $t_{c_i}$, $s_i$ and $\D_i$.

For $g=t_{c_i}^{\pm1}$, the sets $S_{in}$ and $S_{tan}$ are not affected, while $l\circ L_\R$ changes by $1$, since
the classes $[l'],[l]\in H_1(X_\R)$ differ by the fiber-class.
The action of $s_i$ alternates the pairity of 
$l\circ o_i$, and, in particular, changes
 $|S_{in}\sm S_{tan}|$ if $l$ is underpassing oval $o_i$,
while the intersection index $l\circ L_\R$ alternates
only if $l$ is overpassing $o_i$.
The action of $\D_i$ does not affect 
$l\circ o_i$, while alternates
"overpasses" and "underpasses" of $l$ over $o_i$.
Fig. \ref{Delta-action} 
shows how the action of $\D_i$ affects $l\circ L_\R$.
\end{proof}

\begin{table}[h!]
\caption{Variation of $r$ and its summands}
\label{change-t-s-D}
\begin{tabular}{|c|c|c|c|c|c|}
\hline
$g$&position of $o_i$ and $l$&$|S_{in}\sm S_{tan}|$&$|S_{tan}\cap\{1,3\}|$&$l\circ L_\R\!\mod 2$&$r$\\
\hline
$t_{c_i}^{\pm1}$& in all positions&0&0&$1$&$1$\\
\hline
$s_1^{\pm1}$ or $s_3^{\pm1}$&\bottan\ or \bott&0&$\pm1$&$1$&0\\
$s_1^{\pm1}$ or $s_3^{\pm1}$&\uptan\ or \upp&$\pm1$&$\pm1$&0&0\\
$s_2^{\pm1}$ or $s_4^{\pm1}$&\bottan\ or \bott&0&0&$1$&$1$\\
$s_2^{\pm1}$ or $s_4^{\pm1}$&\uptan\ or \upp&$\pm1$&0&0&$1$\\
\hline
$\D_i^{\pm1}$&\bottan&0&0&0&0\\
$\D_i^{\pm1}$&\bott&$\pm1$&0&$1$&0\\
$\D_i^{\pm1}$&\uptan&0&0&0&0\\
$\D_i^{\pm1}$&\upp&$\pm1$&0&$1$&0\\
\hline
\end{tabular}
\end{table}

\begin{figure}[h!]
\caption{ }\label{Delta-action}
\includegraphics[height=3.5cm]{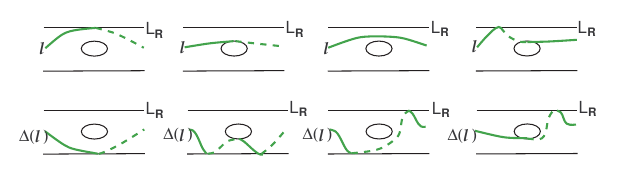}
\end{figure}

\begin{theorem}\label{section-criterion-X}
If $X$ satisfying the assumption $\A$ is
 endowed with a fixed real line $L$ and has $X_\R=\Kl\#4\T^2$, then 
a smooth section 
$l\subset X_\R$ is ambient isotopic to the real locus of a real line if and only if the sets
$S_{in}(l),S_{tan}(l)\subset\{1,\dots,4\}$ defined by $l$ satisfy
$$|S_{in}(l)\sm S_{tan}(l)|+|S_{tan}(l)\cap\{1,3\}|=l\circ L_\R+1\!\mod2.$$
\end{theorem}

\begin{proof}
Due to Proposition \ref{arithm},
the statement holds for 
smooth
sections $l$ represented by real lines disjoint from $L$. 
As it follows from Lemma \ref{section-action},
the mapping $g\in \Map^s(X_\R)\mapsto l=g(L_\R)$ establishes a bijection between $\Map^s(X_\R)$ and the set of ambient isotopy classes of smooth sections $l\subset X_\R$, and restricts
to a bijection between $\Phi(\MW)$ and the ambient isotopy classes of 
smooth sections represented by real lines.\
Thus, the claim follows from Lemma \ref{simultaneous}.
\end{proof}

The next theorem concerns 
the case $X_\R=\Kl\#\T^2\+\SSS^2$.
In this case the curve $C_\R\subset Q_\R$ has two ovals and we denote by $o$ the positive one.
On this oval the projection $f_\R:X_\R\to\P^1_\R$ has 
two critical points, $x,y\in o$, that we connect by a curve $\g\subset X_\R$ (see the leftmost sketch on Fig. \ref{1-1-obstruction})
with the following properties:
\begin{itemize}\item
$\g$ is a section of $f_\R$ over the interval $\P^1_\R\sm \operatorname{Int} (f_\R(o))$ bounded by the critical values $f_\R(x)$, $f_\R(y)$. 
\item
$\g$ intersects
neither $L_\R$ 
nor the J-component of $C_\R$.
\end{itemize}
We choose arbitrarily a coorientation of $\g$, and note that the intersection index $\g\circ l\in\Z$
is well-defined for any smooth section $l\subset X_\R$. Clearly, this index is preserved under continuous variations of $l$ in the space of smooth
sections.

Note that for the Bertini-partner $\g'=\b_L(\g)$ of $\g$ the same properties are satisfied and $\g\cup\g'$ form a simple closed curve
since $\g\cap\g'=\{x,y\}$. Moreover, $\b_L$ induces a coorientation of $\g'$ compatible with that of $\g$, so that $\g\cup\g'$ becomes a cooriented closed curve.
Since, 
in addition,
$\g\cup\g'$ bounds in $X_\R$, this implies that $\g\circ l=-\g'\circ l$ for any smooth section $l$.

\begin{figure}[h!]
\caption{
}\label{1-1-obstruction}
\includegraphics[height=3cm]{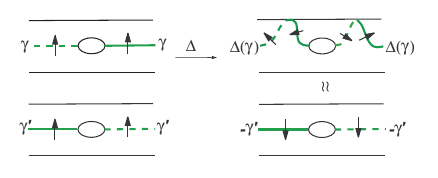}
\end{figure}

\begin{theorem}\label{section-criterion-X1|1}
If $X$ satisfying the assumption $\A$ is endowed with a fixed real line $L$ and has $X_\R=\Kl\#\T^2\+\SSS^2$, then 
a smooth section 
$l\subset X_\R$ is 
ambient
isotopic to the real locus of a real line if and only if $\g\circ l=0$.
\end{theorem}

\begin{proof}
Consider $g\in G(X_\R)$ such that $g(l)=L_\R$.
By Proposition  \ref{maps-spherical}, the image of $g$ in $\Maps(X_\R)$ can be written as image of
$\Delta_1^{\k}s_1^{n}t_{c_1}^{m}$ where $n,m\in \Z$, $\k\in\{0,1\}$, and $\Delta_1, s_1, t_{c_1}$ stand for
standard representatives of the generators of $\Maps(X_\R)$ specified in Proposition \ref{maps-spherical}.
Clearly, $s_1$ leaves $\g$ invariant, while $\D_1$ sends $\g$ to a curve isotopic to $\g'=\beta_L(\g)$
with the opposite coorientation (see Fig.\,\ref{1-1-obstruction}). Therefore,
$$
\g\circ l=g(\g)\circ L_\R=\begin{cases}(\gamma+mc_1)\circ L_\R=m,\quad &\text{if}\,\, \k=0,\\
(-\gamma'+mc_1)\circ L_\R=m,\quad &\text{if}\,\, \k=1.
\end{cases}
$$
On the other hand, according to Proposition\,\ref{T-and-S}  the mapping class of $\Delta_1^{\k}s_1^{n}t_{c_1}^{m}$ belongs to $\Phi(\MW)$ if and only if $m=0$.
Thus, there remains to notice that due to Lemma \ref{section-action} $l$ is ambient isotopic to a real line, if and only $ = g^{-1}=(\Delta_1^{\k}s_1^{n}t_{c_1}^{m})^{-1}$ belongs to $\Phi(\MW)$.
\end{proof}

\begin{remark}
The conclusion of Theorem \ref{section-criterion-X1|1} 
becomes stronger if we add the assumption
that $X_\R\to \P^1_\R$ has no zigzags.
In this case, the same proof, but with Lemma  \ref{section-action} replaced  by Proposition \ref{no-zigzag-isomorphism}, shows that 
$\gamma\circ l =0$ if and only if $l$ is isotopic to a real line through smooth sections (not only ambient isotopic as in Theorem \ref{section-criterion-X1|1}).
\end{remark}

\section{The action of $\MW_\R(X)$ in $H_1(X_\R)$}\label{sect8}

\subsection{Decomposition of $H_1(X_\R)$}\label{decompH1} Let us fix a real line 
$L_{\R}\subset X_\R=\K^2\#p\T^2\+q\SSS^2$ and choose a connected fiber $F_\R\subset X_\R$. Then, as in Section \ref{cut-section}, consider
the subsurfaces $N_i$ of the non-spherical component $\K^2\#p\T^2$ of $X_\R$ and their skeletons $a_i\cup o_i\cup b_i, 1\le i\le p$, 
see Fig. \ref{cutting},
where we assume in addition that $o_i$ are positive ovals of the sextic $C$ defined by $L$ (see Sec.\,\ref{criterion}) and $b_i$ are disjoint from $L_\R$.
We orient $L_{\R}$ in accord with the orientation of $\P^1_\R$. To orient the circles $o_i$ we notice that each of them splits into a pair of arcs connecting critical points of the projection $f_\R:X_\R\to\P^1_\R$: the {\it upper arc} intersecting $a_i$ and the {\it lower arc} intersecting $b_i$. We orient $o_i$ so that $f_\R$ preserves
the orientation on the lower arc (and, thus, reverses on the upper one). For $a_i, b_i$,  we fix an orientation of $X_\R\sm F_\R$ and then orient $a_i, b_i$ in a way to obtain the following local intersection indices
\be\label{local-indices} o_i\circ a_i=b_i\circ o_i=1,\quad \text{from where } a_i\circ [L_{\R}]=1.\ee
Finally, we notice that $[F_\R]$ 
is 
the 2-torsion element of $H_1(X_\R)$ and forms together with $[L_\R], o_i, b_i$ a basis of $H_1(X_\R)$. This leads to a natural decomposition
\be\label{decomposition-formula}H_1(X_\R)=\la [F_\R]\ra\oplus\Big\lbrack\bigoplus_{i=1}^p\la b_i,o_i\ra\Big\rbrack\oplus\la [L_{\R}]\ra=
\Z/2\oplus\Big\lbrack\bigoplus_{i=1}^p(\Z\oplus\Z)\Big\rbrack\oplus\Z.\ee
With respect to this basis,
the class of any 
smooth
section has a 
coordinate expression
\be\label{in-coordinates}
\kk [F_\R]+\sum_{i=1}^p(m_ib_i+\k_i o_i)+ [L_\R],\quad \kk\in\Z/2, \ m_i\in\Z,\  \k_i\in\{0,1\}.
\ee

The proof of the following is straightforward.

\begin{proposition}
For a fixed real line $L$, 
the identification $\la b_i,o_i\ra=\Z\oplus\Z$
may be changed only by an automorphism $n\mapsto -n$ of the first $\Z$-summand,
which happens
if the 
orientation of $N_i$
is changed by another
choice of $F_\R$ or/and another choice of a fixed orientation of $X_\R\sm F_\R$.
\qed\end{proposition}

\subsection{Matrix description of the action of $\Maps(X_\R)$ in $H_1(X_\R)$}
Here, in addition to $t_{c_i}, s_i, \D_i\in \Maps(X_\R)$ we consider auxilliary elements  $\bar\D_i=\D_i t_{c_{i+1}}^{-1}$.
If $p=1$ we use notation $a,b,c,o,s,\D, \bar \D,\dots$ instead of $a_1,b_1,c_1,o_1,s_1,\D_1, \bar \D_1,\dots$.

\begin{lemma}\label{block-matrix-action}
If $X_\R=\K\#\T^2$, then the matrices of the action of $\D, \bar \D, t_{c}, s\in\Maps(X_\R)$ in $H_1(X_\R)$ with respect to the decomposition
 $$H_1(X_\R) = \la [F_\R]\ra\oplus\la b\ra\oplus\la o\ra\oplus\la [L_\R]\ra\cong\Z/2\oplus\Z\oplus\Z\oplus\Z$$
  are as follows {\rm(}integers in brackets 
  stand for their $\Z/2$-residues{\rm):}

\hskip-4mm\scalebox{0.77}{$M_\D=\begin{bmatrix*}[r]
[1]&[1]&0&[1]\\
0&-1&0&0\\
0&0&-1&1\\
0&0&0&1
\end{bmatrix*}$}
\hskip1mm
\scalebox{0.77}{$M_{\bar\D}=\begin{bmatrix*}[r]
[1]&[1]&0&0\\
0&-1&0&0\\
0&0&-1&1\\
0&0&0&1
\end{bmatrix*}$}
\hskip1mm
\scalebox{0.77}{$M_{t_{c}}=\begin{bmatrix*}[c]
[1]&0&0&[1]\\
0&1&0&0\\
0&0&1&0\\
0&0&0&1
\end{bmatrix*}$}
\hskip1mm
\scalebox{0.77}{$M_s=
\begin{bmatrix*}[r]
[1]&0&[1]&[1]\\
0&1&2&-1\\
0&0&1&0\\
0&0&0&1
\end{bmatrix*}
$}
\end{lemma}

\begin{proof}
The first column of all matrices is $1,0,0,0$ because the order 2 element is invariant. The Dehn twist
$t_{c}$ acts trivially on $\la b,o\ra$ and sends the homology class $[L_\R]$ to $[L_\R]+[F_\R]$ which gives $M_{t_{c}}$.
To obtains $M_\D$ we notice that $\D$ sends $o$ to $-o$ and $b$ to $a=[F_\R]-b$, whereas $[L_\R]$ is sent to $[L_\R]+o+[F_\R]$.
The product $M_\D M_{t_{c}}^{-1}=M_\D M_{t_{c}}$ is the matrix of $\bar \D$.

To obtain $M_s$ we notice that $s$ preserves the class $b$, sends $[L_\R]$ to $[L_\R]+a=[L_\R]-b+[F_\R]$ and
$o$ to $t_at_b(o)=t_a(o+b)=o+b-a=o+2b-[F_\R]$, since our choice of orientations gives
$b\cdot o=-a\cdot o=1$, $a+b=[F_\R]$.
\end{proof}

\begin{cor}\label{matrices-Fbo}
\scalebox{0.80}{$M_{s}^mM_{t_c}^n=
\begin{bmatrix*}[r]
[1]&0&[m]&[m+n]\\
0&1&2m&-m\\
0&0&1&0\\
0&0&0&1
\end{bmatrix*}
%
\,\text{and}\,\,
M_{\bar\D}M_{s}^mM_{t_c}^n=
\begin{bmatrix*}[r]
[1]&[1]&[3m]&[n]\\
0&-1&-2m&m\\
0&0&-1&1\\
0&0&0&1
\end{bmatrix*}.
$}
\newline
In particular, for any $\k\in\{0,1\}$, $n,m\in\Z$, and $g=\bar\D^\k s^mt_c^{n}$, we have
$$[g(L)_\R]=[L_\R]+\kk [F_\R]+(-1)^{1-\k}mb+\k o,\ \ \kk=[n+(1-\k)m]\in\Z/2. \qed$$
\end{cor}

\subsection{A special decomposition in $\Maps(X_\R)$}

\begin{proposition}\label{g-action}
If $X_\R\cong\K\#p\T^2\+q\SSS^2$, then every
element $g\in\Maps(X_\R)$ can be presented in a form
\be\label{geometric-presentation}g=\bar\D_1^{\k_1}\dots\bar\D_p^{\k_p}t_{c_1}^{n_1}\dots t_{c_p}^{n_p} s_1^{m_1}\dots s_p^{m_p},\quad
\k_i\in\{0,1\}, n_i,m_i\in\Z,
\ee
and such presentation is unique.

With respect to this presentation,
the class of $g(L)_\R$ in $H_1(X_\R)$ is
$$
[g(L)_\R]=[L_\R]+\kk [F_\R]+(-1)^{1-\k_1}m_1b_1+\k_1o_1+\dots+(-1)^{1-\k_p}m_pb_p+\k_po_p 
$$
where $\kk=n_1+\dots+n_p+m_1(1-\k_1)+\dots+m_p(1-\k_p)\, \,{\rm mod\,2}.$
\end{proposition}

\begin{proof}
Proposition \ref{maps-spherical} implies
 that
$\bar \D_i$, $s_i$ and $t_{c_1}$ generate $\Maps(X_\R)$ with only one relation $\bar\D_1^2\dots\bar\D_p^2=t_{c_1}^2$.
Then a presentation $g=t_{c_1}^n\prod_{i=1}^p(\bar\D_i^{k_i}s_i^{m_i})$ is transformed to the form (\ref{geometric-presentation})
using the relations $\bar\D_i^2=t_{c_{i+1}}t_{c_i}^{-1}$ 
and $t_{c_{p+1}}=t_{c_1}^{-1}$.

Since
the relations in $\Maps(X_\R)$ involve only even powers of $\bar \D_i$, an equality
\[
\bar\D_1^{\k_1}\dots\bar\D_p^{\k_p}t_{c_1}^{n_1}\dots t_{c_p}^{n_p} s_1^{m_1}\dots s_p^{m_p}=
\bar\D_1^{\k_1'}\dots\bar\D_p^{\k_p'}t_{c_1}^{n_1'}\dots t_{c_p}^{n_p'} s_1^{m_1'}\dots s_p^{m_p'}
\]
may hold only if
$\k_i=\k_i'$
for each $1\le i\le p$.
So, to prove uniqueness
of presentation in the form  (\ref{geometric-presentation}), it is left to notice that 
$t_{c_i},s_i$, $i=1,\dots,p$ generate a free abelian subgroup in $\Maps(X_\R)$, which follows from Lemma \ref{cuts}.

To evaluate the class $[g(L)_\R]\in H_1(X_\R)$
we
determine the contribution of each factor $\bar\D_i^{\k_i}s_i^{m_i}t_{c_i}^{n_i}$ precisely like
in Lemma \ref{block-matrix-action} and Corollary \ref{matrices-Fbo}.
\end{proof}

\subsection{Proof of Theorem \ref{real-action-matrix}}
By Proposition \ref{g-action}, each element $g\in\Maps(X_\R)$, and, in particular, such that $L'=g(L)$,
can be decomposed in the form
(\ref{geometric-presentation}). This
identifies the coordinate expression of $[L'_\R]$ with the last column of the matrix $M$. The first column of $M$ is determined by
the invariance of the $\Z/2$-generator, $[g(F)_\R]=[F_\R]$. 
The Dehn twists $t_{c_i}^{n_i}$ being supported in neighborhoods of the fibers $c_i$ act only on $[F_\R]\in H_1(X_\R)$, but not on
$b_i,o_i\in H_1(N_i)$.
The factor $\bar\D_j^{\k_j}s_j^{m_j}$ of $g$ 
acts identically on $b_i,o_i\in H_1(N_i)$, $j\ne i$, since the corresponding diffeomorphism is supported in $N_j$.
Thus, the action of $g$ on $b_i,o_i\in H_1(N_i)\subset H_1(X_\R)$ is reduced to the action of
$\bar\D_i^{\k_i}s_i^{m_i}$, and its calculation is literally the same as
in Lemma \ref{block-matrix-action} and Corollary \ref{matrices-Fbo}.
\qed

\subsection{Proof of Theorem \ref{MW-sum-general}}
Immediate from multiplication of the matrix of $g$ as given in Theorem \ref{real-action-matrix} by the column of the coordinates of $[L''_\R]$,
and an observation
that $(-1)^{\k_{1i}}m_{2i} - 2m_{1i}\k_{2i}+m_{1i}= (-1)^{\k_{1i}}m_{2i}+(-1)^{\k_{2i}}m_{1i}$.
\qed

\begin{remark}
Theorem \ref{MW-sum-general} gives a simple description of the group operation  induced from
$\MW_\R$ on the set $\Hl\subset H_1(X_\R)/Tors$ of classes realized by real lines.
Namely, 
for $X_\R=\Kl\#p\T^2\+q\SSS^2$, 
this set is contained in $L_\R+\Big\lbrack\oplus_{i=1}^p(\Z b_i+\{0,1\}o_i)\Big\rbrack$,
the group operation 
on the direct sum $\oplus_{i=1}^p(\Z b_i+\{0,1\}o_i)$ is component-wise, and 
on each of the summands it turns into multiplication of triangular matrices
$
\pm\begin{bmatrix}
1 &n\\
0&1
\end{bmatrix}\in{\rm SL}(2,\Z)$
via an identification
$$m b_i\mapsto 
\begin{bmatrix}
1 &-m\\
0&1
\end{bmatrix},\quad
mb_i+o_i\mapsto \begin{bmatrix}
-1 &m\\
0&-1
\end{bmatrix}.$$

\end{remark}

\section{Concluding remarks}\label{concluding}

\subsection{Modulo 2 real MW-action}\label{MWwithZ2}
Fixing a line $L$ on a 
rational relatively minimal
complex 
elliptic surface $X$ without multiple fibers leads to a direct sum decomposition
$$H_2(X)=\la F\ra\oplus W_L\oplus\la L\ra\cong\Z\oplus E_8\oplus\Z,
$$
where $F$ stands for a fiber and $W_L=F^\perp\cap L^\perp\cong E_8$. The following proposition is well known
(for coordinate presentation of lines and notation $L_w$, see Prop. \ref{line-root-correspondence}(2)).

\begin{proposition}\label{Eichler-Siegel}
If $X$ has only one-nodal singular fibers, then
the automorphism in $H_2(X)$ induced by a $\MW$-transform
sending $L$ to $L_w=kF+w+L$, $k=\frac{w^2}2$, has a block-matrix presentation {\rm (}in the above derect-sum-decomposition{\rm )}
$$\begin{bmatrix}
1&w^*&k\\
0&I_V&w\\
0&0&1
\end{bmatrix},\qquad
\begin{bmatrix}m\\v\\n\end{bmatrix}
\mapsto 
\begin{bmatrix}m+v\cdot w+kn\\
v+nw\\n\end{bmatrix}\ 
\begin{matrix}m,n\in\Z,\ v\in W_L=E_8,\\
v\cdot w \text{ 
states for product in $E_8$}. 
\end{matrix}
$$
In terms of $D=L_w-L=kF+w$, this
action can be wriiten as
$$x\mapsto x+(Fx)D-((Dx)+\frac12D^2(Fx))F.$$
In particular, any other line, $L_{w'}=k'F+w'+L$, is sent to the line 
$$
L_{w+w'}=(k+w\cdot w'+k')F+(w+w')+L. \phantom{aaaaaaaaaa}\qed 
$$
\end{proposition}

In the real setting, we fix a real line $L$ and associate with it a decomposition
$$H_1(X_\R; \Z/2)=\la F_\R\ra\oplus \W_L\oplus\la  L_\R\ra=\Z/2\oplus\W_L\oplus\Z/2,\quad \W_L= F_\R^\perp\cap L_\R^\perp
$$
where we do not distinguish in notation the real loci $F_\R, L_\R$ and the classes realized by them in $H_1(X_\R; \Z/2)$.

\begin{proposition}\label{MW-Z2setting}
The automorphism in $H_1(X_\R;\Z/2)$
induced by a real $\MW$-transform
sending $L$ to $ L_w=k F+w+L, w\in W_L, k=\frac{w^2}2\in\Z$,
has a block-matrix form
$$\begin{bmatrix}
1&w^*&k\\
0&I_V&w\\
0&0&1
\end{bmatrix},\qquad  
\begin{bmatrix}\mu \\ v \\ \nu \end{bmatrix}
\mapsto 
\begin{bmatrix}\mu+v\cdot w+k\nu\\
v+\nu w \\ \nu\end{bmatrix}\ 
\begin{matrix}\mu,\nu\in\Z/2,v\in \W_L\\
\mu+v\cdot w+k\nu\in\Z/2\\
v+\nu w\in\W_L.
\end{matrix}
$$
or in terms of the class $D=L_{w\R}-L_\R\in H_1(X_\R;\Z/2)$
this action on $x\in H_1(X_\R;\Z/2)$ is 
$$x\mapsto x+(F_\R\cdot x)D+((D\cdot x)+k(F_\R\cdot x))F_\R\!\mod2.$$
\end{proposition}

\begin{proof} Direct application of the Viro homomorphism to Proposition \ref{Eichler-Siegel}.
\end{proof}

The restriction $\L_X \to W_L^\R$ of the Viro homomorphism $\bhv : H_2^-(X) \to H_1(X_\R;\Z/2)$ (see \cite[Sec. 2.2]{TwoKinds})
factorizes through $V_X=\L_X/2\L_X$ to an isomorphism $V_X/R_X\to W_L^\R$ where $R_X=\{v\in V_X\,|\,v \cdot V_X=0\}$.
The pullback identification of $\L_X$ with $\L=\L_Y$ induces a natural identification of $V_X, R_X$ with $V, R$ studied in Sec. \ref{mod2arithmetic}.
In particular, the function $\q_0:V\to\Z/2$ introduced there descends to $W_L^\R$ if and only if $\q_0$ vanishes on $R$. The latter happens if and only if
$X_\R$ is $\Kl\#4\T^2$, or $\Kl\#\T^2\+\SSS^2$, or $\Kl\+\Kl$. When such a descent
 exists we keep for it the same notation, $\q_0:\W_L\to\Z/2$.

\begin{proposition}\label{ESmod2}
In the above real setting, assume that $X_\R=\Kl\#p\T^2\+q\SSS^2$ with a fixed real line  $L\subset X$.
Then any other real line $L'\subset X$ has
an expression
$L'_\R=\k F_\R+v+L_\R\in H_1(X_\R;\Z/2)$, $\k\in\Z/2$, $v\in \W_L$. Conversely:
\begin{enumerate}\item
If $(p,q)$ is different from $(4,0)$ and $(1,1)$, a class $\k F_\R+v+L_\R$ is realizable by a real line for any $\k\in\Z/2$, $v\in \W_L$.
\item
If $(p,q)$ is $(4,0)$ or $(1,1)$, then class $\k F_\R+v+L_\R$ 
is realizable by a real line if and only if $\k=\q_0(v)$.
\end{enumerate}
\end{proposition}

\begin{proof}
The  coordinate expression for 
the $\Z/2$-homology classes of real lines follows from that of $\Z$-homology classes of complex lines
in Proposition \ref{Eichler-Siegel}  by applying
the Viro homomorphism, which sends 
$F, L\in H_2(X)$ to $F_\R, L_\R\in H_1(X_\R; \Z/2)$, and $\L\subset H_2(X)$ onto $\W_L\subset H_1(X_\R;\Z/2)$.

By Proposition \ref{line-root-correspondence}(3) the set of $\Z$-homology classes of real lines is
$$\{L'=L+\frac{w^2}2F+w\,|\, w\in\L\}\subset H_2(X).$$
As we apply the Viro homomorphism, this gives $L'_\R=\k F_\R+v+L_\R\in H_1(X_\R;\Z/2)$
with $v=\Upsilon (w)$ and $\k=\frac{w^2}2\rm\,mod\, 2$. The Viro homomorphism establishes an isomorphism between $V/R$ and $\W_L$ preserving 
the intersection indices $\rm{mod}\,  2$, and therefore there remains to notice that in the case of non vanishing $\q_0\vert_R$ (in which $\q_0$
does not descend to $\W_L$) we can get any $\k\in\Z/2$ independently of $v\in\W_L$ by choosing an appropriate $w\in\Upsilon^{-1}(v)$.
\end{proof}

\begin{remark}
A similar result holds for real 
del Pezzo surfaces $Y$ of degree 1: \newline
{\it If  $Y_\R$ is $\Rp2\#4\T^2$ or $\Rp2\#\T^2\+\SSS^2$, and $K_\R$ is the real canonical divisor {\rm(}dual to $w_1(Y_\R)${\rm)},
then a class $h\in H_1(Y_\R; \Z/2)$
is realized by a real line if and only if}
$$ h\in 
K_\R+\{v\in K_\R^\perp\,|\, \q_0(v)=1.\}
$$ 
This is a straightforward application of Propositions \ref{one-to-one} and \ref{lift-to-root} by means of the Viro homomorphism.
\end{remark}

\subsection{An obstruction to the realizability
of homology classes by real lines}\label{obstruction}
In Theorem \ref{real-action-matrix} to simplify the formulation we omitted a description of the range for the coefficients 
$\kk\in\Z/2,\ m_i\in\Z,\ \k_i\in\{0,1\}$ realizable by real lines $L'$ in 
coordinate expression
$$L'_\R= \kk F_\R+\sum_{i=1}^4m_ib_i+\sum_{i=1}^4\k_io_i + L_\R.$$
It can be deduced from Proposition \ref{ESmod2} (cf. also Theorem \ref{section-criterion-X})
that for $X_\R=\Kl\#4\T^2$  the coefficients
$m_i\in\Z,\ \k_i\in\{0,1\}$ can take any values, while 
$$
\kk= m_1+m_3+\sum_{i=1}^4m_i\k_i+\sum_{i=1}^4\k_i\,\,\rm{mod}\,2.
$$
Due to the same proposition, in the case of $X_\R=\Kl\#\T^2\+\SSS^2$ we have a 
relation
$$
\kk= m+m\k+\k=\begin{cases}m \,\rm{mod}\,2&\text{ if }\k=0\\
1\,\rm{mod}\,2&\text{ if }\ \k=1
\end{cases}\,
$$
whenever the 
$\kk F_\R+mb+\k o+ L_\R\in H_1(X_\R)$ is realizable by a real line.

\subsection{Application: Conics tangent to a pair of lines and a cubic}\label{decomposable-quintics}
Consider a pair $L_1, L_2\subset\P^2$ of distinct real lines and a nonsingular real cubic $A\subset\P^2$ transversal to $L_1\cup L_2$. 
Let us
enumerate the set $\Cal B$ of
real nonsingular conics $B\subset\P^2$ tangent to both $L_1, L_2$ and tritangent to $A$. 
Consider for that the double covering $\pi : Q\to \P^2$ branched along $L_1\cup L_2$ and observe that the real structure of $\P^2$ lifts to two real structures on
$Q$ that differ by composing with a deck transformation $s:Q\to Q$ of $\pi$. Furthermore, for each $B\in \Cal B$, its preimage 
$\pi^{-1}(B)$ splits into a pair of distinct conic sections, $l$ and $s(l)$, which are tritangent to 
the sextic $C=\pi^{-1}(A)$ and real with respect to one, and only one,
of the real structures. In the opposite direction we deal with an alternative. If for a tritangent $l\subset Q$, which is real with respect to one of the real structures, we have $l\ne s(l)$, then the pair $l, s(l)$ projects to a conic $B\in\Cal B$. 
If, on the contrary, $l=s(l)$ is real with respect to one real structure, then $l$ is real with respect to the other real structure too and projects to a real line
passing through one of the 6 intersection points of $A$ with $L_1\cup L_2$ and tangent to $A$ at some other point. 
This leads to a formula $|\Cal B| = \frac1{2}(|\Cal T_1| + |\Cal T_2|)- |\Cal R|$ where $\Cal T_1$ and $\Cal T_2$ denote the sets of tritangents to $C$ which are 
real with respect to the corresponding real structures on $Q$, while $\Cal R$ is the set of real lines in $\P^2$ passing through one of the 6 intersection points of $A$ with $L_1\cup L_2$ and tangent to $A$ at some other point. 
\begin{figure}[h!]
\caption{}\label{M-d-quintic}
\includegraphics[height=2cm]{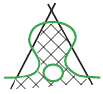}
\end{figure}
For example, in the case of configuration $L_1, L_2, A$ shown at Fig. \ref{M-d-quintic}
for one of the 2 covering real structures on $Q$ the sextic is of type  $\la4\vert0\ra$, and of type $\la 1\vert 1\ra$ for the other real structure, so that  
we obtain
$|\Cal B|=\frac1{2} (120 + 24) -24 =48$, with all conics 
from the set $\Cal B$ lying in the shaded domain
(because all $|\Cal T_2|=24$ tritangents to the sextic of type $\la1|1\ra$ must be represented by the $|\Cal R|=24$ lines).

\subsection{Five types of real theta characteristics on real sextics lying on a quadric cone}\label{5types-theta}
As is known, a nonsingular complete intersection of a quadric surface with a cubic surface in $\P^3$ is a canonically embedded curve of genus 4. Furthermore, every non-hyperelliptic genus 4 curve $C$ arises as such a complete intersection sextic. 
The corresponding quadric, $Q\supset C$ is defined uniquely by sextic $C$ and it
 is a quadratic cone if and only if $C$ has a 
 vanishing even theta-characteristic, $\theta_0$.
The latter is of dimension 2 and, thus, defines a map $\pi : C\to \P^1$ 
which can be identified
with  the central projection of $Q\to \P^1$ from the vertex 
$v\in Q$, where $\P^1$ is identified with
 the generating conic of $Q$.
 
Over the reals, $\theta_0$ and $\pi$
are real too. 
They allow us to distinguish the $J$-component of $C_\R$ from its ovals. Namely,
the restriction $\pi|_{C_\R}:C_\R\to \P^1_\R$
is of degree 1 on the $J$-component and of degree 0 on the ovals. 

On the other hand, the real tritangents to $C$ are in 1-to-1 correspondence with the real odd theta-characteristics. Together with above property of $\pi$,
we may distinguish 4 types of real odd theta-characteristics, equivalently 4 types of real tritangents, by counting the number $\tau$ of ovals on which a given
characteristic has odd number of zeros, $0\le\tau\le3$.
For $\tau\ne0$, the corresponding tritangents are of type $T_\tau$, while for $\tau=0$ we have types $T_0$ and $T_0^*$.

It would be interesting to find how to distinguish in a language of theta-character\-is\-tics
positive tritangents from negative, and elliptic ones from hyperbolic.

\subsection{Non rational elliptic surfaces}\label{nonrational}
In the case of non rational elliptic surfaces the Mordell-Weil group is no more stable under deformations in the class of elliptic surfaces.
So, none of the questions treated in this paper seems to be meaningful
beyond the rational case.
However, it looks 
 interesting to find how the {\it maximal rank} of the real Mordell-Weil group depends on the geometric genus of the elliptic surface.
For instance,
in the case of genus 1 (elliptic K3 surfaces) the maximal rank of the Mordell-Weill group is 18, both over $\C$ and over $\R$ (see \cite{Cox} for $\C$; a similar application of strong Torelli can be adapted to $\R$). It seems to be unknown whether such a coincidence holds for genus $>1$.

\subsection{10 real vanishing classes on del Pezzo surfaces}\label{puzzle10}
The set of complex vanishing cycles on a del Pezzo surface
 $Y$  is formed by the $(-2)$-roots in $K^\perp\subset H_2(Y)$.
By analogy, one could think that for a real $Y$ any $-2$-root in $\L=K^\perp\cap\ker(1+\conj_*)$
gives a real vanishing class, but
it is far from the truth. For example, if
$Y$ is a real del Pezzo surface of degree $K^2=1$ 
with $Y_\R=\Rp2\#4\T^2$,
then $\L=E_8$ has 120 pairs, $\pm e$, of roots,
but among them
only 10 pairs are real vanishing classes: the 4 pairs of oval-classes and 6 pairs of  bridge-classes depicted on the rightmost diagram in the first row of
Fig. \ref{oval-bridge-graphs}.

Mysteriously, the same number 10 appears  
for real del Pezzo surfaces 
$Y$ 
of other degrees  $2\le d=K^2\le5$,
as we count
pairs of real vanishing classes in the maximal case $Y_\R=\Rp2\#
(9-d)\Rp2$.
 On Fig. \ref{graphs} we show the intersection graph
 of these real  vanishing classes 
 for $d=1,\dots, 4$. Each vertex stands for a pair, $\pm e\in \L$, of real vanishing classes, while edges
 indicate the intersection indices $\pm 1$. For $d=1, 2$ the graphs are bipartite wherein the oval-classes and the bridge-classes are 
 represented by circle- and cross-vertices, respectively.

\begin{figure}[h!]
\caption{}\label{Graphs of real vanishing classes}\label{graphs}
\includegraphics[height=2.7cm]{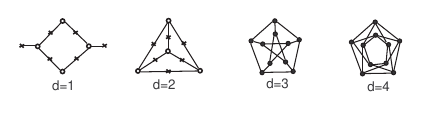}
\end{figure}


\begin{thebibliography}{1}
 \bibitem[At]{Atiyah}
 {\sc M. ~Atiyah,}
 {\em On analytic surfaces with double points.}
 Proc. Roy. Soc. London. Ser. A,
247 (1958),  237 -- 244.


\bibitem[C]{Cox}
{\sc D.~Cox,}
{\em  Mordell-Weil Groups of Elliptic Curves over $\C(t)$ with $p_g =0$ or $1$.} 
Duke Math. Journal, vol. 49, No. 3 (1982), 677 -- 689.


 \bibitem[DIK]{DIK}
{\sc A.~ Degtyarev, I.~Itenberg, V.~Kharlamov,}
{\em  Real Enriques Surfaces.}
Lecture Notes in Mathematics 1746, Springer-Verlag (2000).

 \bibitem[DK]{DK}
{\sc A.I.~Degtyarev, V.M.~Kharlamov,}
{\em  Topological properties of real algebraic varieties: Rokhlin's
way.}
Russian Math. Surveys,  55:4 (2000), 735--814.


\bibitem[FM]{farb}
{\sc B.~Farb, D.~Margalit,}
{\em  A Primer on Mapping Class Groups.}
Princeton University Press, Princeton (2012).

\bibitem[FK-1]{TwoKinds}
{\sc S.~Finashin, V.~Kharlamov,}
\textit{Two kinds of real lines  on real del Pezzo surfaces of degree 1.}
 Selecta Math. (N.S.)
 27 (2021), no. 5, 
 Paper No. 83, 23 pp.

\bibitem[FK-2]{Combined}
{\sc S.~Finashin, V.~Kharlamov,}
\textit{Combined count of real rational curves of canonical degree 2 on real del Pezzo surfaces with $K^2=1$.}
J. Inst. Math. Jussieu 23 (2024), no. 1, 123–148.

\bibitem[FK-3]{Surgery}
{\sc S.~Finashin, V.~Kharlamov,}
\textit{On wall-crossing invariance of certain sums of Welschinger numbers.}
Selecta Math. (N.S.) 29 (2023), no. 3, Paper No. 41, 31 pp.



\bibitem[K]{Kodaira}
{\sc K.~Kodaira, }
{\em On Stability of Compact Submanifolds of Complex Manifolds.}
 American Journal of Mathematics 85 (1963), no. 1, 79 -- 94.



\bibitem[IKS]{IKS}
{\sc I.~Itenberg, E.~Shustin,V.~Kharlamov,}
\textit{  Welschinger
invariants of real del Pezzo surfaces of degree $\ge2$.}
 Internat. J. Math.
26 (2015), no. 8., 1550060, 63 pp.

 \bibitem[Ro]{Ro}
 {\sc V.A.~Rokhlin,}
 Complex orientation of real algebraic curves. (Russian)
Funkcional. Anal. i Prilozen. 8(1974), no.4, 71–75.
English translation: Functional Analysis and Its Applications, 1974, Volume 8, Issue 4, Pages 331–334.

\bibitem[Ru]{Russo}
{\sc F.~Russo,}
{\em  The antibirational involutions of the plane and the classification of real del Pezzo surfaces.}
in Algebraic geometry de Gruyter, Berlin, 2002, 289 -- 312.

\bibitem[SS]{shioda}
{\sc M.~Schütt, T.~Shioda}
{\em Mordell–Weil Lattices.}
Ergebnisse der Mathematik und ihrer Grenzgebiete. 3. Folge, 70, xvi+431 pp., Springer-Verlag (2019).

\bibitem[S]{stukow}
{\sc M.~Stukow}
{\em Commensurability of geometric subgroups of mapping class groups.}
Geom. Dedicata 143 (2009), 117-142.




 \bibitem[Z]{Zvo}
 {\sc V. I.~Zvonilov,}
{\em Complex orientations of real algebraic curves with singularities.(Russian)}
Dokl. Akad. Nauk SSSR, 268(1983), no.1, 22–26.
English translation: Soviet Math. Dokl. 27 (1983), no. 1–6, 14–17.

\end{thebibliography}
\end{document}